%% file: main.tex
\documentclass{article}

\usepackage{arxiv}
\usepackage[utf8]{inputenc} 
\usepackage[T1]{fontenc}    
\usepackage{hyperref}       
\hypersetup{colorlinks, citecolor=red, linkcolor=blue, urlcolor=blue}
\usepackage{amsfonts}       
\usepackage{nicefrac}       
\usepackage{microtype}      
\usepackage{graphicx}
\usepackage[square,numbers]{natbib}

\usepackage{enumitem}
\usepackage{amsmath,amssymb,amsthm}
\usepackage{bbm}
\usepackage{xfrac}
\usepackage[caption=false]{subfig}

\newtheorem{theorem}{Theorem}[section]
\newtheorem{proposition}[theorem]{Proposition}
\newtheorem{lemma}[theorem]{Lemma}
\newtheorem{corollary}[theorem]{Corollary}
\theoremstyle{definition}
\newtheorem{definition}[theorem]{Definition}

\theoremstyle{remark}
\newtheorem{remark}{Remark}

\DeclareMathOperator*{\dom}{\textbf{dom}}
\DeclareMathOperator*{\argmin}{argmin}
\newcommand*\diff{\mathop{}\!\mathrm{d}}
\def\0{\boldsymbol{0}}
\def\M{\mathcal{M}}
\def\Mup{\mathcal{M}^{\uparrow}}
\def\Mdo{\mathcal{M}_{\downarrow}}
\def\R{\mathbb{R}}
\def\A{\boldsymbol{A}}
\def\B{\mathcal{B}}
\def\C{\mathcal{C}}
\def\F{\mathcal{F}}
\def\H{\mathcal{H}}
\def\Q{\mathcal{Q}}
\def\S{\mathcal{S}}
\def\T{\mathcal{T}}
\def\b{\boldsymbol{b}}
\def\g{\boldsymbol{g}}
\def\p{\boldsymbol{p}}
\def\u{\boldsymbol{u}}
\def\v{\boldsymbol{v}}
\def\x{\boldsymbol{x}}
\def\y{\boldsymbol{y}}
\def\z{\boldsymbol{z}}
\def\xprj{\boldsymbol{x}^*_{\text{prj}}}
\def\tilphi{\tilde{\phi}}

\title{The Minimizer of the Sum of Two Strongly Convex Functions}
\date{} 

\author{ \href{https://orcid.org/0000-0002-0769-1399}{\includegraphics[scale=0.06]{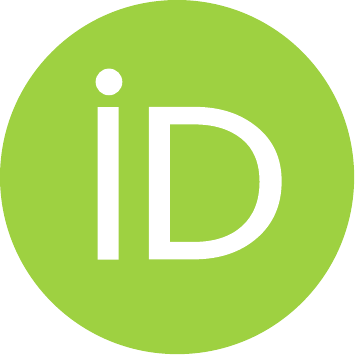}\hspace{1mm}Kananart~Kuwaranancharoen} \\
	Intel Labs\\
	Intel Corporation\\
	Hillsboro, OR 97124 \\
        \href{mailto:kananart.kuwaranancharoen@intel.com}{\texttt{kananart.kuwaranancharoen@intel.com}} \\
        \And
	\href{https://orcid.org/0000-0002-5390-2505}{\includegraphics[scale=0.06]{orcid.pdf}\hspace{1mm}Shreyas~Sundaram} \\
	School of Electrical and Computer Engineering\\
	Purdue University\\
	West Lafayette, IN 47907 \\
	\href{mailto:sundara2@purdue.edu}{\texttt{sundara2@purdue.edu}} \\
}

\renewcommand{\shorttitle}{The Minimizer of the Sum of Two Strongly Convex Functions}

\hypersetup{
pdftitle={The Minimizer of the Sum of Two Strongly Convex Functions},
pdfauthor={Kananart~Kuwaranancharoen, Shreyas~Sundaram},
pdfkeywords={Convex Optimization, Distributed Optimization, Minimizer Analysis, Quadratic Functions, Strongly Convex Functions},
}

\begin{document}
\maketitle

\begin{abstract}
  The optimization problem concerning the determination of the minimizer for the sum of convex functions holds significant importance in the realm of distributed and decentralized optimization. In scenarios where full knowledge of the functions is not available, limiting information to individual minimizers and convexity parameters -- either due to privacy concerns or the nature of solution analysis -- necessitates an exploration of the region encompassing potential minimizers based solely on these known quantities. The characterization of this region becomes notably intricate when dealing with multivariate strongly convex functions compared to the univariate case. This paper contributes outer and inner approximations for the region harboring the minimizer of the sum of two strongly convex functions, given a constraint on the norm of the gradient at the minimizer of the sum. Notably, we explicitly delineate the boundaries and interiors of both the outer and inner approximations. Intriguingly, the boundaries as well as the interiors turn out to be identical. Furthermore, we establish that the boundary of the region containing potential minimizers aligns with that of the outer and inner approximations.
\end{abstract}


\input{contents/sec-intro}
\input{contents/sec-preliminaries}

\input{contents/sec-problem}
\input{contents/sec-outer}
\input{contents/sec-inner}
\input{contents/sec-solution}
\input{contents/sec-discussion}
\input{contents/sec-conclusion}

\bibliographystyle{unsrtnat}
\bibliography{main}  

\newpage
\appendix
\setcounter{equation}{0}
\renewcommand{\theequation}{SM.\arabic{equation}}

\input{contents/supp-outer}
\input{contents/supp-inner}
\input{contents/supp-discussion}

\end{document}

%% file: contents/sec-intro.tex
\section{Introduction}

Optimization of a sum of functions is a ubiquitous challenge with applications spanning machine learning \cite{shalev2012online, boyd2011distributed, sayed2014adaptive}, control of large-scale systems \cite{molzahn2017survey, li2011optimal}, and cooperative robotic systems \cite{zavlanos2012network, tron2016distributed}. In these scenarios, each node in a network is assumed to possess a local objective function. Following the framework described in \cite{nedic2018network}, the problem is commonly explored within two main architectural paradigms: the \emph{distributed} setting, marked by a client-server architecture \cite{zhang2012communication, shamir2014communication, konevcny2016federated}, and the \emph{decentralized} setting, characterized by a peer-to-peer architecture \cite{nedic2009distributed, zhu2011distributed, nedic2014distributed}.

\subsection{Byzantine-resilient decentralized optimization}
\label{subsec: byzantine decentralized}

In decentralized optimization problems, nodes in the network are connected to neighbors through communication links, allowing them to exchange information such as estimates of the solution at the current time-step (referred to as state) \cite{nedic2018Adistributed, nedic2018Bdistributed, yang2019survey}. Unlike the distributed setting, there is no centralized server to aggregate and distribute the updated states in this scenario. However, many algorithms proposed to address decentralized optimization problems, such as \cite{nedic2009distributed, duchi2011dual, nedic2017achieving}, assume the reliability of all nodes in the network -- they follow the designed algorithm. Consequently, these algorithms are vulnerable and prone to failure in the presence of a malicious node in the network \cite{su2016fault, sundaram2018distributed, liu2021survey}. Addressing this security concern, a branch of study focuses on Byzantine-resilient decentralized optimization, where malicious nodes can behave arbitrarily, including sending incorrect information to neighboring nodes, updating their states as they wish, and possessing complete knowledge of the network, including the algorithm deployed, states of all nodes, and communication topology. The primary goal in this context is to solve the decentralized optimization objective of minimizing the sum of cost functions associated with regular (non-Byzantine) nodes.

The Byzantine-resilient decentralized optimization community has primarily focused on proposing new resilient algorithms under various assumptions, encompassing network topology, local function properties, and problem dimensionality. However, a fundamental limitation persists -- the true solution of the sum of regular nodes in the network cannot be determined, irrespective of the deployed resilient algorithms \cite{su2016multi, sundaram2018distributed}, without employing very strong assumptions on the local functions \cite{yang2019byrdie, gupta2020fault, gupta2021byzantine}.

In general cases, the proposed algorithms can guide the states to converge towards a region, the size of which depends on the specific resilient algorithm considered \cite{kuwaran2020byzantine, su2020byzantine}. Previous works have primarily provided upper bounds on the size of such regions for each algorithm. It is natural to inquire about the lower bound of this region, aiming to understand the best achievable outcome, regardless of the chosen algorithm and network topology (as discussed in \cite{kuwaran2023scalable}). This inquiry is based on mild assumptions about local functions, such as strong convexity \cite{fang2022bridge, kuwaran2023geometric}. Understanding how the best convergence region scales with fundamental parameters that specify the class of functions under consideration is a key aspect of this research endeavor. Importantly, this serves as a fundamental limit for Byzantine-resilient decentralized optimization problems, providing more insightful information than the general impossibility theorem \cite{su2016multi, sundaram2018distributed}.

Traditionally, studies in decentralized optimization have often defined the convergence region based on the diameter of the set of local minimizers for regular nodes \cite{sundaram2018distributed, kuwaran2023scalable, kuwaran2023geometric, kuwaran2023analyses}. It seems natural, therefore, to establish the lower bound of this region by considering the diameter and essential parameters that characterize the function class. To elaborate, with a minimizer identified for each node (or given the diameter of the set of local minimizers), each node has the flexibility to choose any function from a specified class (e.g., the strongly convex class). The key inquiry is the maximum potential distance from the minimizer of the sum to the center of set of local minimizers. 

Our work takes a pioneering step in addressing this fundamental question within Byzantine-resilient decentralized optimization, with a specific focus on fully characterizing the potential solution region for the sum of two strongly convex functions. Intriguingly, our initial discoveries unveil a non-convex potential solution set, even in cases where the participating functions are strongly convex (as shown in Fig.~\ref{fig: boundary M}). We anticipate that the insights from this exploration can be harnessed to tackle more intricate challenges, such as the sum of multiple functions, with the goal of discerning the optimal dependency of the convergence region size on fundamental parameters.

\subsection{Privacy-preserving federated learning}
\label{subsec: privacy federated}

In the context of a distributed optimization problem, envision a scenario where nodes operate similar machine learning models, trained on comparable data types, with a shared objective of achieving an improved machine learning model. Due to privacy concerns \cite{shokri2015privacy}, the dataset used for training each local model may not be shared. In this setup, the centralized server aggregates the current parameters of the models or the gradients computed with respect to local loss functions and mini-batches of data. The server then updates the model or combines the received gradients before distributing the updated information back to the nodes.

Consider, for instance, a federated learning setup \cite{konevcny2016federated, kairouz2021advances}, which is a distributed setting where nodes submit their current parameters to a centralized computing server, receiving an updated parameter based on the aggregated parameters. Typically, each node is required to periodically submit its current model to the server for updates \cite{reisizadeh2020fedpaq}. Recent research \cite{zhu2019deep, geiping2020inverting, wei2020framework} has highlighted privacy concerns in federated learning, particularly vulnerability to gradient leakage attacks. Specifically, if an adversary intercepts the local gradient update of a node before the server performs the federated aggregation to generate the global parameter update for the next round, the adversary can steal the sensitive local training data of this node. To enhance privacy, one approach is to minimize communication with the server \cite{guha2019one, zhang2022dense}, allowing nodes to train local models until achieving optimal performance before sending corresponding parameters.

In the specific case we analyze in this work, we assume the loss functions are strongly convex. Despite this simplification, it aligns effectively with traditional machine learning models, including linear regression \cite{weisberg2005applied, montgomery2021introduction}, logistic regression \cite{menard2002applied, hosmer2013applied}, and support vector machines (SVM) \cite{hearst1998support, steinwart2008support, mountrakis2011support}. From the server's perspective, it solely observes the minimizers submitted by participating nodes and may possess knowledge of the general model class. In this scenario, the server's objective is to select a suitable parameter based on the available information. As detailed in Section~\ref{subsec: one-shot learning}, our results offer a set of legitimate solutions that can be employed as selected parameters, particularly in the straightforward case involving only two nodes. Additionally, we propose a compelling candidate for use as the selected parameter, drawing from insights gained through the analysis in this work.

\subsection{Inherent challenges}

Returning to the general setting of our problem, a unique challenge arises when determining the potential solution region in cases of multivariate functions. In the scenario where local functions $f_i$ at each node $v_i$ are univariate (i.e., $f_i: \R \rightarrow \R$) and strongly convex, it is straightforward to argue that the minimizer of the sum must lie in the interval bracketed by the smallest and largest minimizers of the functions \cite{sundaram2018distributed}. This is because the gradients of all the functions will have the same sign outside that region, preventing them from summing to zero. However, a similar characterization of the region containing the minimizer of multivariate functions is lacking in the literature and is significantly more challenging to obtain.

For instance, the conjecture that the minimizer of a sum of convex functions lies in the convex hull of their local minimizers can be easily disproved with simple examples. Consider $f_1(x, y) = x^2 - xy + \frac{1}{2} y^2$ and $f_2(x, y) = x^2 + xy + \frac{1}{2} y^2 - 4x - 2y$ with minimizers $(0, 0)$ and $(2, 0)$, respectively. The sum has a minimizer at $(1, 1)$ which is obviously outside the convex hull of $\{ (0, 0), (2, 0) \}$. In our recent work \cite{kuwaran2018location, kuwaran2020set}, we delved into this problem and provided an \emph{outer approximation} of the region containing the minimizer of two strongly convex functions in a specific case. This region is determined by the minimizers of the individual functions, their strong convexity parameters, and the specified bound on the norms of the gradients of the functions at the location of the minimizer.

\subsection{Objective}

In this study, \textbf{our objective is to delineate the potential solution region (i.e., the set of all minimizers) for the sum of two unknown strongly convex functions.} To be more specific, we present an \emph{outer approximation} (i.e., a region containing all valid minimizers) as well as an \emph{inner approximation} (i.e., a region where every point is a valid minimizer) for this problem. The outer approximation characterized in this work is more general than the one provided in \cite{kuwaran2018location}. As we will observe, the inner approximation essentially almost coincides with the outer approximation. To be precise, the boundary of both outer and inner approximations is the same under the assumption that the gradients of the two original functions are bounded by a finite number at the potential minimizer of the sum. Thus, our analysis in this paper complements and completes the analysis in \cite{kuwaran2018location} by fully characterizing the region where the minimizer of the sum of two strongly convex functions can lie. 

It is imperative to highlight that while our question is motivated by the specific scenario of the sum of two strongly convex functions, our results extend to hold under a more general class of functions. For instance, functions satisfying the restricted secant inequality \cite{zhang2013gradient, zhang2015restricted} or restricted strongly convex functions \cite{lai2013augmented, zhang2017restricted} also fall within the scope of our analysis. Moreover, our approach potentially applies to practical scenarios involving (deep) supervised learning. In Section~\ref{subsec: applicability}, we will formally articulate the connections to these general classes of functions and discuss relevant applications. Despite the broad applicability, we acknowledge that our work is currently limited to the context of the sum of two functions due to the intricacies of the analysis, as demonstrated throughout this work. However, we view this work as a pioneering effort in addressing such problems, and in Section~\ref{subsec: multiple functions}, we will briefly explore the idea of extending our analysis to cases involving multiple functions.

\subsection{Paper structure and organization}

The paper is structured as follows. Section~\ref{sec: preliminaries} introduces the notations used throughout the paper and provides preliminaries. Section~\ref{sec: problem} outlines the problem formulation. Our main results regarding the outer approximation are presented in Section~\ref{sec: outer}, while the inner approximation is discussed in Section~\ref{sec: inner}. Section~\ref{sec: solution region} presents the characterization of the potential solution region. The discussion is covered in Section~\ref{sec: discussion}, and the conclusion is summarized in Section~\ref{sec: conclusion}.

%% file: contents/sec-preliminaries.tex
\section{Preliminaries} \label{sec: preliminaries}

\subsection{Sets}
Let $\mathbb R$ denotes the set of real numbers. We denote by $\R^n$ the $n$-dimensional Euclidean space. For a subset $\mathcal{E}$ of a topological space, we denote the complement, closure and interior of a set $\mathcal{E}$ by $\mathcal{E}^{c}$, $\overline{\mathcal{E}}$ and ${\mathcal{E}}^\circ$, respectively. The boundary of $\mathcal{E}$ is defined as $\partial \mathcal{E} = \overline{\mathcal{E}} \setminus {\mathcal{E}}^\circ$. We also use $\textbf{dom}(f)$ to denote the domain of function $f$. In addition, we use $\sqcup$ to denote the disjoint union operation. We will use this simple lemma later in the paper.

\begin{lemma}
Let $\mathcal{G}$ and $\H$ be subsets of a topological space $X$ such that $\mathcal{G} \subseteq \H$. Let $\mathfrak{P}$ be a partition of $\H$. Then, 
\begin{equation*}
    \mathcal{G}^{\circ} = \bigsqcup_{\mathcal{Z} \in \mathfrak{P}} 
    \big( (\mathcal{G} \cap \mathcal{Z}) \setminus (\partial \mathcal{G} \cap \mathcal{Z}) \big).
\end{equation*}
\label{lem: interior eqn}
\end{lemma}

\begin{proof}
For $\mathcal{Z} \in \mathfrak{P}$, since $\mathcal{G} \cap \mathcal{Z} \cap \mathcal{Z}^c = \emptyset$, we have 
\begin{equation*}
    \mathcal{G}^{\circ} \cap \mathcal{Z}
    = (\mathcal{G} \cap (\partial \mathcal{G})^c \cap \mathcal{Z}) \cup (\mathcal{G} \cap \mathcal{Z} \cap \mathcal{Z}^c)
    = (\mathcal{G} \cap \mathcal{Z}) \cap (\partial \mathcal{G} \cap \mathcal{Z})^c
    = (\mathcal{G} \cap \mathcal{Z}) \setminus (\partial \mathcal{G} \cap \mathcal{Z}).
\end{equation*}
Using the above equation, we can write
\begin{align}
    \mathcal{G}^{\circ}
    = \bigsqcup_{\mathcal{Z} \in \mathfrak{P}} ( \mathcal{G}^{\circ} \cap \mathcal{Z} )
    = \bigsqcup_{\mathcal{Z} \in \mathfrak{P}} \big( (\mathcal{G} \cap \mathcal{Z}) \setminus (\partial \mathcal{G} \cap \mathcal{Z}) \big). \nonumber 
\end{align}
\end{proof}

\subsection{Linear algebra}
For simplicity, we often use $(x_1, \ldots, x_n)$ to represent the column vector 
$\x = \begin{bmatrix} x_1 & x_2 & \cdots & x_n \end{bmatrix}^\intercal$. We use $\boldsymbol{0}$ to denote the all-zero vector with appropriate dimension and $\boldsymbol{e}_i$ to denote the $i$-th basis vector (the vector of all zeros except for a one in the $i$-th position). We denote by $\langle \u, \v \rangle$ the Euclidean inner product of vectors $\u$ and $\v$, i.e., $\langle \u, \v \rangle \triangleq \u^\intercal \v$, 
by $\Vert \u \Vert$ the Euclidean norm of $\u$, i.e., $\| \u \| \triangleq \sqrt{ \langle \u, \u \rangle} = (\sum_i u_i^2)^{1/2}$.
We define the functions $\angle : (\R^n \setminus \{ \mathbf{0} \}) \times (\R^n \setminus \{ \mathbf{0} \}) \to [0, \pi]$ and $\measuredangle : (\R^2 \setminus \{ \mathbf{0} \}) \times (\R^2 \setminus \{ \mathbf{0} \}) \to \big[ -\frac{\pi}{2}, \frac{\pi}{2} \big]$ as 
\begin{equation}
    \angle (\u, \v) 
    \triangleq \arccos \bigg( \frac{ \langle \u, \v \rangle}{\Vert \u \Vert \Vert \v \Vert}  \bigg)
    \quad \text{and} \quad 
    \measuredangle (\u, \v) 
    \triangleq \arcsin \bigg( \frac{u_2 v_1 - u_1 v_2}{\| \u \|  \| \v \|} \bigg),
    \label{def: angle func}
\end{equation}
respectively. Note that $\angle (\u, \v) = \angle (\v, \u)$ but $\measuredangle (\u, \v) = - \measuredangle (\v, \u)$.
We use 
\begin{equation}
    \B (\x_0, r_0) \triangleq \{ \x \in \R^n: \Vert \x - \x_0 \Vert < r_0  \}
    \label{def: ball B}
\end{equation}
and $\overline{\B} (\x_0, r_0)$ to denote the open and closed balls, respectively, in $\R^n$ centered at $\x_0 \in \R^n$ and with radius $r_0 \in \R_{>0}$. We use $\boldsymbol{I}$ to denote the identity matrix of appropriate dimension. For square matrix $\A \in \R^{n \times n}$, we use $\lambda (\A)$, $\lambda_{\text{min}} (\A)$ and $\text{Tr}(\A)$ to denote an eigenvalue, the minimum eigenvalue and the trace of matrix $\A$, respectively. For $\A \in \R^{m \times n}$, we use $\mathcal{R}(\A)$ and $\mathcal{N}(\A)$ to denote the column space and null space of matrix $\A$, respectively.

\subsection{Convex sets and functions}  \label{subsec: convex}
A set $\C \subseteq \R^n$ is said to be convex if, for all $\x$ and $\y$ in $\C$ and all $t$ in the interval $(0, 1)$, the point $(1- t) \x + t \y$ also belongs to $\C$. A differentiable function $f$ is called strongly convex with parameter $\sigma \in \R_{>0}$ (or $\sigma$-strongly convex) if 
\begin{equation}
    \langle \nabla f(\x) - \nabla f(\y), \; \boldsymbol{x } - \y \rangle \geq \sigma \Vert \x - \y \Vert^2 
    \label{def: strongly cvx}
\end{equation}
holds for all points $\x, \y$ in its domain. We use $\S(\x^*, \sigma)$ to denote the set of all differentiable and $\sigma$-strongly convex functions that have their minimizer at $\x^* \in \R^n$. Define $\mathsf{S}^n$ to be the set of symmetric matrices in $\R^{n \times n}$, and $\mathsf{Q}^{n}$ to be the set of all quadratic functions that map $\R^n$ to $\R$.
A quadratic function $f \in \mathsf{Q}^{n}$ parameterized by $\boldsymbol{Q} \in \mathsf{S}^n$, $\b \in \R^n$, and $c \in \R$ is given by
\begin{displaymath}
    f( \x ; \boldsymbol{Q}, \b, c) = \frac{1}{2} \x^\intercal \boldsymbol{Q} \x + \b^\intercal \x + c.
\end{displaymath}
For $\x^* \in \R^n$ and $\sigma \in \R_{> 0}$, define
\begin{equation}
    \Q^{(n)}(\x^*, \sigma) 
    \triangleq \big\{ f (\x ; \boldsymbol{Q}, \b, c) \in \mathsf{Q}^{n} : 
    \lambda_{\text{min}}( \boldsymbol{Q} ) = \sigma, \;\; \boldsymbol{Q} \x^* = - \b \big\}.
    \label{def: Quadratic Coll}
\end{equation}
We will omit the superscript $(n)$ of $\Q^{(n)}$ when it is clear from contexts.
Note that every function in $\Q(\x^*, \sigma)$ is $\sigma$-strongly convex quadratic and has the minimizer at $\x^*$, and 
\begin{equation}
    \Q(\x^*, \sigma) \subset \bigcup_{\Tilde{\sigma} \geq \sigma} \Q(\x^*, \Tilde{\sigma}) \subset \S(\x^*, \sigma).
\label{eqn: set of functions inclusion}
\end{equation}

The following lemma shows that the strong convexity of functions is invariant under some particular affine transformations. This property will help us to simplify the analysis throughout the paper.
\begin{lemma}
Let $\A \in \R^{n \times n}$ be an orthogonal matrix and $\b \in \R^n$. Suppose $f: \R^n \to \R$ is a differentiable function and define $h (\x) = f ( \A \x + \b)$. Then, $f$ is $\sigma$-strongly convex if and only if $h$ is $\sigma$-strongly convex.
\label{lem: strong cvx}
\end{lemma}

\begin{proof}
By the definition of strongly convex functions in \eqref{def: strongly cvx}, we have that
\begin{equation*}
    \langle \nabla f ( \x ) - \nabla f ( \y ), \; \x - \y \rangle 
    \geq \sigma \| \x - \y \|^2
    \quad \text{for all} \quad \x, \y \in \R^n.
\end{equation*}
Since $\A$ is invertible, we can replace $\x$ and $\y$ by $\A \x + \b$ and $\A \y + \b$, respectively, and the above inequality is equivalent to
\begin{equation}
    \big\langle \nabla f ( \A \x + \b ) - \nabla f ( \A \y + \b ), \; \A (\x - \y) \big\rangle 
    \geq \sigma \| \A (\x - \y) \|^2
    \;\; \text{for all} \;\; \x, \y \in \R^n.
    \label{eqn: strongly cvx ineq}
\end{equation}
Since $\nabla h(\x) = \A^\intercal \nabla f ( \A \x + \b )$, we can rewrite the LHS of \eqref{eqn: strongly cvx ineq} as $\big\langle \A^\intercal \big( \nabla f ( \A \x + \b ) - \nabla f ( \A \y + \b ) \big), \; \x - \y \big\rangle = \langle \nabla h ( \x ) - \nabla h ( \y ), \; \x - \y \rangle$.
On the other hand, since $\A$ is an orthogonal matrix, the RHS of \eqref{eqn: strongly cvx ineq} becomes $\sigma \| \x - \y \|^2$.
\end{proof}

%% file: contents/sec-problem.tex
\section{Problem formulation} \label{sec: problem}
Consider two (unknown) functions $f_1$ and $f_2$. In order to investigate the minimizer of the sum of two unknown functions $f_1 + f_2$, we will impose the following assumptions on the structure of both functions.
\begin{enumerate}
    \item Given $\sigma_1, \sigma_2 \in \R_{>0}$, the functions $f_1: \R^n \rightarrow \R$ and $f_2: \R^n \rightarrow \R$ are differentiable and strongly convex with parameters $\sigma_1$ and $\sigma_2$, respectively. \label{asm: function}
    \item \label{asm: minimizer} Given $\x_1^*, \x_2^* \in \R^n$, the minimizers of $f_1$ and $f_2$ are at $\x_1^*$ and $\x_2^*$, respectively.
    \item Suppose $\x^* \in \R^n$ is the minimizer of $f_1 + f_2$. There is a finite (given) number $L \in \R_{>0}$ such that the norm of gradient of $f_1$ and $f_2$ evaluated at $\x^*$ is no larger than $L$. \label{asm: grad norm}
\end{enumerate}
Assumption \ref{asm: function} and \ref{asm: minimizer} will be captured using the notations introduced earlier: $f_1 \in \S (\x_1^*, \sigma_1)$ and $f_2 \in \S (\x_2^*, \sigma_2)$. 
For Assumption \ref{asm: grad norm}, since $\x^*$ is the minimizer of $f_1 + f_2$, we have that $\nabla f_1(\x^*) = - \nabla f_2(\x^*)$. In addition, we can rewrite the bounded gradient at $\x^*$ condition as $\| \nabla f_1(\x^*) \| =$ $\| \nabla f_2(\x^*) \| \leq L$. Essentially, our goal is to estimate the region $\M$ containing all possible values $\x^*$ satisfying the above conditions. More specifically, given $\x_1^*, \x_2^*  \in \R^n$, $\sigma_1, \sigma_2 \in \R_{>0}$, and $L \in \R_{>0}$, we wish to estimate the potential solution region 
\begin{multline}
    \M(\x_1^*, \x_2^*, \sigma_1, \sigma_2, L) \triangleq \big\{ \x \in \R^n : \exists f_1 \in \S( \x_1^*, \sigma_1 ), \quad \exists f_2 \in \S( \x_2^*, \sigma_2 ), \\
    \nabla f_1(\x) = -\nabla f_2(\x), \quad \Vert \nabla f_1(\x) \Vert = \Vert \nabla f_2(\x) \Vert \leq L \big\}. \label{def: set M}
\end{multline}
For simplicity of notation, we will omit the argument of the set $\M(\x_1^*, \x_2^*, \sigma_1, \sigma_2, L)$ and write it as $\M$.

It is crucial to emphasize the nature of the solution region $\M \subset \R^n$ defined in \eqref{def: set M}. $\M$ comprises points where each $\x \in \M$ corresponds to at least one pair of functions $(f_1, f_2)$ satisfying the following three properties: (1) $f_i \in \S( \x_i^*, \sigma_i )$ for $i \in \{ 1, 2 \}$, (2) $\nabla f_1(\x) = - \nabla f_2(\x)$, and (3) $\| \nabla f_1(\x) \| = \| \nabla f_2(\x) \| \leq L$. In simpler terms, a point within $\M$ serves as the minimizer of the sum of two strongly convex functions chosen from specific classes. However, as the definition of $\M$ guarantees the existence of such pairs, it is conceivable that multiple pairs of functions correspond to a single point $\x$ within the potential solution region $\M$. With this consideration, we will formally investigate this question in Section~\ref{subsec: correspondence}.

\subsection{Discussion of assumptions}
Functions that satisfy both differentiable and strongly convex conditions (Assumption \ref{asm: function}) are common in many applications. In machine learning applications, for example, linear regression and logistic regression models with $L_2$-regularization are commonly used when only a small amount of training data is available \cite{hastie2009elements, ng2004feature}.

Assumption~\ref{asm: minimizer} can be generalized by assuming that for $i \in \{ 1, 2 \}$, the minimizer $\x_i^*$ of the function $f_i$ is not available but instead $\x_i^*$ is located in a known compact set $\mathcal{A}_i \subset \R^n$ as in \cite{kuwaran2020set}. However, the analysis will be more involved, so we defer these assumptions to our future works.

Assumption~\ref{asm: grad norm} is a technical assumption. Given $\x_1^*, \x_2^* \in \R^n$ such that $\x_1^* \neq \x_2^*$, let
\begin{equation*}
    \mathcal{L} = \big\{ \x \in \R^n :  \text{there exists} \; k \in \R \setminus (-1, 1) \;\; \text{such that} \;\; \x - \x_1^* = k (\x_2^* - \x_1^*)   \big\}. 
\end{equation*}
Without Assumption~\ref{asm: grad norm} -- meaning the norm of the gradient of each function at the minimizer of the sum can be arbitrarily large -- one can utilize the result from Proposition~\ref{prop: function exist n-dim} to demonstrate that $\M = \R^n \setminus \mathcal{L}$. It is noteworthy that for $n \in \mathbb{N} \setminus \{ 1 \}$, the set $\mathcal{L}$ has measure zero, implying that $\M$ covers almost the entire space. In simpler terms, almost all points have the potential to be minimizers. A detailed exploration of this perspective is provided in Section~\ref{subsec: assumption3}. While imposing bounds on gradients can be viewed as a means of implicitly limiting the functions within $\S (\x_1^*, \sigma_1)$ and $\S (\x_2^*, \sigma_2)$, alternative methods may exist to constrain the class of functions, as we will discuss in some intriguing alternatives in Section~\ref{subsec: alternative assumptions}. However, for now, we confine ourselves to the simpler Assumption~\ref{asm: grad norm}, leaving exploration of such alternative assumptions for future work.

\subsection{A preview of the solution}
Recall the definition of the potential solution region $\M$ from \eqref{def: set M}. One way to characterize the set $\M$ is to provide an explicit formula for the boundary $\partial \M$ in terms of $\x_1^*$, $\x_2^*$, $\sigma_1$, $\sigma_2$ and $L$. In Fig.~\ref{fig: boundary M}, we provide a preview of the boundary $\partial \M$ in $\R^2$ given fixed parameters $\sigma_1 = 1.5$, $\sigma_2 = 1$, and $L = 10$, and a variable parameter $r \in \R_{>0}$. Suppose $\x_1^* = (-r, 0)$ and $\x_2^* = (r, 0)$. We illustrate $\partial \M$ for the case where $r = 2, 4$, and $6$ in Fig.~\ref{fig: boundary M Case1}, Fig.~\ref{fig: boundary M Case2} and Fig.~\ref{fig: boundary M Case3}, respectively. The different colors in the figures indicate different equations that combine together to yield the boundary (as we will explicitly characterize in the rest of the paper).

\begin{figure}
\centering
\subfloat[For $r=2$, $\partial \M$ consists of $1$ curve (blue curve) and $\{ \x_1^*, \x_2^* \}$.]{\includegraphics[width=.30\textwidth]{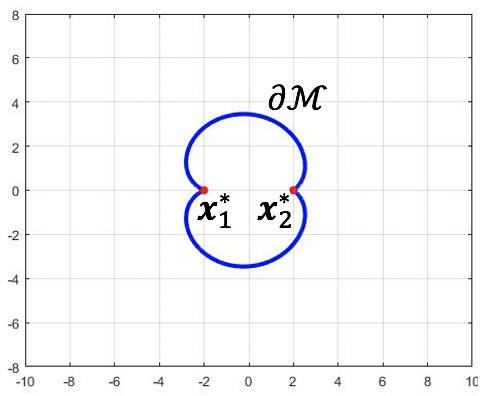} \label{fig: boundary M Case1}}\;\;
\subfloat[For $r=4$, $\partial \M$ consists of $2$ curves (blue and cyan curves) and $\{ \x_1^* \}$.]{\includegraphics[width=.30\textwidth]{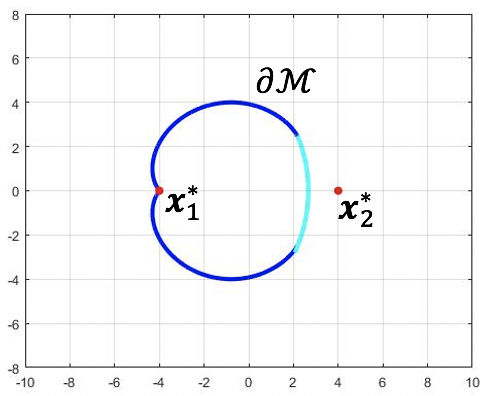} \label{fig: boundary M Case2}}\;\;
\subfloat[For $r=6$, $\partial \M$ consists of $3$ curves (blue, cyan and magenta curves).]{\includegraphics[width=.30\textwidth]{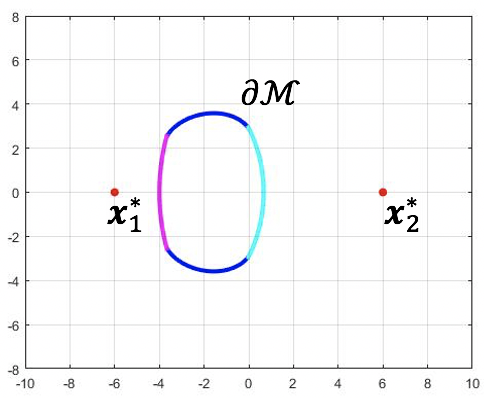} \label{fig: boundary M Case3}}
\caption{The boundary $\partial \M$ in $\R^2$ is plotted given minimizers $\x_1^* = (-r,0)$ and $\x_2^* = (r,0)$ and fixed parameters $\sigma_1 = 1.5$, $\sigma_2 = 1$, and $L = 10$. Different colors denote different equations that combine together to yield the boundary $\partial \M$.}
\label{fig: boundary M}
\end{figure}

\subsection{Solution approach}  
\label{subsec: solution approach}

Since the analysis of the case $\x_1^* = \x_2^*$ is trivial (i.e., the potential solution region is $\M = \{ \x_1^* \}$), without loss of generality, we assume that $\x_1^* = (-r, 0, \ldots, 0) \in \R^n$ and $\x_2^* = (r, 0, \ldots, 0) \in \R^n$ with $r = \frac{1}{2} \| \x_2^* - \x_1^* \| > 0$. 

To show this, given general $\x_1^*, \x_2^* \in \R^n$ with $\x_1^* \neq \x_2^*$, let the set of new bases $\mathcal{J} = \{ \boldsymbol{e}'_1, \boldsymbol{e}'_2, \ldots, \boldsymbol{e}'_n \}$ be such that $\boldsymbol{e}'_1 = \frac{ \x_2^* - \x_1^* }{ \| \x_2^* - \x_1^* \|}$ and $\{ \boldsymbol{e}'_2, \boldsymbol{e}'_3, \ldots, \boldsymbol{e}'_n \}$ is obtained by Gram-Schmidt orthogonalization. Let 
\begin{equation*}
    \boldsymbol{E} = \begin{bmatrix} \boldsymbol{e}_1' & \boldsymbol{e}_2' & \cdots & \boldsymbol{e}_n' \end{bmatrix}
    \quad \text{and} \quad
    \b = \frac{1}{2} (\x_1^* + \x_2^*).
\end{equation*}
We let $\x_{\mathcal{J}} = \boldsymbol{E}^\intercal (\x - \b)$ be the coordinate transformation. One can verify that if $\x = \x_1^*$ then $\x_{\mathcal{J}} = \big( - \frac{1}{2} \| \x_2^* - \x_1^* \|, \; \mathbf{0} \big) = (-r, \mathbf{0} )$ and if $\x = \x_2^*$ then $\x_{\mathcal{J}} = \big( \frac{1}{2} \| \x_2^* - \x_1^* \|, \; \mathbf{0} \big) = (r, \mathbf{0} )$.
For $i \in \{ 1, 2 \}$, let $\Tilde{f}_i : \R^n \to \R$ be the function such that $\tilde{f}_i (\x_{\mathcal{J}}) = f_i (\x)$ for all $\x \in \R^n$, i.e., $\tilde{f}_i$'s value at the coordinate of point $\x$ on the new bases $\mathcal{J}$ is the same as $f_i$ at point $\x$. We can write $\tilde{f}_i (\x) = f_i ( \boldsymbol{Ex} + \b)$ for $i \in \{ 1 ,2 \}$. Applying Lemma~\ref{lem: strong cvx}, we have that $\tilde{f}_i$ is $\sigma_i$-strongly convex for $i \in \{ 1 ,2 \}$. Once we attain the potential solution region $\M$ in terms of $\x_{\mathcal{J}}$, we can always use the transformation to obtain the region in terms of $\x$, i.e., the original coordinate system.

For convenience, we introduce the shorthand notation of sets that will be encountered throughout the paper. Recall the definition of $\B$ from \eqref{def: ball B}. For $i \in \{ 1, 2 \}$, define
\begin{equation}
    \B_i \triangleq \B \Big( \x_i^*, \frac{L}{\sigma_i} \Big).
    \label{def: set B}
\end{equation}
Now, we introduce the functions that will be used to define the outer and inner approximations of $\M$.
For $i \in \{ 1, 2 \}$, define the functions $\tilphi_i: \; \overline{\B}_i \to \big[ 0, \frac{\pi}{2} \big]$ to be such that
\begin{equation}
    \tilphi_i (\x) \triangleq \arccos{ \Big( \frac{\sigma_i}{L} \| \x - \x_i^* \| \Big) },
    \label{def: phi tilde}
\end{equation}
and the functions $\alpha_i: \; \R^n \setminus \{ \x_i^* \} \to [0, \pi]$ to be such that
\begin{equation}
    \alpha_i(\x) \triangleq \angle ( \x - \x_i^*, \; \x_2^* - \x_1^* ),
    \label{def: alpha_i}
\end{equation}
i.e., the angle between vectors $\x - \x_i^*$ and $\x_2^* - \x_1^*$. Note that $\alpha_2 (\x) \geq \alpha_1 (\x)$ for all $\x \in \R^n \setminus \{ \x_1^*, \x_2^* \}$ due to the assumption that $\x_1^* = (-r, \mathbf{0})$ and $\x_2^* = (r, \mathbf{0})$.
We define $\psi: \R^n \setminus \{ \x_1^*, \x_2^* \} \to [0, \pi]$ to be the function such that
\begin{equation}
    \psi( \x ) \triangleq \pi - \big( \alpha_2(\x) - \alpha_1(\x) \big). \label{def: psi}
\end{equation}
The interpretation of the angles $\tilphi_i (\x)$ and $\psi(\x)$ will be clarified later (in Fig.~\ref{fig: angle and nc}).
In addition, given $\x_1^*, \x_2^*  \in \R^n$, $\sigma_1, \sigma_2 \in \R_{>0}$, and $L \in \R_{>0}$, we define 
\begin{equation}
    \mathcal{X} \triangleq \begin{cases} 
    \Big\{ \x \in \R^n : 
    \| \x - \x_i^* \| = \frac{L}{\sigma_i} 
    \;\; \text{for all} \; i \in \{ 1, 2 \} \Big\}
    \; &\text{if} \;
    \| \x_2^* - \x_1^* \| = L \big( \frac{1}{\sigma_1} + \frac{1}{\sigma_2} \big),  \\ 
    \emptyset &\text{otherwise}.
    \end{cases}
    \label{def: set X}
\end{equation}
Due to the assumption that $\x_1^* = (-r, \mathbf{0})$ and $\x_2^* = (r, \mathbf{0})$, for $\| \x_2^* - \x_1^* \| = L \big( \frac{1}{\sigma_1} + \frac{1}{\sigma_2} \big)$, we have $\mathcal{X} = \big\{ \big( -r + \frac{L}{\sigma_1}, \; \mathbf{0} \big) \big\}$.

With these definitions in place, given $\x_1^*, \x_2^*  \in \R^n$, $\sigma_1, \sigma_2 \in \R_{>0}$, and $L \in \R_{>0}$, we define the outer and inner approximations of $\M$ as
\begin{equation}
    \Mup (\x_1^*, \x_2^*, \sigma_1, \sigma_2, L) \triangleq \big\{ \x \in 
    \R^n : \;
    \tilphi_1 (\x) + \tilphi_2 (\x) \geq \psi(\x) \big\} 
    \label{def: set M up}
\end{equation}
and
\begin{equation}
    \Mdo (\x_1^*, \x_2^*, \sigma_1, \sigma_2, L) \triangleq \big\{ \x \in \R^n : \;
    \tilphi_1 (\x) + \tilphi_2 (\x) > \psi(\x) \big\} 
    \cup \mathcal{X},
    \label{def: set M down}
\end{equation}
respectively. As before, we will omit the argument of the sets $\Mup(\x_1^*, \x_2^*, \sigma_1, \sigma_2, L)$ and $\Mdo(\x_1^*, \x_2^*, \sigma_1, \sigma_2, L)$, and write them as $\Mup$ and $\Mdo$, respectively.

\begin{remark}
Recall the definition of $\tilphi_i$ for $i \in \{ 1, 2 \}$ and $\psi$ from \eqref{def: phi tilde} and \eqref{def: psi}, respectively. Since $\Mup$ and $\Mdo$ are defined using $\tilphi_1$, $\tilphi_2$ and $\psi$, implicitly, they must be subsets of $\textbf{dom} (\tilphi_1) \cap \textbf{dom} (\tilphi_2) \cap \textbf{dom} (\psi)$. In other words, the sets $\Mup \subseteq (\overline{\B}_1 \cap \overline{\B}_2) \setminus \{ \x_1^*, \x_2^* \}$ and $\Mdo \subseteq (\overline{\B}_1 \cap \overline{\B}_2) \setminus \{ \x_1^*, \x_2^* \}$ where $\B_i$ for $i \in \{ 1 ,2 \}$ are defined in \eqref{def: set B}.
\end{remark}

In order to characterize the potential solution region $\M$, we proceed as follows. First, in Proposition~\ref{prop: angle nc}, we show that $\M \subseteq \Mup$ by considering a property of strongly convex functions. Then, we characterize the boundary and interior of the outer approximation ($\partial \Mup$ and $(\Mup )^{\circ}$) for each value of $r$ in Theorem~\ref{thm: up BD}. In Proposition~\ref{prop: function exist n-dim}, we consider quadratic functions and show that $\Mdo \subseteq \M$ in Proposition~\ref{prop: angle sc}. We use a similar approach as in Theorem~\ref{thm: up BD} to characterize the boundary and interior of the inner approximation ($\partial \Mdo$ and $(\Mdo )^{\circ}$) for each value of $r$ which is presented in Theorem~\ref{thm: down BD}. Finally, by observing that $\partial \Mup = \partial \Mdo$ and $(\Mup )^{\circ} = (\Mdo )^{\circ}$ from Theorem~\ref{thm: up BD} and Theorem~\ref{thm: down BD}, we conclude the paper by showing that, in fact, the boundary of the potential solution region, outer approximation, and inner approximation are identical, i.e., $\partial \M = \partial \Mup = \partial \Mdo$, in Theorem~\ref{thm: boundary M}.

%% file: contents/sec-outer.tex
\section{Outer approximation} \label{sec: outer}

In this section, we derive necessary conditions for a point to be in the potential solution region $\M$ and show that $\M \subseteq \Mup$ in Proposition~\ref{prop: angle nc}. 
Then, we explicitly characterize an important part of $\partial \Mup$ (and also $\partial \Mdo$) in Proposition~\ref{prop: T_n}. In Theorem~\ref{thm: up BD}, which is the main result of this section, we identify $\partial \Mup$ and $(\Mup )^{\circ}$, and also provide a property of $\Mup$.
Other lemmas in this section are presented as tools that will be utilized in the proof of Theorem~\ref{thm: up BD} (and also Theorem~\ref{thm: down BD}).

We will be using the following functions throughout our analysis. For $i \in \{1, 2\}$, define $\u_i: \; \R^n \setminus \{ \x_i^* \} \to \R^n$ to be the function such that
\begin{equation}
    \u_i(\x) \triangleq \frac{\x - \x_i^*}{ \| \x - \x_i^* \|},
    \label{def: unit vector}
\end{equation}
i.e., the unit vector in the direction of $\x - \x_i^*$. 
Recall the definition of $\angle (\cdot, \cdot)$ from \eqref{def: angle func}. For $i \in \{ 1, 2 \}$, we define $\phi_i: \R^n \setminus \{ \x_i^* \} \to \big[ 0, \frac{\pi}{2} \big]$ to be the function such that
\begin{equation}
    \phi_i( \x ) \triangleq \angle \big( \nabla f_i (\x), \; \u_i(\x) \big),
    \label{def: phi_i}
\end{equation}
and
$\underline{L}_i: \R^n \to \R$ to be the function such that 
\begin{equation}
    \underline{L}_i (\x) \triangleq \sigma_i  \Vert \x - \x_i^* \Vert. 
    \label{eqn: L_lb}
\end{equation}
Note that for $i \in \{ 1, 2 \}$, the quantity $\underline{L}_i (\x)$ is a lower bound on the norm of the gradient of $f_i$ at $\x \in \R^n$ if $f_i \in \S (\x_i^*, \sigma_i)$.

In Fig.~\ref{fig: angle and nc}, we illustrate the definition of $\u_i$, $\phi_i$, $\tilphi_i$, $\alpha_i$ for $i \in \{ 1, 2 \}$, and $\psi$. Moreover, we illustrate the inequality $\tilphi_1 (\x) + \tilphi_2 (\x) \geq \psi(\x)$ which is used to describe the outer approximation $\Mup$ in \eqref{def: set M up}.

In the following proposition, we show a crucial result that the set $\Mup$ covers the set that we want to characterize, $\M$. In other words, the points in the set $\Mup$ satisfy necessary conditions of a point to be a minimizer of the sum $f_1 + f_2$.

\begin{proposition}
Suppose the sets $\M$ and $\Mup$ are defined as in \eqref{def: set M} and \eqref{def: set M up}, respectively. Then, $\M \subseteq \Mup$.
\label{prop: angle nc}
\end{proposition} 

\begin{proof}
Recall the definition of the sets $\B_i$ for $i \in \{ 1, 2 \}$, the angles $\tilphi_i$ for $i \in \{ 1, 2 \}$, and the angle $\psi$ from \eqref{def: set B}, \eqref{def: phi tilde}, and \eqref{def: psi}, respectively. First, we want to show that the necessary conditions for a point $\x \in \R^n \setminus \{ \x_1^*, \x_2^* \}$ to be in $\M$ are
\begin{enumerate}[label=(\roman*)]
    \item $\x \in \overline{\B}_1 \cap \overline{\B}_2$, and 
    \label{lem: necessary part1}
    \item $\tilphi_1 (\x) + \tilphi_2 (\x) \geq \psi(\x)$. 
    \label{lem: necessary part2}
\end{enumerate}

From the definition of strongly convex functions in \eqref{def: strongly cvx}, we have 
\begin{equation*}
\big\langle \nabla f_i(\x) - \nabla f_i(\y) , \; \x - \y \big\rangle \geq \sigma_i \Vert \x - \y \Vert^2
\end{equation*}
for all $\x, \y \in \R^n$ and for $i \in \{1, 2 \}$. For $i \in \{ 1, 2 \}$, recall the definition of $\u_i(\x)$ and $\phi_i (\x)$ from \eqref{def: unit vector} and \eqref{def: phi_i}, respectively. Since $\x_1^*$ and $\x_2^*$ are the minimizers of $f_1$ and  $f_2$, respectively, for $\x \notin \{ \x_1^*, \x_2^* \}$, we get 
\begin{align}
    \big\langle \nabla f_i(\x) - \nabla f_i(\x_i^*), \; \x - \x_i^* \big\rangle &\geq \sigma_i \Vert \x - \x_i^* \Vert^2, \nonumber \\
    \Leftrightarrow \quad 
    \Vert \nabla f_i(\x) \Vert \cos(\phi_i (\x))
    = \langle \nabla f_i(\x), \; \u_i (\x) \rangle 
    &\geq \sigma_i \Vert \x - \x_i^* \Vert > 0. \label{eqn: strongly convex eq1}
\end{align}
Suppose $\x$ is a candidate minimizer. Then, we have that $\Vert \nabla f_i(\x) \Vert \leq L$ for $i \in \{ 1, 2 \}$ by our assumption. Recall the definition of $\underline{L}_i$ for $i \in \{ 1, 2 \}$ from \eqref{eqn: L_lb}. Inequality \eqref{eqn: strongly convex eq1} becomes
\begin{equation}
\cos(\phi_i (\x)) \geq \frac{ \sigma_i}{L} \Vert \x - \x_i^* \Vert = \frac{\underline{L}_i (\x)}{L}. \label{eqn: angle eq1}
\end{equation}
If $\underline{L}_1 (\x) > L$ or $\underline{L}_2 (\x) > L$, we have that $\x$ cannot be the minimizer of the function $f_1 + f_2$ since there is no $\phi_i(\x)$ that can satisfy inequality \eqref{eqn: angle eq1}. Thus, a necessary condition for $\x \in \R^n \setminus \{ \x_1^*, \x_2^* \}$ to be a minimizer of $f_1 + f_2$ is that $\underline{L}_i (\x) \leq L$ for $i \in \{1, 2 \}$ or equivalently, $\x \in \overline{\B}_1 \cap \overline{\B}_2$, yielding part~\ref{lem: necessary part1} of the claim. We now prove part \ref{lem: necessary part2}. 

From the definition of $\psi(\x)$ in \eqref{def: psi} and that $\angle (-\u_1 (\x), -\u_2 (\x) )$, $\alpha_1(\x)$ and $\pi - \alpha_2(\x)$ are the angles of the triangle formed by the points $\x$, $\x_1^*$ and $\x_2^*$, we can write that for all $\x \in \R^n \setminus \{ \x_1^*, \x_2^* \}$,
\begin{equation}
    \psi(\x) 
    = (\pi - \alpha_2(\x)) + \alpha_1(\x) 
    = \pi -  \angle (-\u_1 (\x), -\u_2 (\x)) 
    = \angle (\u_1 (\x), -\u_2 (\x)).
    \label{eqn: psi equality}
\end{equation}

Suppose that $\x \in \overline{\B}_1 \cap \overline{\B}_2$. Recall the definition of $\tilphi_i$ for $i \in \{ 1 ,2 \}$ from \eqref{def: phi tilde}. From inequality \eqref{eqn: angle eq1}, we have $\phi_i (\x) \leq \tilphi_i (\x)$ for $i \in \{ 1, 2 \}$. 
If $\tilphi_1 (\x) + \tilphi_2 (\x) < \psi(\x)$, then using \eqref{def: phi_i} and \eqref{eqn: psi equality}, we have
\begin{equation*}
    \angle ( \nabla f_1 (\x), \u_1 (\x) ) + \angle ( -\nabla f_2 (\x), -\u_2 (\x) )
    < \angle ( \u_1 (\x), - \u_2 (\x) ).
\end{equation*}
However, using \cite[Corollary~12]{castano2016angles}, we can write $\angle ( \u_1 (\x), - \u_2 (\x) ) \leq$ $\angle ( \nabla f_1 (\x), \u_1 (\x))$ $+ \angle ( \nabla f_1 (\x), - \u_2 (\x))$. 
Therefore, if $\tilphi_1 (\x) + \tilphi_2 (\x) < \psi(\x)$, we have that $\nabla f_1 (\x) \neq - \nabla f_2 (\x)$ which implies that $\x$ is not the minimizer of $f_1 + f_2$. This means that one of the necessary conditions is that $\tilphi_1 (\x) + \tilphi_2 (\x) \geq \psi(\x)$ which completes the proof of the claim.

In the above analysis, we considered the case when $\x \in \R^n \setminus \{ \x_1^*, \x_2^* \}$. We are left with the case when $\x \in \{ \x_1^*, \x_2^* \}$.
From the definition of strongly convex functions, for all $\x, \y \in \R^n$,
\begin{equation*}
    \big\langle \nabla f_2(\x) - \nabla f_2(\y) , \; \x - \y \big\rangle \geq \sigma_2 \Vert \x - \y \Vert^2.
\end{equation*}
Since $\x_2^*$ is the minimizer of $f_2$ and $\x_1^* \neq \x_2^*$, we get 
\begin{displaymath}
    \big\langle \nabla f_2(\x_1^*) , \; \x_1^* - \x_2^* \big\rangle
    = \big\langle \nabla f_2(\x_1^*) - \nabla f_2(\x_2^*) , \; \x_1^* - \x_2^* \big\rangle
    \geq \sigma_2 \Vert \x_1^* - \x_2^* \Vert^2
    > 0,
\end{displaymath}
and thus, $\nabla f_2(\x_1^*) \neq \boldsymbol{0}$. This implies that $\nabla f_2(\x_1^*) + \nabla f_1(\x_1^*) \neq \boldsymbol{0}$ and $\x_1^*$ is not the minimizer of $f_1 + f_2$. By using similar approach, we can also conclude that $\x_2^*$ is not the minimizer of $f_1 + f_2$.
\end{proof}

\begin{figure}
\centering
\subfloat[]{\includegraphics[width=.45\linewidth]{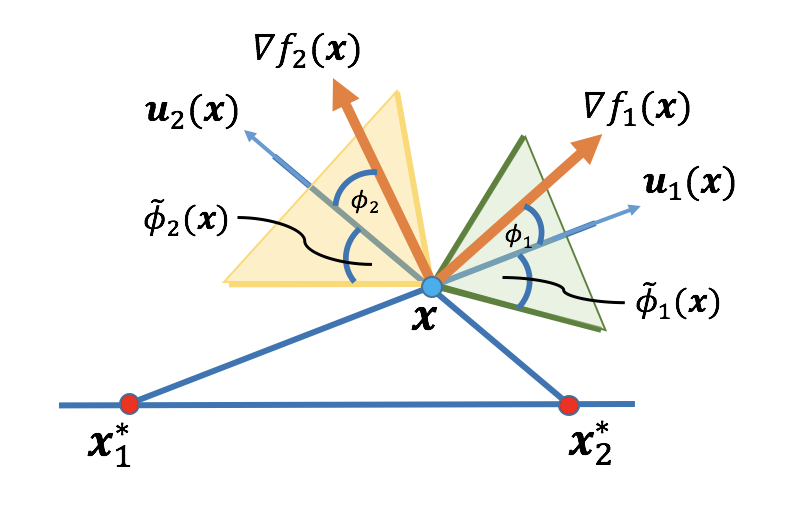} \label{fig: angle and nc fig1}}\quad
\subfloat[]{\includegraphics[width=.45\linewidth]{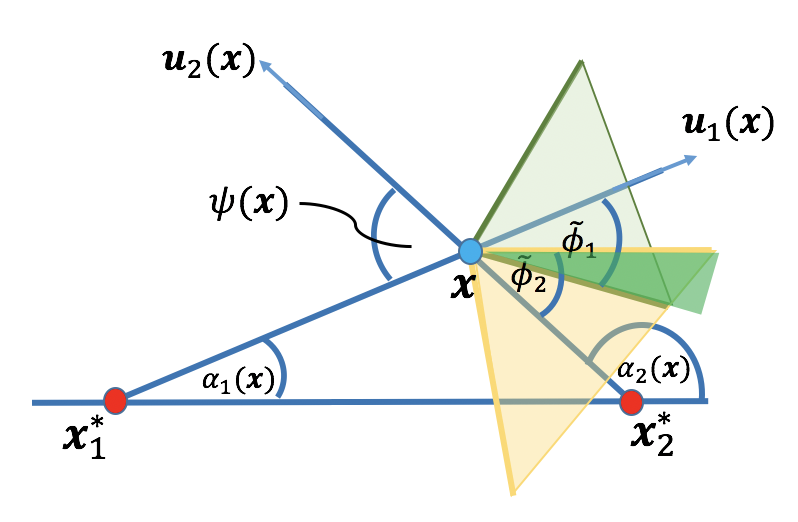} \label{fig: angle and nc fig2}}
\caption{(a) The figure illustrates the definition of $\u_i$, $\phi_i$, and $\tilphi_i$ for $i \in \{ 1, 2 \}$. In particular, inequality \eqref{eqn: angle eq1} implies that $\phi_i (\x) \in [0, \tilphi_i (\x) ]$ for $i \in \{ 1, 2 \}$, i.e., the gradient vectors $\nabla f_1 (\x)$ and $\nabla f_2 (\x)$ must lie in the corresponding shaded regions.
(b) The figure illustrates the definition of $\alpha_i$ for $i \in \{ 1, 2 \}$ and $\psi$. In addition, the inequality $\tilphi_1 (\x) + \tilphi_2 (\x) \geq \psi(\x)$ in $\Mup$ means that there is an overlapping region (light green region in the figure) caused by one shaded region and the mirror of the other shaded region.} 
\label{fig: angle and nc}
\end{figure}

\begin{remark}
The angle functions $\tilphi_i$ and $\alpha_i$ for $i \in \{ 1, 2\}$ defined in \eqref{def: phi_i} and \eqref{def: alpha_i}, respectively, can be expressed as functions of the distances $\Vert \x_1^* - \x_2^* \Vert$, $\Vert \x - \x_1^* \Vert$, and $\Vert \x - \x_2^* \Vert$. 
This means that the inequality $\tilphi_1 (\x) + \tilphi_2 (\x) \geq \psi (\x)$ depends only on the distance among three points $\x$, $\x_1^*$ and $\x_2^*$.
Since $\x_1^* = (-r, \mathbf{0})$ and $\x_2^* = (r, \mathbf{0})$, we conclude that the shape of $\Mup$ (and $\Mdo$) is symmetric around $x_1$-axis.
\end{remark}

From this point, we will denote $\x = (x_1, \tilde{\mathbf{x}}) \in \R^n$ where $x_1 \in \R$ and $\tilde{\mathbf{x}} = (x_2, x_3, \ldots, x_n) \in \R^{n-1}$. Next, we will provide an algebraic expression for a certain portion of $\partial \Mup$ (and $\partial \Mdo$) based on the geometric equation $\tilphi_1 (\x) + \tilphi_2 (\x) = \psi (\x)$, where $\tilphi_i$ for $i \in \{ 1, 2 \}$ and $\psi$ are defined in \eqref{def: phi tilde} and \eqref{def: psi}, respectively. For convenience, we define
\begin{equation}
\begin{aligned}
    d_1(\x) &\triangleq \| \x - \x_1^* \|
    = \sqrt{(x_1 + r)^2 + \Vert \tilde{\mathbf{x}} \Vert^2}
    \quad \text{and} \\
    d_2(\x) &\triangleq \| \x - \x_2^* \|
    = \sqrt{(x_1 - r)^2 + \Vert \tilde{\mathbf{x}} \Vert^2}.
\end{aligned}
\label{def: distance}
\end{equation}
Define the set of points 
\begin{equation}
    \T \triangleq \bigg\{ \x \in \R^n: \;
    \frac{\Vert \x \Vert^2 - r^2}{d_1^2 (\x) \cdot d_2^2 (\x)}  + \frac{\sigma_1 \sigma_2}{L^2} 
    = \sqrt{\frac{1}{d_1^2(\x)} - \frac{\sigma_1^2}{L^2}} \cdot \sqrt{\frac{1}{d_2^2(\x)} - \frac{\sigma_2^2}{L^2}}  \bigg\}. \label{def: set T}
\end{equation}

\begin{proposition}  \label{prop: T_n}
The set $\T$ defined in \eqref{def: set T} can equivalently be written as $\T = \big\{ \x \in \R^n : \tilphi_1 (\x) + \tilphi_2 (\x) = \psi (\x) \big\}$.
\end{proposition}

\begin{proof}
Based on the definition of $\alpha_i (\x)$ for $i \in \{ 1, 2 \}$ in \eqref{def: alpha_i}, for any point $\x \in \R^n \setminus \{ \x_1^*, \x_2^* \}$, we have
\begin{align}
    x_1 = d_1 (\x) \cos (\alpha_1 (\x) ) - r 
    =& \; d_2 (\x) \cos (\alpha_2 (\x) ) + r,  \nonumber \\
    \Leftrightarrow \quad 
    \cos (\alpha_1 (\x) ) = \frac{x_1 + r}{d_1 (\x)}
    \quad \text{and}& \quad
    \cos (\alpha_2 (\x) ) = \frac{x_1 - r}{d_2 (\x)}. 
    \label{eqn: cos alpha}
\end{align}
Similarly,
\begin{align}
    \Vert \tilde{\mathbf{x}} \Vert 
    = d_1(\x) \sin (\alpha_1 (\x) ) 
    =& \; d_2(\x) \sin ( \alpha_2 (\x) ), 
    \nonumber \\
    \Leftrightarrow \quad 
    \sin ( \alpha_1 (\x) ) = \frac{\Vert \tilde{\mathbf{x}} \Vert}{d_1 (\x)}
    \quad \text{and}& \quad
    \sin (\alpha_2 (\x) ) = \frac{\Vert \tilde{\mathbf{x}} \Vert}{d_2 (\x)}. 
    \label{eqn: sin alpha}
\end{align}
Since $\tilphi_i (\x) \in \big[0, \frac{\pi}{2} \big]$ for $i \in \{1, 2\}$, we get $\tilphi_1 (\x) + \tilphi_2 (\x) \in [0, \pi]$. 
Recall from \eqref{def: psi} that $\psi (\x) \in [0, \pi]$. Since the cosine function is one-to-one for this range of angles, equation $\tilphi_1 (\x) + \tilphi_2 (\x) = \psi (\x)$ is equivalent to
\begin{displaymath}
	\cos \big( \tilphi_1 ( \x) + \tilphi_2 (\x) \big) 
	= \cos \big(\pi - (\alpha_2 (\x) - \alpha_1 (\x)) \big) 
    = - \cos \big(\alpha_2 (\x) - \alpha_1 (\x) \big). 
\end{displaymath}
Expanding this equation and substituting \eqref{eqn: cos alpha}, \eqref{eqn: sin alpha}, and $\cos(\tilphi_i (\x)) = \frac{ \sigma_i}{L} d_i (\x)$ for $i \in \{ 1,2 \}$, we get
\begin{equation*}
	\frac{\sigma_1}{L} d_1 (\x) \cdot \frac{\sigma_2}{L} d_2 (\x) - \sqrt{1 - \Big(\frac{\sigma_1}{L} d_1 (\x) \Big)^2} \cdot \sqrt{1 - \Big(\frac{\sigma_2}{L} d_2 (\x) \Big)^2}  
    = - \frac{x_1 -r}{d_2 (\x) } \cdot\frac{x_1 +r}{d_1 (\x)}  - \frac{\Vert \tilde{\mathbf{x}} \Vert}{d_2 (\x)} \cdot \frac{\Vert \tilde{\mathbf{x}} \Vert}{d_1 (\x)}. 
\end{equation*}
Dividing the above equation by $d_1 (\x) \cdot d_2 (\x)$ and rearranging it yields the result.
\end{proof}

The subsequent lemmas (Lemma~\ref{lem: 1D analyze} - \ref{lem: boundary inclusion}) are useful ingredients for proving the characterization of the outer approximation $\Mup$ (defined in \eqref{def: set M up}) given in Theorem~\ref{thm: up BD}, and their proofs are provided in Appendix~\ref{sec: proof lemma outer}.

The following lemma provides a sufficient condition for the minimizers $\x_1^*$ and $\x_2^*$ to be on the boundary of the outer approximation $\Mup$ and the inner approximation $\Mdo$.

\begin{lemma}
Let $\Mup$ and $\Mdo$ be as defined in \eqref{def: set M up} and \eqref{def: set M down}, respectively. 
\begin{enumerate}[label=(\roman*)]
    \item If $r \in \big( 0, \; \frac{L}{2 \sigma_2} \big]$ then $\x_1^* \in \partial \Mup$ and $\x_1^* \in \partial \Mdo$.
    \label{lem: 1D case1}
    \item If $r \in \big( 0, \; \frac{L}{2 \sigma_1} \big]$ then $\x_2^* \in \partial \Mup$ and $\x_2^* \in \partial \Mdo$.
    \label{lem: 1D case2}
\end{enumerate}
\label{lem: 1D analyze}
\end{lemma}

In the next lemma, we provide a property of points in a particular set which will be used to characterize the sets $\Mup$ and $\Mdo$ defined in \eqref{def: set M up} and \eqref{def: set M down}, respectively. Roughly speaking, if $\x \in \Mup$ and $x_1 \in [-r, r]$, then each point that has the same first component and is closer to the $x_1$-axis is also in $\Mup$.
\begin{lemma}
Consider two points $\x = (x_1, \tilde{\mathbf{x}} )$ and $\y = ( y_1, \tilde{\mathbf{y}})$. Suppose $-r \leq x_1 = y_1 \leq r$ and $\Vert \tilde{\mathbf{x}} \Vert > \Vert \tilde{\mathbf{y}} \Vert$. If $\tilphi_1 (\x) + \tilphi_2 (\x) \geq \psi (\x)$ then either $\tilphi_1 (\y) + \tilphi_2 (\y) > \psi (\y)$ or $\y \in \{ \x_1^*, \x_2^* \}$.
\label{lem: point below}
\end{lemma}

Recall the definition of $\B$ from \eqref{def: set B}. Since $\x_1^* = (-r, \mathbf{0})$ and $\x_2^* = (r, \mathbf{0})$ by our assumption, we can explicitly write $\partial \B_i$ for $i \in \{1,2 \}$ as follows:
\begin{equation}
\begin{aligned}
    \partial \B_1
    = \partial \B \Big(\x_1^*, \frac{L}{\sigma_1} \Big) &= \bigg\{ \x \in \R^n: (x_1+r)^2 + \Vert \tilde{\mathbf{x}} \Vert^2 
    = \frac{L^2}{\sigma_1^2} \bigg\}, \\
    \partial \B_2
    = \partial \B \Big(\x_2^*, \frac{L}{\sigma_2} \Big) &= \bigg\{ \x \in \R^n: (x_1-r)^2 + \Vert \tilde{\mathbf{x}} \Vert^2 
    = \frac{L^2}{\sigma_2^2} \bigg\}. 
    \label{def: set B boundary}
\end{aligned}
\end{equation}

For convenience, we define 
\begin{equation}
    \gamma_i \triangleq \frac{L^2}{\sigma_i^2} 
    \;\; \text{for} \;\; i \in \{1, 2 \} 
    \quad \text{and} \quad 
    \beta \triangleq \frac{\sigma_2}{\sigma_1}.
    \label{def: gamma beta}
\end{equation}
By using the definitions above, we define 
\begin{equation}
    \lambda_1 \triangleq \Big(\frac{1 + \beta}{1 + 2 \beta} \Big) \frac{\gamma_1}{2 r} - \frac{r}{1 + 2 \beta} 
    \quad \text{and} \quad
    \lambda_2 \triangleq -\Big(\frac{1 + \beta}{2 +  \beta} \Big) \frac{\gamma_2}{2 r} + \frac{\beta r}{2 + \beta}. 
\label{var: lambda}
\end{equation}

In the following lemma, we will show that if $\x \in \partial \B_1 \cup \partial \B_2$, the value of the first component $x_1$ is necessary and sufficient to determine whether $\x$ is in $\Mup$ and $\Mdo$, which are defined in \eqref{def: set M up} and \eqref{def: set M down}, respectively. In other words, the angle condition $\tilphi_1 (\x) + \tilphi_2 (\x) \lessgtr \psi(\x)$ can be simplified if we consider a point in $\partial \B_1$ or $\partial \B_2$.
\begin{lemma}
Let $\lambda_1$ and $\lambda_2$ be as defined in \eqref{var: lambda}. Consider $\x = (x_1, \tilde{\mathbf{x}}) \in (\overline{\B}_1 \cap \overline{\B}_2) \setminus \{ \x_1^*, \x_2^* \}$.
\begin{enumerate}[label=(\roman*)]
    \item If $\x \in \partial \B_1$ then $\tilphi_1 (\x) + \tilphi_2 (\x) \lessgtr \psi(\x)$ if and only if $x_1 \lessgtr \lambda_1$.
    \label{lem: lambda_1}
    \item If $\x \in \partial \B_2$ then $\tilphi_1 (\x) + \tilphi_2 (\x) \lessgtr \psi(\x)$ if and only if $x_1 \gtrless \lambda_2$.
    \label{lem: lambda_2}
\end{enumerate}
\label{lem: point on R_n}
\end{lemma}

In the following lemma, we will show that the points in the set of intersection between $\T$ and $\partial \B_1$ (resp. $\T$ and $\partial \B_2$) have the same first component, if the intersection is non-empty. Moreover, the first component of these points is $\lambda_1$ (resp. $\lambda_2$) where $\lambda_i$ for $i \in \{ 1, 2 \}$ are defined in \eqref{var: lambda}. By using the definition of $\gamma_1$, $\gamma_2$ and $\beta$ in \eqref{def: gamma beta}, define
\begin{equation}
\begin{aligned}
    \nu_1 &\triangleq \frac{r}{2(1+2 \beta)} \sqrt{- \Big(\frac{\gamma_1}{r^2} - 4 \Big) \Big((1+\beta)^2 \frac{\gamma_1}{r^2}  - 4 \beta^2 \Big)}
    \quad \text{and}\\
    \nu_2 &\triangleq \frac{r}{2(2+ \beta)} \sqrt{- \Big(\frac{\gamma_2}{r^2} - 4 \Big) \Big((1+\beta)^2 \frac{\gamma_2}{r^2}  - 4  \Big)}.  
    \label{var: nu}
\end{aligned}
\end{equation}

\begin{lemma}
Consider the sets of points $\T$ and $\partial \B_i$ for $i \in \{1, 2\}$ defined in \eqref{def: set T} and \eqref{def: set B boundary}, respectively. 
Let $\lambda_i$ and $\nu_i$ for $i \in \{ 1, 2 \}$ be as defined in \eqref{var: lambda} and \eqref{var: nu}, respectively.
\begin{enumerate}[label=(\roman*)]
    \item For $i \in \{ 1, 2 \}$, if $r \in \big( 0, \frac{L}{2 \sigma_i} \big]$, then $\T \cap \partial \B_i = \emptyset$.
    \label{lem: intersection 1}
    \item For $i \in \{ 1, 2 \}$, if $r \in \big( \frac{L}{2 \sigma_i}, \frac{L}{2} (\frac{1}{\sigma_1} + \frac{1}{\sigma_2}) \big]$, then 
    $\T \cap \partial \B_i
    = \big\{ \x \in \R^n: x_1 = \lambda_i, \;\; \Vert \tilde{\mathbf{x}} \Vert = \nu_i \big\}$.
    \label{lem: intersection 2}
\end{enumerate}
\label{lem: intersection}
\end{lemma}

Recall that $\lambda_1$ and $\lambda_2$ are defined in \eqref{var: lambda}. In the following lemma, for $i \in \{ 1, 2 \}$, we consider a relationship between $\frac{L}{\sigma_i r}$ and $\frac{\lambda_i}{r}$. In particular, for $r \in \big( 0, \frac{L}{2} (\frac{1}{\sigma_1} + \frac{1}{\sigma_2}) \big]$, recall from Lemma~\ref{lem: intersection} that if $\T \cap \partial \B_1 \neq \emptyset$ (resp. $\T \cap \partial \B_2 \neq \emptyset$), then every point in the intersection has the first component equal to $\lambda_1$ (resp. $\lambda_2$).
The next lemma compares $\lambda_1$ to the maximum value of the first component over all points of $\partial \B_1$ (which is $-r + \frac{L}{\sigma_1}$), and compares $\lambda_2$ to the minimum value of the first component over all points of $\partial \B_2$ (which is $r - \frac{L}{\sigma_1}$), respectively. 
\begin{lemma}
Let $\lambda_i$ for $i \in \{1, 2 \}$ be as defined in \eqref{var: lambda}.
\begin{enumerate}[label=(\roman*)]
    \item If $r \in \big( 0, \; \frac{L}{2 \sigma_1} \big]$ then $\lambda_1 \geq \frac{L}{\sigma_1} -r$, with equality only if $r = \frac{L}{2 \sigma_1}$.
    \label{lem: compare x1 part1}
    \item $r \in \big( \frac{L}{2 \sigma_1}, \; \frac{L}{2} \big(\frac{1}{\sigma_1} + \frac{1}{\sigma_2} \big) \big)$
    if and only if $\lambda_1 < \frac{L}{\sigma_1} -r$. 
    \label{lem: compare x1 part2}
    \item If $r \in \big( 0, \; \frac{L}{2 \sigma_2} \big]$ then $\lambda_2 \leq r- \frac{L}{\sigma_2}$, with equality only if $r = \frac{L}{2 \sigma_2}$.
    \label{lem: compare x1 part3}
    \item $r \in \big( \frac{L}{2 \sigma_2}, \; \frac{L}{2} \big(\frac{1}{\sigma_1} + \frac{1}{\sigma_2} \big) \big)$ if and only if $\lambda_2 > r - \frac{L}{\sigma_2} $.
    \label{lem: compare x1 part4}
\end{enumerate}
\label{lem: compare x1}
\end{lemma}

In the next lemma, we will consider a relationship between $\T$ and $\partial \Mup$ (the boundary of the outer approximation), and $\T$ and $\partial \Mdo$ (the boundary of the inner approximation). In particular, we will show that $\T \subseteq \partial \Mup$ and $\T \subseteq \partial \Mdo$ for a particular range of $r$.
\begin{lemma}
Let $\Mup$, $\Mdo$ and $\T$ be defined as in \eqref{def: set M up}, \eqref{def: set M down} and \eqref{def: set T}, respectively. If $r \in \Big(0, \; \frac{L}{2} \big( \frac{1}{\sigma_1} + \frac{1}{\sigma_2} \big) \Big)$, then $\T \subseteq \partial \Mup$ and $\T \subseteq \partial \Mdo$.
\label{lem: boundary}
\end{lemma}

For $i \in \{ 1, 2 \}$, define half-planes
\begin{equation}
    \H_i^+ \triangleq \{ \x \in \R^n: x_1 \geq \lambda_i \}
    \quad \text{and} \quad 
    \H_i^- \triangleq \{ \x \in \R^n: x_1 \leq \lambda_i \},
    \label{def: set H_i}
\end{equation}
where $\lambda_i$ for $i \in \{1, 2 \}$ are defined in \eqref{var: lambda}. In the lemma below, we examine properties of points $\x \in \partial \B_1 \cap \H_1^+$ (resp. $\x \in \partial \B_2 \cap \H_2^-$).

\begin{lemma}
Let the sets $\B_i$ for $i \in \{1,2\}$, and   $\H_1^+$ and $\H_2^-$ be defined as in \eqref{def: set B} and \eqref{def: set H_i}, respectively.
\begin{enumerate}[label=(\roman*)]
    \item If $r \in \Big( \frac{L}{2 \sigma_1}, \frac{L}{2} \big( \frac{1}{\sigma_1} + \frac{1}{\sigma_2} \big) \Big)$
    then $\big[ \lambda_1, \; -r + \frac{L}{\sigma_1} \big] \subseteq (-r, r)$ and
    $\partial \B_1 \cap \H_1^+ \subseteq (\overline{\B}_1 \cap \overline{\B}_2) \setminus 
    \{ \x_1^*, \x_2^* \}$.
    \label{lem: boundary inclusion part1}
    \item If $r \in \Big( \frac{L}{2 \sigma_2}, \frac{L}{2} \big( \frac{1}{\sigma_1} + \frac{1}{\sigma_2} \big) \Big)$
    then $\big[ r - \frac{L}{\sigma_2}, \; \lambda_2  \big] \subseteq (-r, r)$ and 
    $\partial \B_2 \cap \H_2^- \subseteq (\overline{\B}_1 \cap \overline{\B}_2) \setminus 
    \{ \x_1^*, \x_2^* \}$.
    \label{lem: boundary inclusion part2}
\end{enumerate}
\label{lem: boundary inclusion}
\end{lemma}

In the theorem below, we give the characterization of the boundary $\partial \Mup$ and interior $( \Mup )^{\circ}$, and also a property of the set $\Mup$ for each range of $r$. Define the set
\begin{equation}
    \widetilde{\T} \triangleq \{ \x \in \R^n : \tilphi_1 (\x) + \tilphi_2 (\x) > \psi (\x) \},
    \label{def: set T tilde}
\end{equation}
which will be used especially in Theorem~\ref{thm: up BD} and Theorem~\ref{thm: down BD}.

\begin{theorem} 
Assume $\sigma_1 \geq \sigma_2$. Let the sets $\Mup$, $\T$, $\widetilde{\T}$, and $\B_i$ for $i \in \{ 1, 2 \}$ be defined as in \eqref{def: set M up}, \eqref{def: set T}, \eqref{def: set T tilde}, and \eqref{def: set B}, respectively. Also, let the sets $\H_i^+$ and $\H_i^-$ for $i \in \{ 1, 2 \}$ be defined as in \eqref{def: set H_i}.
\begin{enumerate}[label=(\roman*)]
    \item \label{thm: up BD part1} If $r \in \big( 0, \; \frac{L}{2 \sigma_1} \big]$ then  $\Mup \sqcup \{ \x^*_1, \x^*_2 \}$ is closed,
    \begin{displaymath}
        \partial \Mup = \T \sqcup \{ \x^*_1, \x^*_2 \}
        \quad \text{and} \quad
        (\Mup)^{\circ} = \widetilde{\T}.
    \end{displaymath}
    
    \item \label{thm: up BD part2} If $r \in \big( \frac{L}{2 \sigma_1}, \; \frac{L}{2 \sigma_2} \big]$ then $\Mup \sqcup \{ \x^*_1 \}$ is closed,
    \begin{equation*}
        \begin{aligned}
        \partial \Mup 
        &= \big[ \partial \B_1 \cap (\H_1^-)^c \big]
        \sqcup \T \sqcup \{ \x^*_1 \} 
        \quad \text{and} \\
        (\Mup)^{\circ}
        &= \big[ \B_1 \cap (\H_1^-)^c \big] \sqcup \big[ \widetilde{\T} \cap \H_1^- \big].
        \end{aligned}
    \end{equation*}
    
    \item \label{thm: up BD part3} If $r \in \Big( \frac{L}{2 \sigma_2}, \; \frac{L}{2} \big( \frac{1}{\sigma_1} + \frac{1}{\sigma_2} \big) \Big)$ then $\Mup$ is closed,
    \begin{equation*}
        \begin{aligned}
        \partial \Mup 
        &= \big[ \partial \B_1 \cap (\H_1^-)^c \big]
        \sqcup \big[ \partial \B_2 \cap (\H_2^+)^c \big]
        \sqcup \T
        \quad \text{and} \\
        (\Mup)^{\circ}
        &= \big[ \B_1 \cap (\H_1^-)^c \big] 
        \sqcup \big[ \B_2 \cap (\H_2^+)^c \big] 
        \sqcup \big[ \widetilde{\T} \cap ( \H_1^- \cap \H_2^+ ) \big].
        \end{aligned}
    \end{equation*}
    
    \item \label{thm: up BD part4} If $r = \frac{L}{2} \big( \frac{1}{\sigma_1} + \frac{1}{\sigma_2} \big)$ then $\Mup = \Big\{ \Big( \frac{L}{2} \big(\frac{1}{\sigma_1} - \frac{1}{\sigma_2} \big), \; \mathbf{0} \Big) \Big\}$.
    \item \label{thm: up BD part5} If $r \in \Big( \frac{L}{2} \big( \frac{1}{\sigma_1} + \frac{1}{\sigma_2} \big), \; \infty \Big)$ then $\Mup = \emptyset$.
\end{enumerate}
\label{thm: up BD}
\end{theorem}

\begin{proof}
For convenience, we define function $\varphi: ( \overline{\B}_1 \cap \overline{\B}_2 ) \setminus \{ \x_1^*, \x_2^* \} \to [-\pi, \pi]$ to be such that
\begin{equation}
    \varphi (\x) \triangleq \tilphi_1 (\x) + \tilphi_2 (\x) - \psi (\x),
    \label{def: varphi}
\end{equation}
where $\tilphi_i$ for $i \in \{1 ,2\}$ and $\psi$ are defined in \eqref{def: phi tilde} and $\eqref{def: psi}$, respectively.

\textbf{Part \ref{thm: up BD part1}:} $r \in \big( 0, \; \frac{L}{2 \sigma_1} \big]$. First, we want to show that 
\begin{equation}
    \text{if} \quad
    \x \in \partial (\B_1 \cap \B_2)
    \quad \text{then} \quad
    \x \in \{ \z \in \R^n :
    \varphi (\z) < 0 \} 
    \cup \{ \x_1^*, \x_2^* \}.
    \label{eqn: part 1 claim 1}
\end{equation}
Suppose $\x \in \partial (\B_1 \cap \B_2)$ and $\x \in \partial \B_1$. Since $\overline{\B}_1$ is closed and $x_1 \in \big[ -r- \frac{L}{\sigma_1}, -r+ \frac{L}{\sigma_1} \big]$, from Lemma~\ref{lem: compare x1} part~\ref{lem: compare x1 part1}, we get $x_1 \leq \frac{L}{\sigma_1}-r \leq \lambda_1$. 
If $x_1 < \lambda_1$, from Lemma~\ref{lem: point on R_n} part~\ref{lem: lambda_1}, we obtain $\varphi (\x) < 0$. On the other hand, if $x_1 = \lambda_1$ (i.e., $\frac{L}{\sigma_1}-r = \lambda_1$), from Lemma~\ref{lem: compare x1} part~\ref{lem: compare x1 part1}, we get $r = \frac{L}{2 \sigma_1}$.
Substituting into $x_1 = \frac{L}{\sigma_1} - r$, we obtain that $x_1 = \frac{L}{2 \sigma_1} = r$. Since $\x \in \partial \B (\x_1^*, 2r )$ and $x_1 = r$, we conclude that $\x = \x_2^* = (r, \mathbf{0})$. 

From the assumption $\sigma_1 \geq \sigma_2$ and the inequality $r \leq \frac{L}{2 \sigma_1}$, we get $r \leq \frac{L}{2 \sigma_2}$. We can similarly show that if $\x \in \partial (\B_1 \cap \B_2)$ and $\x \in \partial \B_2$ then either $\varphi (\x) < 0$ or $\x = \x_1^* = (-r, \mathbf{0})$ by using Lemma~\ref{lem: compare x1} part~\ref{lem: compare x1 part3} and Lemma~\ref{lem: point on R_n} part~\ref{lem: lambda_2}. Since $\partial (\B_1 \cap \B_2) \subseteq \partial \B_1 \cup \partial \B_2$, we have proved our claim.

Since $\Mup \subseteq \overline{\B}_1 \cap \overline{\B}_2$ from the definition of $\Mup$ in \eqref{def: set M up} and $\partial ( \B_1 \cap \B_2 ) \subset (\Mup)^{c}$ from \eqref{eqn: part 1 claim 1}, we have $\Mup \subseteq \B_1 \cap \B_2$. Recall the definition of $\varphi$ in \eqref{def: varphi}. Let $\mathcal{R} = (\B_1 \cap \B_2) \setminus \{ \x_1^*, \x_2^* \}$. We then partition the set $\mathcal{R}$ into 3 parts as follows:
\begin{equation*}
    \mathcal{R}_1 = \big\{ \z \in \mathcal{R}: \varphi(\z) > 0 \big\}, \;\;
    \mathcal{R}_2 = \big\{ \z \in \mathcal{R} : \varphi(\z) < 0 \big\},  
    \quad \text{and} \quad
    \mathcal{R}_3 = \big\{ \z \in \mathcal{R} : \varphi(\z) = 0 \big\} = \T,
\end{equation*}
where the last equality comes from Proposition~\ref{prop: T_n}. We will show that 
\begin{equation}
    \begin{cases}
    \mathcal{R}_1 \subset ( \partial \Mup )^c, \\
    \mathcal{R}_2 \subset ( \partial \Mup )^c, \\
    \mathcal{R}_3 \subseteq \partial \Mup, \\
    \partial ( \B_1 \cap \B_2 ) \setminus \{ \x_1^*, \x_2^* \} \subset ( \partial \Mup )^c, \\
    ( \dom (\varphi) )^c \setminus \{ \x_1^*, \x_2^* \} \subset ( \partial \Mup )^c.
    \end{cases}
    \label{eqn: R inclusion 1}
\end{equation}

Suppose $\x \in \mathcal{R}_1$. Since $\varphi$ is continuous, there exists $\epsilon > 0$ such that for all $\x_0 \in \B(\x, \epsilon)$, we have $\x_0 \in \mathcal{R}_1$ and $\varphi(\x_0) > 0$.
Since $\mathcal{R}_1 \subseteq \Mup$ and is open, we have $\mathcal{R}_1 \subseteq (\Mup)^{\circ}$. Similarly, we have $\mathcal{R}_2 \subseteq ( \mathcal{R} \setminus \Mup )^{\circ}$. 
Suppose $\x \in \mathcal{R}_3 = \T$. Using Lemma~\ref{lem: boundary}, we have that  $\mathcal{R}_3 \subseteq \partial \Mup$. Since $(\Mup)^{\circ}$, $( \mathcal{R} \setminus \Mup )^{\circ}$, and $\partial \Mup$ are disjoint, we conclude that $\mathcal{R}_1 \subset (\partial \Mup)^c$ and $\mathcal{R}_2 \subset (\partial \Mup)^c$.

Consider $\x \in \partial (\B_1 \cap \B_2) \setminus \{ \x_1^*, \x_2^* \}$. From \eqref{eqn: part 1 claim 1}, we have $\x \in \{ \z \in \R^n: \varphi (\z) < 0 \}$.
Since $\textbf{dom} (\varphi) = (\overline{\B}_1 \cap \overline{\B}_2) \setminus \{ \x_1^*, \x_2^* \}$ and $\varphi$ is continuous, there exists $\epsilon > 0$ such that for all $\x_0 \in \B(\x, \epsilon) \cap \textbf{dom} (\varphi)$, we have $\varphi(\x_0) < 0$. 
Thus, $\partial (\B_1 \cap \B_2) \setminus \{ \x_1^*, \x_2^* \} \subset ((\Mup)^c)^{\circ}$ which implies that $\partial (\B_1 \cap \B_2) \setminus \{ \x_1^*, \x_2^* \} \subset (\partial \Mup)^c$.
In addition, we have $(\textbf{dom} (\varphi))^c \setminus \{ \x_1^*, \x_2^* \} \subseteq ((\Mup)^c)^{\circ}$ (since $(\textbf{dom} (\varphi))^c \setminus \{ \x_1^*, \x_2^* \} \subseteq (\Mup)^c$ and is open) which implies $(\textbf{dom} (\varphi))^c \setminus \{ \x_1^*, \x_2^* \} \subset (\partial \Mup)^c$. Therefore, we have proved the claim \eqref{eqn: R inclusion 1}.

Since  we can partition $\R^n$ into $\mathcal{R}$, $\partial (\B_1 \cap \B_2) \setminus \{ \x_1^*, \x_2^* \}$, $(\textbf{dom} (\varphi))^c \setminus \{ \x_1^*, \x_2^* \}$ and $\{ \x_1^*, \x_2^* \}$, using \eqref{eqn: R inclusion 1}, we obtain that $\partial \Mup \subseteq \mathcal{R}_3 \sqcup \{ \x_1^*, \x_2^* \}$. 
However, we know that $\mathcal{R}_3 \subseteq \partial \Mup$ from the above analysis and $\{ \x_1^*, \x_2^* \} \subseteq \partial \Mup$ from Lemma~\ref{lem: 1D analyze}. Thus, we have $\partial \Mup = \mathcal{R}_3 \sqcup \{ \x_1^*, \x_2^* \} = \T \sqcup \{ \x_1^*, \x_2^* \}$ by Proposition~\ref{prop: T_n}.

From Proposition~\ref{prop: T_n}, we have $\T = \{ \z \in \R^n: \varphi (\z) = 0 \}$. Using the definition of $\Mup$ in \eqref{def: set M up} and $\partial \Mup = \T \sqcup \{ \x_1^*, \x_2^* \}$, we can write $(\Mup)^{\circ} = \Mup \setminus \partial \Mup = \widetilde{\T}$ where $\widetilde{\T}$ is defined in \eqref{def: set T tilde}.
Since $\T \subseteq \Mup$, this implies that $\partial \Mup = \T \sqcup \{ \x^*_1, \x^*_2 \} \subseteq \Mup \sqcup \{ \x^*_1, \x^*_2 \}$. Thus, the set $\Mup \sqcup \{ \x^*_1, \x^*_2 \}$ is closed.

\textbf{Part \ref{thm: up BD part2}:} $r \in \big( \frac{L}{2 \sigma_1}, \frac{L}{2 \sigma_2} \big]$. We separately consider three disjoint regions: $(\H_1^-)^c$, $\H_1^+ \cap \H_1^-$, $(\H_1^+)^c$. For the first region, we want to show that
\begin{equation}
    \T \cap (\H_1^-)^c = \emptyset
    \quad \text{and} \quad 
    \partial \Mup \cap (\H_1^-)^c = \partial \B_1 \cap (\H_1^-)^c.
    \label{eqn: nc thm part2.1}
\end{equation}
From Lemma~\ref{lem: boundary inclusion} part~\ref{lem: boundary inclusion part1}, we have $\partial \B_1 \cap (\H_1^-)^c \subseteq \partial \B_1 \cap \H_1^+ \subseteq \textbf{dom}(\varphi)$. Consider $\x \in \partial \B_1 \cap ( \H_1^- )^c$ and note that $x_1 \in \big( \lambda_1, \; -r + \frac{L}{\sigma_1} \big]$. From Lemma~\ref{lem: point on R_n} part~\ref{lem: lambda_1}, we have $\varphi(\x) > 0$.
Since $\big[ \lambda_1, \; -r + \frac{L}{\sigma_1} \big] \subseteq (-r, r)$ from Lemma~\ref{lem: boundary inclusion} part~\ref{lem: boundary inclusion part1}, we can apply Lemma~\ref{lem: point below} to get that for all $\x \in \overline{\B}_1 \cap (\H_1^-)^c$, we have $\varphi(\x) > 0$. This implies that  
\begin{equation}
    \overline{\B}_1 \cap (\H_1^-)^c \subseteq \T^c
    \quad \text{and} \quad
    \overline{\B}_1 \cap (\H_1^-)^c
    \subseteq \Mup \cap (\H_1^-)^c,
    \label{eqn: B1 inclusion}
\end{equation}
by Proposition~\ref{prop: T_n} and the definition of $\Mup$ in \eqref{def: set M up}, respectively. Since $\T \subseteq \textbf{dom} (\varphi)$, we have $( \overline{\B}_1 )^c \cap ( \H_1^- )^c \subseteq \T^c$. Using this inclusion and the first inclusion in \eqref{eqn: B1 inclusion}, we can write
\begin{equation}
    \emptyset = \big[ \T \cap \big( \overline{\B}_1 \cap (\H_1^-)^c \big) \big] \cup \big[ \T \cap \big( ( \overline{\B}_1 )^c \cap ( \H_1^- )^c \big) \big]
    = \T \cap ( \H_1^- )^c,
    \label{eqn: T empty}
\end{equation}
which completes the first part of claim \eqref{eqn: nc thm part2.1}.
Next, note that since $-r + \frac{L}{\sigma_1} < r$, we have  $\x_2^* \in \H_1^+$.
However, we have $\Mup \subseteq \overline{\B}_1$ from the definition of $\Mup$ in \eqref{def: set M up}. Combine this argument with the second inclusion in \eqref{eqn: B1 inclusion} yields
\begin{equation}
    \Mup \cap (\H_1^-)^c 
    = \overline{\B}_1 \cap (\H_1^-)^c.
    \label{eqn: M cap H}
\end{equation}
Since $(\H_1^-)^c$ is open, \eqref{eqn: M cap H} implies that $( \Mup )^\circ \cap (\H_1^-)^c = \B_1 \cap (\H_1^-)^c$. Then, subtracting this equation from \eqref{eqn: M cap H}, we obtain that $\partial \Mup \cap (\H_1^-)^c = \partial \B_1 \cap (\H_1^-)^c$, which completes the second part of claim \eqref{eqn: nc thm part2.1}.

Next, consider the second region $\H_1^+ \cap \H_1^- = \{ \z \in \R^n : z_1 = \lambda_1 \}$. Recall the definition of $\nu_1$ in \eqref{var: nu}. Consider the following three cases.
\begin{itemize}
    \item Suppose $\x \in \{ \z \in \R^n : z_1 = \lambda_1, \; \| \tilde{\mathbf{z}} \| > \nu_1 \}$. Then, $\x \in (\textbf{dom} (\varphi))^c \setminus \{ \x_1^*, \x_2^* \}$ which implies that $\x \notin \T$. Since $(\textbf{dom} (\varphi))^c \setminus \{ \x_1^*, \x_2^* \} \subseteq (\Mup)^c$ and is open, we also have $\x \notin \partial \Mup$.
    
    \item Suppose $\x \in \{ \z \in \R^n : z_1 = \lambda_1, \; \| \tilde{\mathbf{z}} \| = \nu_1 \}$. From Lemma~\ref{lem: intersection} part~\ref{lem: intersection 2}, we have $\x \in \T$. Using Lemma~\ref{lem: boundary}, we obtain that $\x \in \partial \Mup$.
    
    \item Suppose $\x \in \{ \z \in \R^n : z_1 = \lambda_1, \; \| \tilde{\mathbf{z}} \| < \nu_1 \}$. Since $\{ \z \in \R^n : z_1 = \lambda_1, \; \| \tilde{\mathbf{z}} \| = \nu_1 \} \subseteq \T$, from Lemma~\ref{lem: point below}, we have $\tilphi_1 (\x) + \tilphi_2 (\x) > \psi (\x)$ which implies that $\x \notin \T$. Since $\x \in (\B_1 \cap \B_2) \setminus \{ \x_1^*, \x_2^* \}$ and $\varphi$ is continuous, there exists $\epsilon \in \R_{>0}$ such that for all $\x_0 \in \B (\x, \epsilon)$, we have $\x_0 \in \Mup$ by the definition of $\Mup$ in \eqref{def: set M up}. This means that $\x \in (\Mup)^{\circ}$ and thus, $\x \notin \partial \Mup$. 
\end{itemize}
Combining the analysis of these three cases, we have that 
\begin{equation}
    \partial \Mup \cap ( \H_1^+ \cap \H_1^-) 
    = \{ \z \in \R^n : z_1 = \lambda_1, \; \| \tilde{\mathbf{z}} \| = \nu_1 \}
    = \T \cap ( \H_1^+ \cap \H_1^-).
    \label{eqn: nc thm part2.2}
\end{equation}

Next, consider the third region $(\H_1^+)^c = \R^n \setminus \H_1^+$. First, we want to show that 
\begin{equation}
    \text{if} \quad
    \x \in \partial (\B_1 \cap \B_2) \cap (\H_1^+)^c
    \quad \text{then} \quad
    \x \in \{ \z \in \R^n :
    \varphi (\z) < 0 \} \cup \{ \x_1^* \}.
    \label{eqn: part 2 claim 2}
\end{equation}
Suppose $\x \in \partial (\B_1 \cap \B_2)$ and $\x \in \partial \B_1 \cap (\H_1^+)^c$.
Since $x_1 < \lambda_1$, from Lemma~\ref{lem: point on R_n} part~\ref{lem: lambda_1}, we obtain $\varphi (\x) < 0$. By using \eqref{eqn: part 1 claim 1} from part~\ref{thm: up BD part1}, we have that if $\x \in \partial (\B_1 \cap \B_2)$ and $\x \in \partial \B_2$  then either $\varphi (\x) < 0$ or $\x = \x_1^*$. Combining the two results, we have proved the claim.

Since $\Mup \subseteq \overline{\B}_1 \cap \overline{\B}_2$ from the definition of $\Mup$ in \eqref{def: set M up} and $\partial (\B_1 \cap \B_2) \cap (\H_1^+)^c \subset (\Mup)^{c}$ from \eqref{eqn: part 2 claim 2}, we have $\Mup \cap (\H_1^+)^c \subseteq (\B_1 \cap \B_2) \cap (\H_1^+)^c$. Let $\mathcal{R}' = \mathcal{R} \cap (\H_1^+)^c$. We then partition the set $\mathcal{R}'$ into $\mathcal{R}_1'$, $\mathcal{R}_2'$, and $\mathcal{R}_3'$ where $\mathcal{R}_i' = \mathcal{R}_i \cap (\H_1^+)^c$ for $i \in \{1, 2, 3 \}$. We can use a similar argument as in the proof of \eqref{eqn: R inclusion 1} to show that
\begin{equation}
    \begin{cases}
    \mathcal{R}_1' \subset (\partial \Mup)^c \cap (\H_1^+)^c, \\
    \mathcal{R}_2' \subset (\partial \Mup)^c \cap (\H_1^+)^c, \\
    \mathcal{R}_3' \subseteq \partial \Mup \cap (\H_1^+)^c, \\
    ( \partial (\B_1 \cap \B_2) \setminus \{ \x_1^* \} ) \cap (\H_1^+)^c \subset (\partial \Mup)^c \cap (\H_1^+)^c, \\
    \big( (\dom (\varphi))^c \setminus \{ \x_1^* \} \big) \cap (\H_1^+)^c \subset (\partial \Mup)^c \cap (\H_1^+)^c.
    \end{cases}
    \label{eqn: R inclusion 2}
\end{equation}
Since  we can partition $(\H_1^+)^c$ into $\mathcal{R}'$, $( \partial (\B_1 \cap \B_2) \setminus \{ \x_1^* \} ) \cap (\H_1^+)^c$, $\big( (\textbf{dom} (\varphi))^c \setminus \{ \x_1^* \} \big) \cap (\H_1^+)^c$ and $\{ \x_1^* \}$, using \eqref{eqn: R inclusion 2}, we obtain that $\partial \Mup \cap (\H_1^+)^c \subseteq \mathcal{R}_3' \sqcup \{ \x_1^* \}$. 
However, we know that $\mathcal{R}_3' \subseteq \partial \Mup \cap (\H_1^+)^c$ from \eqref{eqn: R inclusion 2} and $\{ \x_1^* \} \subseteq \partial \Mup \cap (\H_1^+)^c$ from Lemma~\ref{lem: 1D analyze}. Thus, using   
$\mathcal{R}_3' = \mathcal{R}_3 \cap ( \H_1^+ )^c$ and Proposition~\ref{prop: T_n} we have 
\begin{equation}
    \partial \Mup \cap (\H_1^+)^c 
    = \mathcal{R}_3' \sqcup \{ \x_1^* \} 
    = \big[ \T \cap ( \H_1^+ )^c \big] \sqcup \{ \x^*_1 \}.
    \label{eqn: nc thm part2.3}
\end{equation}

Since $\R^n = (\H_1^-)^c \sqcup ( \H_1^+ \cap \H_1^- ) \sqcup (\H_1^+)^c$, using \eqref{eqn: nc thm part2.1}, \eqref{eqn: nc thm part2.2} and \eqref{eqn: nc thm part2.3}, we obtain that
\begin{equation}
    \partial \Mup 
    = \big[ \partial \B_1 \cap (\H_1^-)^c \big]
    \sqcup \big[ \T \cap ( \H_1^+ \cap \H_1^-) \big]
    \sqcup \big[ \T \cap ( \H_1^+ )^c \big] 
    \sqcup \{ \x^*_1 \}.
    \label{eqn: up bd temp}
\end{equation}
However, from \eqref{eqn: T empty}, we can write $\T = \big[ \T \cap ( \H_1^+ \cap \H_1^-) \big] \sqcup \big[ \T \cap ( \H_1^+ )^c \big]$ which means that we can rewrite \eqref{eqn: up bd temp} as 
\begin{equation}
    \partial \Mup 
    = \big[ \partial \B_1 \cap (\H_1^-)^c \big]
    \sqcup \T \sqcup \{ \x^*_1 \}.
    \label{eqn: up bd case2}
\end{equation}

From \eqref{eqn: M cap H} and \eqref{eqn: up bd case2}, we can write $\big[ \Mup \cap (\H_1^-)^c \big] \setminus \big[ \partial \Mup \cap (\H_1^-)^c \big] = \big[ \overline{\B}_1 \cap (\H_1^-)^c \big] \setminus \big[ \partial \B_1 \cap (\H_1^-)^c \big] = \B_1 \cap (\H_1^-)^c$.
From the definition of $\Mup$ in \eqref{def: set M up} and equation \eqref{eqn: nc thm part2.2}, we can write $\big[ \Mup \cap ( \H_1^+ \cap \H_1^-) \big] \setminus \big[ \partial \Mup \cap ( \H_1^+ \cap \H_1^-) \big] = \widetilde{\T} \cap ( \H_1^+ \cap \H_1^-)$.
From the definition of $\Mup$ in \eqref{def: set M up} and equation \eqref{eqn: nc thm part2.3}, we can write $( \Mup \cap (\H_1^+)^c ) \setminus ( \partial \Mup \cap (\H_1^+)^c ) = \widetilde{\T} \cap (\H_1^+)^c$. Applying the above three equations to Lemma~\ref{lem: interior eqn}, we obtain the result of $( \Mup )^{\circ}$.
Consider the characterization of $\partial \Mup$ in \eqref{eqn: up bd case2}. Since $\partial \B_1 \cap ( \H_1^- )^c \subseteq \Mup$ from \eqref{eqn: M cap H}, and $\T \subseteq \Mup$ from Proposition~\ref{prop: T_n} and the definition of $\Mup$, we can write $\partial \Mup \subseteq \Mup \sqcup \{ \x^*_1 \}$ and thus, $\Mup \sqcup \{ \x^*_1 \}$ is closed.

\textbf{Part \ref{thm: up BD part3}:} $r \in \Big( \frac{L}{2 \sigma_2}, \; \frac{L}{2} \big( \frac{1}{\sigma_1} + \frac{1}{\sigma_2} \big) \Big)$. We can use a similar argument as in the proof of part~\ref{thm: up BD part2} to show that
\begin{equation*}
    \begin{cases}
    \partial \Mup \cap (\H_1^-)^c = \partial \B_1 \cap (\H_1^-)^c \subseteq \Mup,
    \quad \text{(similar to proving \eqref{eqn: nc thm part2.1})} \\
    \partial \Mup \cap (\H_2^+)^c = \partial \B_2 \cap (\H_2^+)^c \subseteq \Mup, 
    \quad \text{(similar to proving \eqref{eqn: nc thm part2.1})} \\
    \partial \Mup \cap ( \H_1^+ \cap \H_1^-) = \T \cap ( \H_1^+ \cap \H_1^-) \subseteq \Mup, 
    \quad \text{(similar to proving \eqref{eqn: nc thm part2.2})} \\
    \partial \Mup \cap ( \H_2^+ \cap \H_2^-) = \T \cap ( \H_2^+ \cap \H_2^-) \subseteq \Mup, 
    \quad \text{(similar to proving \eqref{eqn: nc thm part2.2})} \\
    \partial \Mup \cap ( \H_1^+ \cup \H_2^- )^c = \T \cap ( \H_1^+ \cup \H_2^- )^c \subseteq \Mup.
    \quad \text{(similar to proving \eqref{eqn: nc thm part2.3})} \\
    \end{cases}
\end{equation*}
Similar to \eqref{eqn: T empty}, in this case, we have that $\T \cap ( \H_1^- )^c = \emptyset$ and $\T \cap ( \H_2^+ )^c = \emptyset$. This means that the last three equations regarding $\partial \Mup$ above can be combined into $\partial \Mup \cap (\H_1^- \cap \H_2^+) = \T$. 
Combining this equation with the first two equations regarding $\partial \Mup$ above, we obtain the characterization of $\partial \Mup$. For the characterization of $( \Mup )^{\circ}$, we can use the same technique as shown in the analysis of part~\ref{thm: up BD part2} to obtain the result. From the five inclusions regarding $\partial \Mup$ above, we can write $\partial \Mup \subseteq \Mup$ and we conclude that $\Mup$ is closed.

\textbf{Part \ref{thm: up BD part4}:} $r = \frac{L}{2} \big( \frac{1}{\sigma_1} + \frac{1}{\sigma_2} \big)$. In this case, we have $\overline{\B}_1 \cap  \overline{\B}_2  = \Big\{ \Big( \frac{L}{2} \big(\frac{1}{\sigma_1} - \frac{1}{\sigma_2} \big), \; \mathbf{0} \Big) \Big\}$. 
Suppose $\x = \Big( \frac{L}{2} \big(\frac{1}{\sigma_1} - \frac{1}{\sigma_2} \big), \; \mathbf{0} \Big)$.
Since $\Mup \subseteq \overline{\B}_1 \cap  \overline{\B}_2$, we only need to check point $\x$.
At this point, we get $\tilphi_1 (\x) + \tilphi_2 (\x) = \psi(\x) = 0$ (since $d_1(\x) = \frac{L}{\sigma_1}$, $d_2(\x) = \frac{L}{\sigma_2}$, $\alpha_1 (\x) = 0$ and $\alpha_2 (\x) = \pi$). 
So, we conclude that $\Mup = \Big\{ \Big( \frac{L}{2} \big(\frac{1}{\sigma_1} - \frac{1}{\sigma_2} \big), \; \mathbf{0} \Big) \Big\}$.

\textbf{Part \ref{thm: up BD part5}:} $r \in \Big( \frac{L}{2} \big( \frac{1}{\sigma_1} + \frac{1}{\sigma_2} \big), \; \infty \Big)$. Since $r > \frac{L}{2} \big(\frac{1}{\sigma_1} + \frac{1}{\sigma_2} \big)$, we have $\overline{\B}_1 \cap  \overline{\B}_2 = \emptyset$. Since $\Mup \subseteq \overline{\B}_1 \cap  \overline{\B}_2$, we conclude that $\Mup = \emptyset$.
\end{proof}

Examples of the boundary $\partial \Mup$ in $\R^2$ for the first three cases of Theorem~\ref{thm: up BD} are shown in Fig.~\ref{fig: outer boundary}. We consider parameters $\sigma_1 = 1.5$, $\sigma_2 = 1$ and $L = 10$. 
For $r = 2$, as we can see from Fig.~\ref{fig: boundary NC Case1}, we have $\partial \Mup = \T \sqcup \{ \x_1^*, \x_2^* \}$ (i.e., solid blue line + two red dots) consistent with part~\ref{thm: up BD part1} of Theorem~\ref{thm: up BD}. 
For $r = 4$, as we can see from Fig.~\ref{fig: boundary NC Case2}, we have $\partial \Mup = \big[ \partial \B_1 \cap (\H_1^-)^c \big] \sqcup \T \sqcup \{ \x^*_1 \}$ (i.e., solid cyan line + solid blue line + left red dot) consistent with part~\ref{thm: up BD part2} of Theorem~\ref{thm: up BD}.
For $r = 6$, as we can see from Fig.~\ref{fig: boundary NC Case3}, we have $\partial \Mup = \big[ \partial \B_1 \cap (\H_1^-)^c \big] \sqcup \big[ \partial \B_2 \cap (\H_2^+)^c \big] \sqcup \T$ (i.e., solid cyan line + solid magenta line + solid blue line) consistent with part~\ref{thm: up BD part3} of Theorem~\ref{thm: up BD}.
Note that the solid blue line, solid cyan line and solid magenta line in the figures indicate that the corresponding sets of points are subsets of the outer approximation $\Mup$, i.e., $\T \subseteq \Mup$, $\partial \B_1 \cap ( \H_1^- )^c \subseteq \Mup$ and $\partial \B_2 \cap ( \H_2^+ )^c \subseteq \Mup$, respectively.

\begin{figure}
\centering
\subfloat[For $r=2$, the characterization of boundary $\partial \Mup$ corresponds to Theorem~\ref{thm: up BD} part~\ref{thm: up BD part1}.]{\includegraphics[width=.30\textwidth]{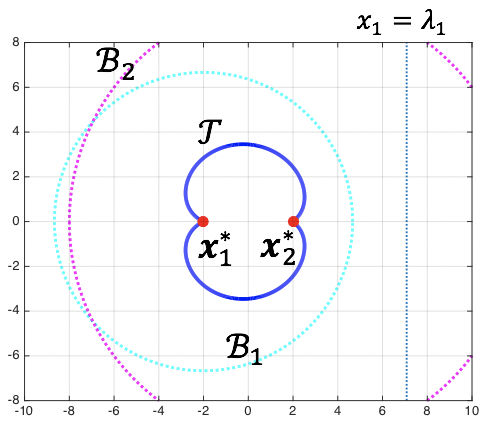} \label{fig: boundary NC Case1}}\;\;
\subfloat[For $r=4$, the characterization of boundary $\partial \Mup$ corresponds to Theorem~\ref{thm: up BD} part~\ref{thm: up BD part2}.]{\includegraphics[width=.30\textwidth]{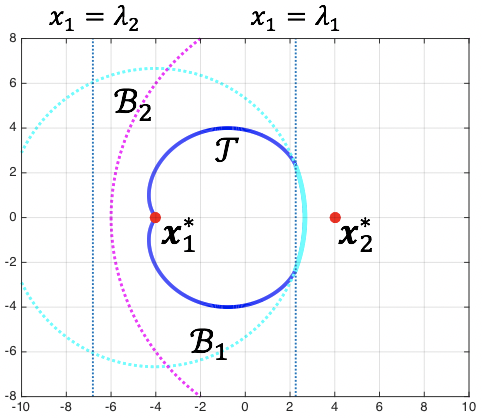} \label{fig: boundary NC Case2}}\;\;
\subfloat[For $r=6$, the characterization of boundary $\partial \Mup$ corresponds to Theorem~\ref{thm: up BD} part~\ref{thm: up BD part3}.]{\includegraphics[width=.30\textwidth]{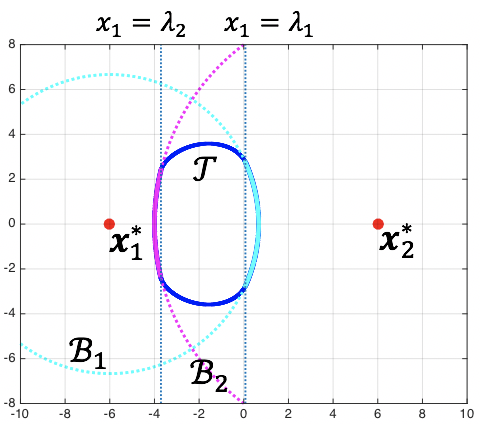} \label{fig: boundary NC Case3}}
\caption{The boundary $\partial \Mup$ in $\R^2$ with different values of $r$ for two original minimizers $x_1^*=(-r,0)$ and $x_2^*=(r,0)$ is plotted given fixed parameters $\sigma_1 = 1.5$, $\sigma_2 = 1$, and $L = 10$. The sets $\T$, $\partial \B_1 \cap (\H_1^-)^c$ and $\partial \B_2 \cap (\H_2^+)^c$ are shown by solid blue, cyan and magenta lines, respectively. The vertical dotted lines represent the equations $x_1 = \lambda_1$ and $x_1 = \lambda_2$ and note that the value of $\lambda_1$ and $\lambda_2$ depends on $r$.} 
\label{fig: outer boundary}
\end{figure}

%% file: contents/sec-inner.tex
\section{Inner approximation and potential solution region} \label{sec: inner}

In the previous section, we showed that the set $\Mup$ defined in \eqref{def: set M up} is an outer approximation for the desired set $\M$ defined in \eqref{def: set M}, in that $\M \subseteq \Mup$. We now turn our attention to the set $\Mdo$ defined in \eqref{def: set M down}. We will show that $\Mdo \subseteq \M$, and consequently, provide a tight characterization of $\M$.

Since we have $\bigcup_{\Tilde{\sigma} \geq \sigma} \Q(\x^*, \Tilde{\sigma}) \subset \S(\x^*, \sigma)$ for all $\x^* \in \R^n$ and $\sigma \in \R_{>0}$ from \eqref{eqn: set of functions inclusion}, we can provide a region contained in the potential solution region $\M$ by restricting our consideration to only some classes of quadratic functions. In Section~\ref{subsec: Quadratic}, we analyze a sufficient and necessary condition for constructing a quadratic function with a given minimizer, gradient and curvature. Then, using the result from Section~\ref{subsec: Quadratic}, in Section~\ref{subsec: characterization}, we prove a relationship between the potential solution region $\M$ and the inner approximation $\Mdo$, and also provide a characterization of $\Mdo$.

\subsection{Quadratic functions analysis} \label{subsec: Quadratic}

In this subsection, first we consider an equivalent condition for the existence of a quadratic function with a given minimizer, gradient at a specific point and the smallest eigenvalue associated to the quadratic term in $n$-dimensional space (i.e., in $\R^n$) which is presented in Proposition~\ref{prop: function exist n-dim}. Then, we present Corollary~\ref{cor: function exist union} in which we provide a similar equivalent condition for the class $\bigcup_{\Tilde{\sigma} \geq \sigma} \Q(\x^*, \Tilde{\sigma})$ for a given $\x^* \in \R^n$ and $\sigma \in \R_{>0}$.

In the following proposition (whose proof is provided in Appendix~\ref{subsec: proof quad}), we consider an equivalent condition for the existence of a quadratic function with more than one independent variable satisfying certain properties.

\begin{proposition}
Let $\Q$ be defined as in \eqref{def: Quadratic Coll}.
For $n \in \mathbb{N} \setminus \{1 \}$, suppose we are given points $\x^* \in \R^n$ and $\x_0 \in \R^n$ such that $\x_0 \neq \x^*$, vector $\g \in \R^n$, and scalar $\sigma \in \R_{>0}$.  Then, there exists a function $f \in  \Q^{(n)}(\x^*, \sigma)$ with a gradient $\nabla f(\x_0) = \g$ if and only if 
\begin{enumerate}[label=(\roman*)]
    \item $\x_0 \in \overline{\B} \big( \x^*, \frac{\| \g \|}{\sigma} \big)$ and
    \item $\angle ( \g, \x_0 - \x^*) \in \{ 0 \} \cup \Big[0, \; \arccos \big( \frac{\sigma}{\Vert \g \Vert} \Vert \x_0 - \x^* \Vert \big) \Big)$.
\end{enumerate}
Note that if $\sigma \Vert \x_0 - \x^* \Vert = \Vert \g \Vert$, then $\Big[0, \; \arccos (\frac{\sigma}{\Vert \g \Vert} \Vert \x_0 - \x^* \Vert) \Big) = \emptyset$.
\label{prop: function exist n-dim}
\end{proposition}

Recall from \eqref{eqn: set of functions inclusion} that $\bigcup_{\hat{\sigma} \geq \sigma} \Q(\x^*, \hat{\sigma}) \subset \S(\x^*, \sigma)$ for all $\x^* \in \R^n$ and $\sigma \in \R_{>0}$. One way to characterize the inner approximation $\Mdo$ is to utilize sufficient conditions for the construction of a function $f \in \bigcup_{\hat{\sigma} \geq \sigma} \Q (\x^*, \hat{\sigma})$. More generally, the following corollary presents necessary and sufficient conditions for such construction.

\begin{corollary}
Let $\Q$ be defined as in \eqref{def: Quadratic Coll}.
For $n \in \mathbb{N} \setminus \{1 \}$, suppose we are given points $\x^* \in \R^n$ and $\x_0 \in \R^n$ such that $\x_0 \neq \x^*$, a vector $\g \in \R^n$, scalar $L \in \R_{>0}$ such that $\| \g \| = L$, and a scalar $\sigma \in \R_{>0}$.  Then, there exists a function $f \in \bigcup_{\hat{\sigma} \geq \sigma} \Q (\x^*, \hat{\sigma})$ with $\nabla f(\x_0) = \g$ if and only if 
\begin{enumerate}[label=(\roman*)]
\item $\x_0 \in \overline{\B} \big( \x^*, \frac{L}{\sigma} \big)$ and
\item $\angle ( \g, \; \x_0 - \x^*) \in \{ 0 \} \cup \Big[0, \; \arccos \big( \frac{\sigma}{L} \Vert \x_0 - \x^* \Vert \big) \Big)$.
\end{enumerate}
Note that if $\sigma \Vert \x_0 - \x^* \Vert = L$, then $\Big[0, \; \arccos (\frac{\sigma}{L} \Vert \x_0 - \x^* \Vert) \Big) = \emptyset$.
\label{cor: function exist union}
\end{corollary}

\begin{proof}
From Proposition~\ref{prop: function exist n-dim}, we can write that there exists a function $f \in \bigcup_{\hat{\sigma} \geq \sigma} \Q (\x^*, \hat{\sigma})$ with a gradient $\nabla f(\x_0) = \g$ and $\| \nabla f(\x_0) \| = L$ if and only if 
\begin{equation*}
    \x_0 \in \bigcup_{\hat{\sigma} \geq \sigma} \overline{\B} \Big( \x^*, \frac{\| \g \|}{\hat{\sigma}} \Big)
    = \overline{\B} \Big( \x^*, \frac{L}{\sigma} \Big),
\end{equation*}
and 
\begin{align*}
    \angle ( \g, \; \x_0 - \x^*) 
    &\in \{ 0 \} \cup \bigcup_{\hat{\sigma} \geq \sigma} \bigg[0, \; \arccos \Big( \frac{\hat{\sigma}}{\Vert \g \Vert} \Vert \x_0 - \x^* \Vert \Big) \bigg) \\
    &= \{ 0 \} \cup \bigg[0, \; \arccos \Big( \frac{\sigma}{L} \Vert \x_0 - \x^* \Vert \Big) \bigg).
\end{align*}
\end{proof}

Given the generality of Proposition~\ref{prop: function exist n-dim} concerning the quadratic function class, we reserve the discussion and its geometric interpretation for Section~\ref{subsec: quadratic-discussion}.

\subsection{Inner approximation characterization} \label{subsec: characterization}

In this subsection, we use results from Section~\ref{subsec: Quadratic} to derive a sufficient condition for a point to be in the potential solution region $\M$, defined in \eqref{def: set M}. In fact, the sufficient condition is encapsulated in the description of the inner approximation $\Mdo$; therefore, $\Mdo \subseteq \M$ which is presented in Proposition~\ref{prop: angle sc}. Then, in Theorem~\ref{thm: down BD}, we characterize the boundary $\partial \Mdo$ and interior $( \Mdo )^{\circ}$, and provide a property of $\Mdo$ similar to Theorem~\ref{thm: up BD}. 

Recall the definition of $\underline{L}_i$ for $i \in \{ 1, 2 \}$ from \eqref{eqn: L_lb}. Given $i \in \{ 1, 2 \}$, $\x_i^* \in \R^n$, $\sigma_i \in \R_{> 0}$, and $L \in \R_{> 0}$, from Corollary~\ref{cor: function exist union}, we define the set of gradient angles $\angle ( \nabla f_i (\x), \; \x - \x_i^*)$ that we can choose to construct a quadratic function $f_i \in \bigcup_{\hat{\sigma}_i \geq \sigma_i} \Q (\x_i^*, \hat{\sigma}_i)$ with $\| \nabla f_i (\x) \| \leq L$ as $\Phi_i : \overline{\B} \big( \x^*_i, \frac{L}{\sigma_i} \big) \to 2^{[0, \frac{\pi}{2} )}$ with
\begin{align}
    \Phi_i (\x) \triangleq \begin{cases}
        \Big[ 0, \; \arccos \big(\frac{ \underline{L}_i (\x) }{L}  \big) \Big) \quad 
        &\text{if} \quad \underline{L}_i (\x) < L, \\
        \{ 0 \} \quad 
        &\text{if} \quad \underline{L}_i (\x) = L.
    \end{cases} \label{def: Phi}
\end{align}
Notice that the supremum of the set of angles $\Phi_i (\x)$ is equal to $\tilphi_i (\x)$ which is defined in \eqref{def: phi tilde}. That is, for $i \in \{ 1,2 \}$, for all $\x \in \overline{\B} \big( \x_i^*, \frac{L}{\sigma_i} \big)$, we have
\begin{align*}
    \sup \Phi_i (\x) 
    = \arccos \Big( \frac{ \underline{L}_i (\x) }{L}  \Big)
    = \tilphi_i(\x).
\end{align*}
In two-dimensional space, for a given $\x \in \overline{\B}_1 \cap \overline{\B}_2$, for $i \in \{ 1, 2 \}$, the set of admissible angles $\Phi_i (\x)$ and the quantity $\tilphi_i (\x)$ are shown in Fig.~\ref{fig: angle and sc fig1}.
In the next proposition, using Corollary~\ref{cor: function exist union}, we show that the inner approximation $\Mdo$ is contained in the potential solution region $\M$.

\begin{proposition}
Suppose the sets $\M$ and $\Mdo$ are defined as in \eqref{def: set M} and \eqref{def: set M down}, respectively. Then, $\M \supseteq \Mdo$.
\label{prop: angle sc}
\end{proposition} 

\begin{proof}
Suppose $\x \in ( \overline{\B}_1 \cap \overline{\B}_2 ) \setminus \{ \x_1^*, \x_2^* \}$. Recall the definition of $\mathcal{X}$ from \eqref{def: set X}. We want to show that if $\tilphi_1 (\x) + \tilphi_2 (\x) > \psi (\x)$ or $\x \in \mathcal{X}$, then $\x \in \M$. Recall the definition of $\u_i (\x)$ for $i \in \{ 1, 2 \}$ from \eqref{def: unit vector}. Consider the following two cases.
\begin{itemize}
    \item Suppose $\tilphi_1 (\x) + \tilphi_2 (\x) > \psi (\x)$. Since $\psi (\x) = \angle (\u_1 (\x), - \u_2 (\x) )$ from \eqref{eqn: psi equality}, there exists a vector $\g \in \R^n$ with $\| \g \| = L$ such that $\angle (\g, \u_1 (\x) ) < \tilphi_1 (\x)$ and $\angle (\g, - \u_2 (\x) ) < \tilphi_2 (\x)$. By the definition of $\Phi_i$ in \eqref{def: Phi}, this means that $\angle (\g, \u_1 (\x) ) \in \Phi_1 (\x)$ and $\angle (- \g, \u_2 (\x) ) \in \Phi_2 (\x)$.
    
    \item Suppose $\x \in \mathcal{X}$. Then, the point $\x = \big( -r + \frac{L}{\sigma_1}, \mathbf{0} \big)$ and  $-r + \frac{L}{\sigma_1} \in (-r, r)$ as discussed below \eqref{def: set X}. In this case, we choose $\g = L \boldsymbol{e}'_1$ where $\boldsymbol{e}'_1 = \frac{\x_2^* - \x_1^*}{\| \x_2^* - \x_1^* \|}$. This implies that $\angle ( \g, \u_1 (\x)) = 0 \in \Phi_1 (\x)$ and $\angle ( - \g, \u_2 (\x)) = 0 \in \Phi_2 (\x)$ by the definition of $\Phi_i$ in \eqref{def: Phi}.
\end{itemize}
Using Corollary~\ref{cor: function exist union}, for both cases, we have that for $i \in \{ 1, 2 \}$, there exist functions $f_i \in \bigcup_{\hat{\sigma} > \sigma_i} \Q (\x_i^*, \hat{\sigma} )$  such that $\g = \nabla f_1 (\x) = - \nabla f_2 (\x)$ and $\| \nabla f_1 (\x) \| = \| \nabla f_2 (\x) \| = L$. 
Using \eqref{eqn: set of functions inclusion}, we have that there exist $f_i \in \S (\x_i^*, \sigma )$ with $\| \nabla f_i (\x) \| \leq L$ for $i \in \{ 1, 2 \}$ and $\x$ is the minimizer of $f_1 + f_2$. Therefore, $\x \in \M$.
Since $\Mdo \subseteq ( \overline{\B}_1 \cap \overline{\B}_2 ) \setminus \{ \x_1^*, \x_2^* \}$, we conclude that $\M \supseteq \Mdo$.
\end{proof}

\begin{figure}
\centering
\subfloat[]{\includegraphics[width=.45\linewidth]{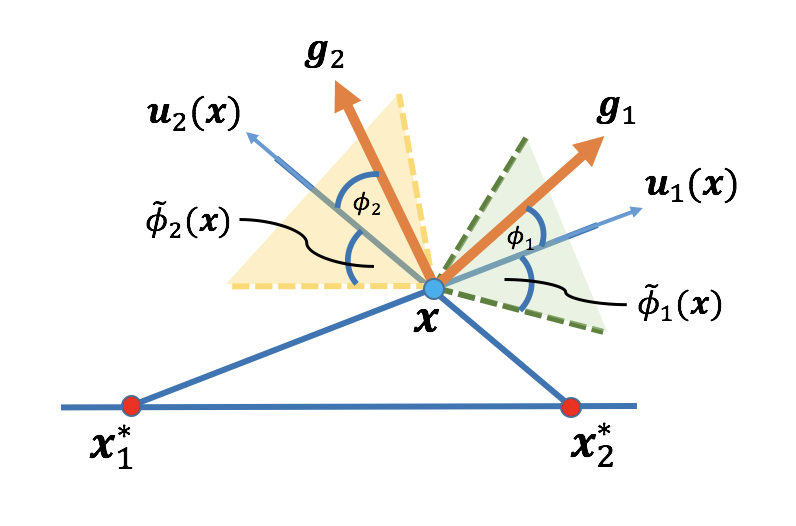} \label{fig: angle and sc fig1}}\quad
\subfloat[]{\includegraphics[width=.45\linewidth]{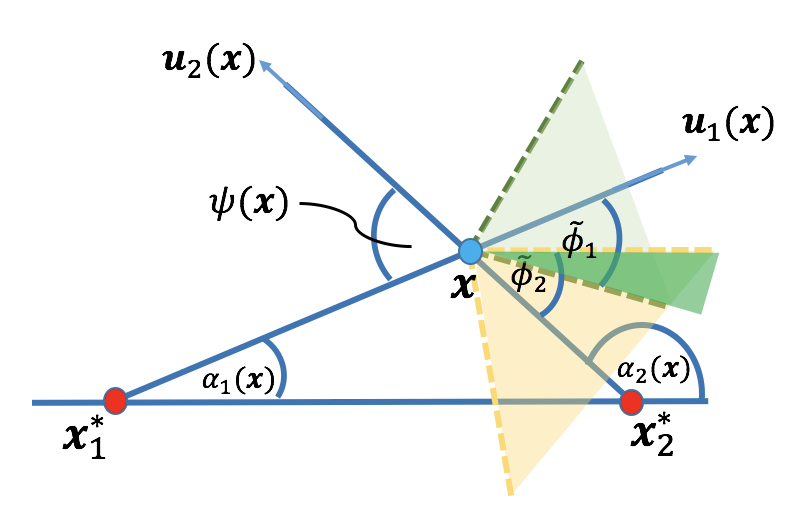} \label{fig: angle and sc fig2}}
\caption{(a) For given $\x$, $\x_1^*$, $\x_2^*$ and $\sigma$, the figure illustrates the regions where the vectors $\g_1$ and $\g_2$ with $\| \g_1 \| = \| \g_2 \| = L$ must lie, so that we can construct quadratic functions $f_i \in \bigcup_{\hat{\sigma} \geq \sigma} \Q (\x_i^*, \hat{\sigma})$ with $\nabla f_i (\x) = \g_i$ for $i \in \{ 1, 2 \}$.
Recall the definition of $\phi_i$ and $\Phi_i$ from \eqref{def: phi_i} and \eqref{def: Phi}, respectively. In particular, for $i \in \{ 1, 2 \}$, it is sufficient to have $\phi_i (\x) \in \Phi_i (\x)$ from Corollary~\ref{cor: function exist union}, i.e., pictorially, the vectors $\g_1$ and $\g_2$ must \textbf{strictly} lie in the corresponding shaded regions.
(b) The figure illustrates the inequality $\tilphi_1 (\x) + \tilphi_2 (\x) > \psi(\x)$ in the description of $\Mdo$ which means that there is an overlapping region (light green region in the figure) caused by one shaded region and the mirror of the other shaded region.}
\end{figure}

In two-dimensional space, for a given $\x \in ( \overline{\B}_1 \cap \overline{\B}_2 ) \setminus \{ \x_1^*, \x_2^* \}$, the geometrical interpretation of the inequality $\tilphi_1 (\x) + \tilphi_2 (\x) > \psi(\x)$, which is used to describe the inner approximation $\Mdo$, is represented in Fig.~\ref{fig: angle and sc fig2}.

Before characterizing the set $\Mdo$, recall the definition of $\lambda_i$ and $\nu_i$ in \eqref{var: lambda} and \eqref{var: nu}, respectively. For $i \in \{ 1, 2 \}$, we define
\begin{equation}
    \C_i \triangleq \big\{ \x \in \R^n: x_1 = \lambda_i, \;\; 
    \Vert \tilde{\mathbf{x}} \Vert = \nu_i \big\}.
    \label{def: set C_i}
\end{equation}

Comparing the definition of $\Mup$ in \eqref{def: set M up} to that of $\Mdo$ in \eqref{def: set M down}, we see that the description of $\Mdo$ involves a strict inequality while it is not for $\Mup$. Since the only difference is the inequality sign, we could expect to see similar results as in Theorem~\ref{thm: up BD}. Specifically, the following theorem provides a characterization of $\partial \Mdo$ and $( \Mdo )^\circ$ explicitly, and also a property of $\Mdo$. Since most parts of the proof are similar to that of Theorem~\ref{thm: up BD}, we defer the proof to Appendix~\ref{subsec: proof inner}.

\begin{theorem}
Assume $\sigma_1 \geq \sigma_2$. Let the sets $\Mdo$, $\T$, $\widetilde{\T}$, and $\B_i$ for $i \in \{ 1, 2 \}$ be defined as in \eqref{def: set M up}, \eqref{def: set T}, \eqref{def: set T tilde}, and \eqref{def: set B}, respectively. Also, let the sets $\H_i^+$ and $\H_i^-$ for $i \in \{ 1, 2 \}$ be defined as in \eqref{def: set H_i}, and the sets $\C_i$ for $i \in \{ 1, 2 \}$ be defined as in \eqref{def: set C_i}.
\begin{enumerate}[label=(\roman*)]
    \item \label{thm: down BD part1} If $r \in \big( 0, \; \frac{L}{2 \sigma_1} \big]$ then $\Mdo$ is open and
    \begin{displaymath}
        \partial \Mdo = \T \sqcup \{ \x_1^*, \x_2^* \}
        \quad \text{and} \quad
        (\Mdo)^{\circ} = \widetilde{\T}.
    \end{displaymath} 
    \item \label{thm: down BD part2} If $r \in \big( \frac{L}{2 \sigma_1}, \; \frac{L}{2 \sigma_2} \big]$ then $ (\Mdo \cup \C_1) \cap \H_1^+$ is closed while $\Mdo \cap (\H_1^+)^c$ is open, and
    \begin{align*}
        \partial \Mdo 
        &= \big[ \partial \B_1 \cap (\H_1^-)^c \big]
        \sqcup \T \sqcup \{ \x^*_1 \}
        \quad \text{and} \\
        (\Mdo)^{\circ}
        &= \big[ \B_1 \cap (\H_1^-)^c \big] \sqcup 
        \big[ \widetilde{\T} \cap \H_1^- \big].
    \end{align*}
    \item \label{thm: down BD part3} If $r \in \Big( \frac{L}{2 \sigma_2}, \; \frac{L}{2} \big( \frac{1}{\sigma_1} + \frac{1}{\sigma_2} \big) \Big)$ then $(\Mdo \cup \C_1) \cap \H_1^+$ and $(\Mdo \cup \C_2) \cap \H_2^-$ are closed while $\Mdo \cap ( \H_1^+ \cup \H_2^- )^c$ is open, and
    \begin{align*}
        \partial \Mdo 
        &= \big[ \partial \B_1 \cap (\H_1^-)^c \big]
        \sqcup \big[ \partial \B_2 \cap (\H_2^+)^c \big]
        \sqcup \T
        \quad \text{and} \\
        (\Mdo)^{\circ}
        &= \big[ \B_1 \cap (\H_1^-)^c \big] 
        \sqcup \big[ \B_2 \cap (\H_2^+)^c \big] 
        \sqcup \big[ \widetilde{\T} \cap ( \H_1^- \cap \H_2^+ ) \big].
    \end{align*}
    \item \label{thm: down BD part4} If $r = \frac{L}{2} \big( \frac{1}{\sigma_1} + \frac{1}{\sigma_2} \big)$ then $\Mdo = \Big\{ \Big( \frac{L}{2} \big(\frac{1}{\sigma_1} - \frac{1}{\sigma_2} \big), \; \mathbf{0} \Big) \Big\}$.
    \item \label{thm: down BD part5} If $r \in \Big( \frac{L}{2} \big( \frac{1}{\sigma_1} + \frac{1}{\sigma_2} \big), \; \infty \Big)$ then $\Mdo = \emptyset$.
\end{enumerate}
\label{thm: down BD}
\end{theorem}

Examples of the boundary $\partial \Mdo$ in $\R^2$ for the first three cases of Theorem~\ref{thm: down BD} are shown in Fig.~\ref{fig: inner boundary}. Again, we consider parameters $\sigma_1 = 1.5$, $\sigma_2 = 1$ and $L = 10$.
For $r = 2$, as we can see from Fig.~\ref{fig: boundary SC Case1}, we have $\partial \Mdo = \T \sqcup \{ \x_1^*, \x_2^* \}$ (i.e., dotted blue line + two red dots) consistent with part~\ref{thm: down BD part1} of Theorem~\ref{thm: down BD}.
For $r = 4$, as we can see from Fig.~\ref{fig: boundary SC Case2}, we have $\partial \Mdo = \big[ \partial \B_1 \cap (\H_1^-)^c \big] \sqcup \T \sqcup \{ \x^*_1 \}$ (i.e., solid cyan line + dotted blue line + left red dot) consistent with part~\ref{thm: down BD part2} of Theorem~\ref{thm: down BD}.
For $r = 6$, as we can see from Fig.~\ref{fig: boundary SC Case3}, we have $\partial \Mdo = \big[ \partial \B_1 \cap (\H_1^-)^c \big] \sqcup \big[ \partial \B_2 \cap (\H_2^+)^c \big] \sqcup \T$ (i.e., solid cyan line + solid magenta line + dotted blue line) consistent with part~\ref{thm: down BD part3} of Theorem~\ref{thm: down BD}.
Note that the dotted blue line in the figures indicates that the corresponding set of points is not a subset of the inner approximation $\Mdo$, i.e., $\T \not\subseteq \Mdo$ whereas the solid cyan line and solid magenta line indicate that $\partial \B_1 \cap ( \H_1^- )^c \subseteq \Mdo$ and $\partial \B_2 \cap ( \H_2^+ )^c \subseteq \Mdo$, respectively.

\begin{figure}
\centering
\subfloat[For $r=2$, the characterization of boundary $\partial \Mdo$ corresponds to Theorem~\ref{thm: down BD} part~\ref{thm: down BD part1}.]{\includegraphics[width=.30\textwidth]{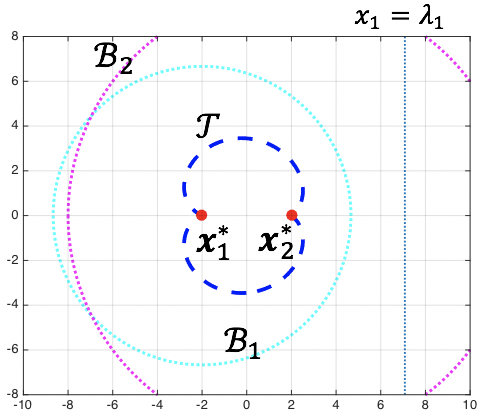} \label{fig: boundary SC Case1}}\;\;
\subfloat[For $r=4$, the characterization of boundary $\partial \Mdo$ corresponds to Theorem~\ref{thm: down BD} part~\ref{thm: down BD part2}.]{\includegraphics[width=.30\textwidth]{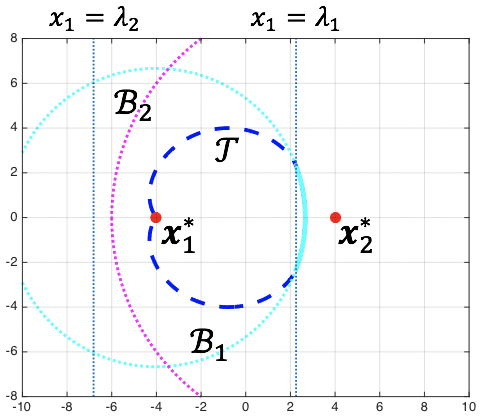} \label{fig: boundary SC Case2}}\;\;
\subfloat[For $r=6$, the characterization of boundary $\partial \Mdo$ corresponds to Theorem~\ref{thm: down BD} part~\ref{thm: down BD part3}.]{\includegraphics[width=.30\textwidth]{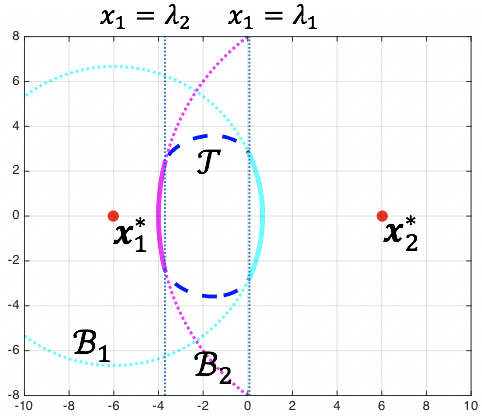} \label{fig: boundary SC Case3}}
\caption{The boundary $\partial \Mdo$ in $\R^2$ with different values of $r$ for two original minimizers $x_1^*=(-r,0)$ and $x_2^*=(r,0)$ is plotted given fixed parameters $\sigma_1 = 1.5$, $\sigma_2 = 1$, and $L = 10$. The set $\T$ is represented by blue dashed lines since $\T \subseteq (\Mdo)^c$, while the sets $\partial \B_1 \cap (\H_1^-)^c$ and $\partial \B_2 \cap (\H_2^+)^c$ are represented by cyan and magenta solid lines, respectively since they are both subsets of $\Mdo$. The vertical dotted lines represent the equations $x_1 = \lambda_1$ and $x_1 = \lambda_2$ and note that the value of $\lambda_1$ and $\lambda_2$ depends on $r$.}
\label{fig: inner boundary}
\end{figure}

%% file: contents/sec-solution.tex
\section{Potential solution region}  \label{sec: solution region}

In this section, using results from analyzing the outer approximation $\Mup$ in Section~\ref{sec: outer} and the inner approximation $\Mdo$ in Section~\ref{subsec: characterization}, we derive relationships among the potential solution region $\M$ (which is the set that we want to identify), outer approximation $\Mup$ and inner approximation $\Mdo$.

Before stating the main theorem, we summarize the inclusions among the three sets. Specifically, based on Proposition~\ref{prop: angle nc} and Proposition~\ref{prop: angle sc}, we get $\M \subseteq \Mup$ and $\Mdo \subseteq \M$, respectively, and we can state the following proposition.

\begin{proposition}
Suppose the sets $\M$, $\Mup$ and $\Mdo$ are defined as in \eqref{def: set M}, \eqref{def: set M up} and \eqref{def: set M down}, respectively. Then, $\Mdo \subseteq \M \subseteq \Mup$.
\label{prop: M inclusion} 
\end{proposition}

Since Theorem~\ref{thm: up BD} and Theorem~\ref{thm: down BD} are similar, we will see that, in fact, the boundary of the outer approximation $\partial \Mup$ and inner approximation $\partial \Mdo$ are equal to the boundary of the potential solution region $\partial \M$ for all values of $r$. This means that we obtain the explicit characterization of $\partial \M$ from Theorem~\ref{thm: up BD} or Theorem~\ref{thm: down BD}. We present this result in the following theorem. 

\begin{theorem}
Suppose $\M$, $\Mup$ and $\Mdo$ are defined as in \eqref{def: set M}, \eqref{def: set M up} and \eqref{def: set M down}, respectively. Then, $\partial \M = \partial \Mup = \partial \Mdo$.
\label{thm: boundary M}
\end{theorem}

\begin{proof}
Recall from Proposition~\ref{prop: M inclusion} that $\Mdo \subseteq \M \subseteq \Mup$. This entails that $(\Mdo)^{\circ} \subseteq (\M)^{\circ} \subseteq (\Mup)^{\circ}$ and $\overline{\Mdo} \subseteq \overline{\M} \subseteq \overline{\Mup}$. On the other hand, we have $(\Mdo)^{\circ} = (\Mup )^{\circ}$ and $\partial \Mdo = \partial \Mup$ from Theorem~\ref{thm: up BD} and Theorem~\ref{thm: down BD}.
Combine the facts regarding the interiors to obtain that $(\Mdo)^{\circ} = (\M )^{\circ} =  (\Mup )^{\circ}$. Then, we can write
\begin{equation*}
    \overline{\Mdo}
    = (\Mdo)^{\circ} \sqcup \partial \Mdo
    = (\Mup)^{\circ} \sqcup \partial \Mup
    = \overline{\Mup}.
\end{equation*}
Combining the above equation with $\overline{\Mdo} \subseteq \overline{\M} \subseteq \overline{\Mup}$, we can write $\overline{\Mdo} = \overline{\M} = \overline{\Mup}$. Since $\partial \mathcal{E} = \overline{\mathcal{E}} \setminus \mathcal{E}^{\circ}$ for any subset $\mathcal{E}$ in a topological space, we conclude that $\partial \M = \partial \Mup = \partial \Mdo$.
\end{proof}

Recall the definition of $\M$, $\T$, $\{ \partial \B_i \; \text{for} \; i \in \{ 1, 2 \} \}$ and $\{ \H_1^-, \H_2^+ \}$ from \eqref{def: set M}, \eqref{def: set T}, \eqref{def: set B boundary}, and \eqref{def: set H_i}, respectively.
Assuming that $\sigma_1 \geq \sigma_2$, we summarize a characterization of the potential solution region $\M$ as follows:
\begin{equation}
    \begin{cases}
    \partial \M = \T \sqcup \{ \x_1^*, \x_2^* \}
    & \text{if} \quad r \in \big( 0, \; \frac{L}{2 \sigma_1} \big], \\
    \partial \M = \big[ \partial \B_1 \cap (\H_1^-)^c \big] \sqcup \T \sqcup \{ \x^*_1 \}
    & \text{if} \quad r \in \big( \frac{L}{2 \sigma_1}, \; \frac{L}{2 \sigma_2} \big], \\
    \partial \M = \big[ \partial \B_1 \cap (\H_1^-)^c \big] \sqcup \big[ \partial \B_2 \cap (\H_2^+)^c \big] \sqcup \T
    & \text{if} \quad r \in \Big( \frac{L}{2 \sigma_2}, \; \frac{L}{2} \big( \frac{1}{\sigma_1} + \frac{1}{\sigma_2} \big) \Big), \\
    \M = \Big\{ \Big( \frac{L}{2} \big(\frac{1}{\sigma_1} - \frac{1}{\sigma_2} \big), \; \mathbf{0} \Big) \Big\} 
    & \text{if} \quad r = \frac{L}{2} \big( \frac{1}{\sigma_1} + \frac{1}{\sigma_2} \big), \\
    \M = \emptyset 
    & \text{if} \quad r \in \Big( \frac{L}{2} \big( \frac{1}{\sigma_1} + \frac{1}{\sigma_2} \big), \; \infty \Big),
    \end{cases}
    \label{eqn: final M}
\end{equation}
where the first three equations are obtained by applying Theorem~\ref{thm: up BD} and Theorem~\ref{thm: down BD} to Theorem~\ref{thm: boundary M}, and the last two equations are obtained by applying Theorem~\ref{thm: up BD} and Theorem~\ref{thm: down BD} to Proposition~\ref{prop: M inclusion}.
Examples of the boundary $\partial \M$ in $\R^2$ with different values of $r$ are shown in Fig.~\ref{fig: boundary M}. In particular, the solid blue, cyan and magenta curves in the figure correspond to the sets $\T$, $\partial \B_1 \cap (\H_1^-)^c$ and $\partial \B_2 \cap (\H_2^+)^c$, respectively.

%% file: contents/sec-discussion.tex
\section{Discussion} \label{sec: discussion}

\subsection{A different perspective of Assumption~\ref{asm: grad norm}}
\label{subsec: assumption3}

Rather than viewing $L$ solely as a bound on the gradient norm of $f_1$ and $f_2$ at the minimizer $\x^* = \argmin_{\x} (f_1 + f_2) (\x)$ as discussed in Section~\ref{sec: problem}, it is insightful to consider $L$ as a parameter influencing the size of the potential solution region $\M$. In essence, an increase in the parameter $L$ results in a larger potential solution region $\M$. This interpretation is based on the following analysis.

As $L$ increases, the quantities $\Tilde{\phi}_i (\x) = \arccos \left( \frac{\sigma_i}{L} \| \x - \x_i^* \| \right)$ for $i \in \{ 1, 2 \}$ (defined in \eqref{def: phi tilde}) increase, while the quantity $\psi (\x) = \pi - \left( \alpha_2 (\x) - \alpha_1 (\x) \right)$ (as defined in \eqref{def: psi}) remains unchanged. Given that both the outer approximation $\Mup$ \eqref{def: set M up} and the inner approximation $\Mdo$ \eqref{def: set M down} are constructed based on the inequality $\Tilde{\phi}_1 (\x) + \Tilde{\phi}_2 (\x) > \psi (\x)$ (differing only at a point and the equality sign), an increase in $L$ results in the expansion of both $\Mup$ and $\Mdo$. Consequently, the potential solution region $\M$ (defined in \eqref{def: set M}) enlarges, given the coincided boundary with both the outer and inner approximations, as established by Theorem~\ref{thm: boundary M}. Specifically, we have the following relationships:
\begin{equation}
    \begin{cases}
        \emptyset \neq \M(\x_1^*, \x_2^*, \sigma_1, \sigma_2, L_{\text{a}}) \subset \M(\x_1^*, \x_2^*, \sigma_1, \sigma_2, L_{\text{b}})
        & \text{if} \quad 2r \big( \frac{1}{\sigma_1} + \frac{1}{\sigma_2} \big)^{-1} \leq L_{\text{a}} < L_{\text{b}}, \\
        \M(\x_1^*, \x_2^*, \sigma_1, \sigma_2, L) = \emptyset
        & \text{if} \quad L < 2r \big( \frac{1}{\sigma_1} + \frac{1}{\sigma_2} \big)^{-1}.
    \end{cases}
    \label{eqn: L increase}
\end{equation}
Thus, the choice of $L$ effectively serves as a means to control the conservativeness of the resulting solution region, with larger $L$ leading to a more conservative and larger solution region. 

In addition, consider the characterization of the potential solution region $\M$ in \eqref{eqn: final M}. Rather than formulating the conditions in terms of the distance between the two minimizers $\x_1^*$ and $\x_2^*$, we can express them in relation to the parameter $L$ to gain insights into the composition of the potential solution region $\M$ for a given set of parameters. However, it is important to note that the dependency of the size of the potential solution region $\M$ on the distance $r$ is not monotonic, unlike its relationship with $L$.

\subsection{Alternative assumptions}
\label{subsec: alternative assumptions}

Conceptually, we can derive an implicit bound on the gradients $L$ (as introduced in Assumption~\ref{asm: grad norm}) by assuming the smoothness of the functions defined on a compact convex domain. This smoothness property is a standard assumption for strongly convex functions in the optimization literature \cite{bubeck2015convex, nesterov2018lectures}. Suppose $\mathcal{C} \subset \R^n$ denotes a compact convex set with diameter $D$. In other words, $D = \sup \{ \| \x - \y \| \in \R : \x, \y \in \mathcal{C} \}$. Furthermore, for $i \in \{ 1, 2 \}$, suppose a differentiable function $f_i: \mathcal{C} \to \R$ has an $L_i$-Lipschitz continuous gradient (i.e., $L_i$-smooth). Mathematically, we can write $\| \nabla f_i(\x) - \nabla f_i(\y) \| \leq L_i \| \x - \y \|$ for all $\x, \y \in \mathcal{C}$. Based on these assumptions, we have $\| \nabla f_i(\x) \| \leq D \max_j L_j$ for all $\x \in \mathcal{C}$ and $i \in \{ 1, 2 \}$, and thus, we can set $L = D \max_i L_i$. However, it is important to note that the parameter $L$ in this setting is relatively more conservative, as the bound applies not only to the point corresponding to the minimizer of the sum, $\argmin (f_1 + f_2)$, but also to every point in the domain $\mathcal{C}$. 

As another alternative, considering that our analysis revolves around examining the gradient angle $\phi_i$ as defined in \eqref{def: phi_i}, an appealing departure from Assumption~\ref{asm: grad norm} could involve directly constraining the class of functions by imposing a bound on the gradient angle, as opposed to enforcing a limit on the gradient norm at the minimizer of the sum. Specifically, we could set conditions such as
\begin{equation*}
    \phi_i (\x) = \angle ( \nabla f_i (\x), \x - \x_i^* ) \leq \theta_i 
    \quad \text{for all} \quad \x \in \dom (f_i) \setminus { \x_i^* }
    \quad \text{and} \quad i \in \{ 1, 2 \},
\end{equation*}
where $\theta_i$ is a parameter defining the class of functions. This condition has the advantage of interpretability, with $\theta_i$ representing the maximum allowable sensitivity of the function $f_i$. Interestingly, it also has a close connection to the condition number of $f_i$ as discussed further below.

For simplicity, we will omit the agent index subscript $i$ in the subsequent analysis and assume $\x_i^* = \0$. Consider a quadratic function $f(\x) = \frac{1}{2} \| \A \x \|^2$, where $\A \in \R^{d \times d}$ is a positive definite matrix. Now, we will examine the quantity $\sup_{\x \neq \0} \angle (\g(\x), \x)$, where $\g(\x) = \nabla f(\x) = \A^\intercal \A \x$. We aim to demonstrate that 
\begin{equation*}
    \sec \Big( \sup_{\x \neq \0} \angle (\g (\x), \x) \Big) \leq (\| \A \| \cdot \| \A^{-1} \|)^2 \triangleq \kappa,
\end{equation*}
where ${\| \; \cdot \; \|}$ denotes the induced matrix norm, and $\kappa$ is the condition number associated with the function $f$ \cite{gutman2021condition}. It is noteworthy that this inequality, with the replacement of $\sup_{\x \neq \0} \angle (\g (\x), \x)$ by $\sup_{\x \neq \x_*} \angle (\g (\x), \x - \x_*)$, holds for the more general case of $f(\x) = \frac{1}{2} \| \A ( \x - \x_* ) + \b \|^2$ with $\x_* \in \R^d$ and $\b \in \R^d$.

To show such result, we proceed as follows:
\begin{align*}
    \cos \Big( \sup_{\x \neq \0} \angle ( \g (\x), \x ) \Big)
    &= \inf_{\x \neq \0} \big( \cos \angle ( \g (\x), \x ) \big) 
    = \inf_{\x \neq \0} \left ( \frac{ \langle \g (\x), \x \rangle }{\| \g (\x) \| \cdot \| \x \|} \right ) \\
    &= \inf_{\x \neq \0} \left ( \frac{\| \A \x \|^2}{\| \x \|^2} \cdot \frac{\| \x \|}{\| \A^\intercal \A \x \|} \right )  
    \geq \bigg ( \inf_{\x \neq \0} \frac{\| \A \x \|}{\| \x \|} \bigg )^2 \bigg ( \sup_{\x \neq \0} \frac{\| \A^\intercal \A \x \|}{\| \x \|} \bigg )^{-1} \\
    &\geq \big( \| \A^{-1} \| \cdot \| \A \| \big)^{-2}.
\end{align*}
In the last inequality, we utilize the properties that $\inf_{\x \neq \0} \frac{\| \A \x \|}{\| \x \|} = \frac{1}{\| \A^{-1} \|}$ due to the invertibility of $\A$ and $\sup_{\x \neq \0} \frac{\| \A^\intercal \A \x \|}{\| \x \|} = \| \A^\intercal \A \| \leq \| \A \|^2$ due to the sub-multiplicative property of induced matrix norm, and $\| \A^\intercal \| = \| \A \|$.

\subsection{Discussion on quadratic function results}
\label{subsec: quadratic-discussion}

While the inner approximation result, Proposition~\ref{prop: angle sc}, specifically employs the reverse direction of the quadratic function result, Proposition~\ref{prop: function exist n-dim}, to establish the existence of a quadratic function with a specified minimizer, curvature, and gradient at a particular point, it is worth noting that Proposition~\ref{prop: function exist n-dim} applies to both directions and generally pertains to quadratic functions with positive definite Hessian matrices. Given its significance, we take a moment to summarize our findings here.

Essentially, the proposition offers an equivalent condition for the existence of a quadratic function $f$ with a specified gradient $\g \in \R^n$ at a given point $\x_0 \in \R^n$, namely $\nabla f (\x_0) = \g$. The two conditions are restated here for convenience (assuming $\x^*$ is the minimizer with $\x^* \neq \x_0$, and $\sigma$ is the minimum eigenvalue of the Hessian matrix of quadratic function $f$).
\begin{enumerate}[label=(\roman*)]
    \item $\x_0 \in \overline{\B} \big( \x^*, \frac{\| \g \|}{\sigma} \big)$ and
    \item $\angle ( \g, \x_0 - \x^*) \in \{ 0 \} \cup \Big[0, \; \arccos \big( \frac{\sigma}{\Vert \g \Vert} \Vert \x_0 - \x^* \Vert \big) \Big)$.
\end{enumerate}
It is important to note that these two conditions are straightforward to verify, as the second condition reduces to confirming that $\langle \g, \x_0 - \x^* \rangle > \sigma \| \x_0 - \x^* \|^2$ or $\g = k (\x_0 - \x^*)$ for some $k \in \R_+$. 

Next, we describe the geometric interpretation of the two conditions as follows: The first condition, relevant to the location of the point $\x_0$, arises due to the lower limit of the gradient norm's linear growth associated with the quadratic function. On the other hand, the second condition is related to the upper limit on the gradient deviation from the vector $\x_0 - \x^*$. We recall $\underline{L} (\x) = \sigma \| \x - \x^* \|$ from \eqref{eqn: L_lb} (with the omitted subscripted index), which in this case can be interpreted as the lower bound of the norm of the gradient of quadratic function $f \in \Q (\x^*, \sigma)$ at point $\x \in \R^n$. Thus, the second condition actually depends on the gradient ratio $\sfrac{\underline{L} (\x)}{\| \g \|}$. Interestingly, the first condition forces $\| \g \|$ to exceed the lower bound $\underline{L} (\x)$ ensuring the quantity $\arccos \big( \sfrac{\underline{L} (\x)}{\| \g \|} \big)$ is well-defined. Lastly, it is worth noting that as the gradient ratio is large (i.e., close to $1$), the allowable angle $\angle ( \g, \x_0 - \x^*)$ becomes small, and vice versa.

\subsection{Correspondence of function pairs and potential solutions}
\label{subsec: correspondence}

The definition of the potential solution region $\M$ in \eqref{def: set M} implies that if a point $\x \in \R^n$ lies within $\M$, there exists at least one pair of functions $( f_1, f_2 )$ such that $f_i \in \S_i (\x_i^*, \sigma_i)$ for $i \in \{ 1, 2 \}$, $\nabla f_1 (\x) = - \nabla f_2 (\x)$, and $\| \nabla f_1 (\x) \| = \| \nabla f_2 (\x) \| \leq L$. However, the question of how many function pairs can correspond to each point $\x$ within the potential solution region $\M$ remains unresolved. To shed light on this matter, we present the following theorem, the proof of which is provided in Appendix~\ref{sec: proof-correspondance}.

\begin{theorem}  \label{thm: correspondance}
Let $\F_{\S} (\x) = \big\{ (f_1, f_2) \in \S (\x_1^*, \sigma_1) \times \S (\x_2^*, \sigma_2): \nabla f_1 (\x) = - \nabla f_2 (\x), \; \| \nabla f_1 (\x) \| = \| \nabla f_2 (\x) \| \leq L \big\}$. Then, it holds that $| \F_{\S} (\x) | = \mathfrak{c}$ almost everywhere within the potential solution region $\M$, where $\mathfrak{c}$ denotes the cardinality of the continuum.
\end{theorem}

In essence, this theorem reveals that for almost every point $\x$ in the potential solution region $\M$, there exist infinitely many function pairs that can yield $\x$ as a minimizer of the sum of those functions.

\subsection{Insight into multiple functions case}
\label{subsec: multiple functions}

From an intuitive standpoint, when assessing whether a point $\x \in \R^n$ lies within the potential solution region $\M \subset \R^n$, the outcome should be contingent solely on the distances between three key points: $\x_1^*$, $\x_2^*$, and $\x$ where $\x_1^*$ and $\x_2^*$ are the minimizers of the functions $f_1$ and $f_2$, respectively. In other words, this dependence should be independent of the specific locations of these points. Given the same set of distances constructed from $\{ \x_1^*, \x_2^*, \x \}$ and $\{ \x_1^*, \x_2^*, \x' \}$, if one can find a pair of functions $f_1 \in \S (\x_1^*, \sigma_1)$ and $f_2 \in \S (\x_2^*, \sigma_2)$ that realize $\x = \argmin (f_1 + f_2)$, then there exists a pair of functions $f_1' \in \S (\x_1^*, \sigma_1)$ and $f_2' \in \S (\x_2^*, \sigma_2)$ that realize $\x' = \argmin (f_1' + f_2')$. This symmetry results in the potential solution region $\M$ being symmetric around the line passing through $\x_1^*$ and $\x_2^*$ (as shown in Fig.~\ref{fig: boundary M} for $\R^2$ case).

In Section~\ref{subsec: solution approach}, we formally substantiate this assertion by considering arbitrary locations for the minimizers $\x_1^* \in \R^n$ and $\x_2^* \in \R^n$, and show that the sets of strongly convex functions $\S (\x_i^*, \sigma_i)$ are closed under translation and rotation of all points in the functions' domain. Consequently, we can, without loss of generality, assume nonzero values only in the first dimension for both $\x_1^*$ and $\x_2^*$, signifying their transformation from the original locations through an affine transformation (corresponding to translation and rotation). Thus, after obtaining the potential solution region (from Theorem~\ref{thm: boundary M}), it becomes imperative to apply the corresponding inverse transformation to the set of points within the region to achieve the desired outcome. 

Similarly, in the context of $m$ functions within at least an $(m-1)$-dimensional space, it is possible, without loss of generality, to assume that the given $m$ minimizers have nonzero values exclusively in the first $m-1$ dimensions. Employing an inverse transformation allows the derivation of the final potential solution region. However, it is important to note that the analytical technique employed in this paper may encounter challenges in directly addressing scenarios involving multiple functions due to increased problem complexity. Roughly speaking, the number of combinations of gradients in which their sum is zero increases with the number of functions. Consequently, there may be a need to develop a novel approach tailored for such cases, as obtaining a tight approximation of the potential solution region, as achieved in the case of two functions, may be more intricate. 

One prospective avenue is to focus on acquiring straightforward inner and outer approximations of the potential solution region. For instance, in the scenario involving $m$ functions, a plausible approach could involve the use of $(m - 1)$-spheres. These two spheres would be centered at a shared midpoint among the $m$ minimizers but might have different radii, serving as inner and outer approximations. The disparity between the radii of these spheres would determine the precision of these approximations. It is hoped that such a streamlined approach would be feasible, providing meaningful outcomes without excessive conservatism.

\subsection{One-shot federated learning}
\label{subsec: one-shot learning}

As discussed in Section~\ref{subsec: privacy federated}, addressing privacy concerns may lead to a desire to minimize the number of communication rounds between each node and the centralized server in the federated setting. In an extreme scenario, one might aim to reduce the communication rounds to a minimum, potentially just one round, to maximize privacy preservation and mitigate the impact of gradient leakage attacks \cite{zhu2019deep, geiping2020inverting, wei2020framework}. In such cases, from the server's perspective, the problem is simplified to proposing a new parameter given available information, including the minimizer and general structure of the function from each node.

While it is possible to estimate the entire potential solution region $\M$ and choose any point within this region, such an approach may require substantial computational resources. Since only one point is required, in general, it is feasible to leverage the sufficient condition, i.e., the inner approximation $\Mdo$ defined in \eqref{def: set M down}, as every point inside $\Mdo$ is guaranteed to be in the potential solution region $\M$ (Proposition~\ref{prop: M inclusion}). Considering that the selected point corresponds to a new model with enhanced performance used for inference at each node, this approach could be considered a form of \emph{one-shot federated learning} \cite{guha2019one, zhang2022dense}.

However, intuitively, a point closer to the boundary of the potential solution region $\partial \M$ might be perceived as more vulnerable than one in the middle. It is desirable to select a point that is not only meaningful but also robust to parameter variations. While various criteria exist for choosing such a point, in the specific scenario involving two nodes, we propose the selection of the point
\begin{equation}
    \p = \Big( \frac{\sigma_1}{\sigma_1 + \sigma_2} \Big) \x_1^* + \Big( \frac{\sigma_2}{\sigma_1 + \sigma_2} \Big) \x_2^*
    \quad \text{for general} \quad \x_1^*, \x_2^* \in \R^n
    \label{eqn: point p}
\end{equation}
as the parameter corresponding to the new model. In addition to its simplicity, this point $\p$ possesses several favorable properties. First, the point $\p \in \R^n$ lies in the convex hull of $\{ \x_1^*, \x_2^* \}$, and the weight assigned to $\x_i^*$ is proportional to its strong convexity parameter $\sigma_i$. Second, this point $\p$ belongs to the potential solution region $\M$ and is robust to variations in the parameter $L$, as we will discuss next.

As discussed in Section~\ref{subsec: assumption3}, the size of the potential solution region $\M$ increases with the parameter $L$. Recall the minimizers $\x_1^* = (-r, \0)$, $\x_2^* = (r, \0)$, and $\| \x_2^* - \x_1^* \| = 2r$. From the relationship provided in \eqref{eqn: L increase}, it implies that choosing a point $\p'$ in the set $\M(\x_1^*, \x_2^*, \sigma_1, \sigma_2, L')$ when $L' = 2r \big( \frac{1}{\sigma_1} + \frac{1}{\sigma_2} \big)^{-1}$ corresponds to selecting the \emph{middle} point of the potential solution region $\M$. Moreover, this point is robust to the parameter $L$ in the sense that it belongs to $\M$ regardless of the specific value of $L$, provided that $\M$ is non-empty. In fact, the set $\M(\x_1^*, \x_2^*, \sigma_1, \sigma_2, L')$ for $L' = 2r \big( \frac{1}{\sigma_1} + \frac{1}{\sigma_2} \big)^{-1}$ contains only one point, corresponding to the fourth case in \eqref{eqn: final M}. Substituting $L' = 2r \big( \frac{1}{\sigma_1} + \frac{1}{\sigma_2} \big)^{-1}$ into the expression for $\p' = \Big( \frac{L'}{2} \big(\frac{1}{\sigma_1} - \frac{1}{\sigma_2} \big), \; \0 \Big)$ yields $\p' = \Big( \big( \frac{\sigma_2 - \sigma_1}{\sigma_2 + \sigma_1} \big) r, \0 \Big)$. Note that this point
\begin{equation*}
    \p' = \left( \Big( \frac{\sigma_2 - \sigma_1}{\sigma_2 + \sigma_1} \Big) r, \0 \right)
    = \Big( \frac{\sigma_1}{\sigma_1 + \sigma_2} \Big) (-r, \0) + \Big( \frac{\sigma_2}{\sigma_1 + \sigma_2} \Big) (r, \0)
\end{equation*}
is identical to the point $\p$ defined in \eqref{eqn: point p} for the case when $\x_1^* = (-r, \0)$, $\x_2^* = (r, \0)$. Furthermore, it is evident that the point $\p$ is a generalization of $\p'$ and inherits the same properties for arbitrary $\x_1^*, \x_2^* \in \R^n$.

\subsection{Applicability of our approach}
\label{subsec: applicability}

Our analysis in this work primarily focuses on strongly convex loss functions commonly encountered in traditional machine learning models, such as linear regression \cite{weisberg2005applied, montgomery2021introduction}, logistic regression \cite{menard2002applied, hosmer2013applied}, and support vector machines (SVM) \cite{hearst1998support, steinwart2008support, mountrakis2011support}. However, extending our approach to more complex loss functions, such as those arising from neural networks, presents challenges due to the intractability of searching for optimal parameters for each function \cite{murty1985some, nesterov2018lectures}.

Nonetheless, our analysis remains relevant in the context of \textit{locally strongly convex functions} \cite{vial1982strong, goebel2008local, yan2014extension}. In distributed scenarios, we may consider a restricted domain containing a local minimum of each (potentially stochastic) non-convex loss function, as identified by practical optimization algorithms like Stochastic Gradient Descent (SGD) \cite{bottou2010large} or Adam \cite{kingma2014adam}. Within this restricted region, if the two loss functions can be approximated by strongly convex functions or a more general class of functions (as we will introduce shortly), our analysis, including Theorem~\ref{thm: up BD}, Theorem~\ref{thm: down BD}, and our proposed solution for one-shot federated learning (discussed in Section~\ref{subsec: one-shot learning}), can be applied. While this idea is conceptually plausible, a detailed exploration of this intricate case is beyond the scope of our current work, and we leave it as an intriguing avenue for future research.

Apart from supervised learning problems, our approach can be applied to some important classes of problems closely related to machine learning in general. Before delving into the detailed discussion, we introduce a class of functions called \textit{restricted secant inequality (RSI)} \cite{zhang2013gradient, zhang2015restricted} and \textit{restricted strong convexity (RSC)} \cite{lai2013augmented, zhang2017restricted} as follows:
\begin{definition}
A function $f: \R^n \to \R$ satisfies the restricted secant inequality (RSI) if $f$ is differentiable and there exists a positive constant $\sigma \in \R_+$ such that
\begin{equation}
    \big\langle \nabla f(\x) - \nabla f(\xprj), \; \x - \xprj \big\rangle 
    \geq \sigma \| \x - \xprj \|^2,
    \label{def: RSI}
\end{equation}
where $\xprj = \Pi_{\mathcal{X}^*} (\x)$ and $\Pi_{\mathcal{X}^*} ( \cdot )$ is the projection operator onto the set of minimizers of function $f$, denoted as $\mathcal{X}^*$.
\end{definition}

It is important to note that $\nabla f(\xprj) = \mathbf{0}$ by the definition of $\xprj$. Additionally, the constant $\sigma$ in \eqref{def: RSI} can be interpreted as a lower bound on the average curvature of $f$ between $\x$ and $\xprj$. However, the inequality \eqref{def: RSI} does not imply that $f$ grows everywhere faster than the quadratic function $\Tilde{f} (\x) = \frac{\sigma}{2} \| \x - \xprj \|^2$. Essentially, RSI functions, in general, can be non-convex, as illustrated in \cite{zhang2013gradient, zhang2015restricted}.

\begin{definition}
A function $f: \R^n \to \R$ is restricted strongly convex (RSC) if it is convex, has a finite minimizer, and satisfies the restricted secant inequality \eqref{def: RSI}.
\end{definition}

It is noteworthy that the definitions of RSC functions used here follow those given in \cite{lai2013augmented, zhang2017restricted} rather than those in \cite{negahban2009unified, negahban2012unified, negahban2012restricted}. Interestingly, the work \cite{lai2013augmented} has demonstrated the applicability of a smooth version (achieved by augmenting with an $L_2$-regularization) of basis pursuit (BP) problems \cite{chen2001atomic} and low-rank matrix recovery problems \cite{fazel2002matrix, recht2010guaranteed, candes2012exact}.

Here, we denote \textit{SC} as the class of strongly convex functions, i.e., $\bigcup_{\x^* \in \R^n} \bigcup_{\sigma \in \R_+} \S (\x^*, \sigma)$, where the set of strongly convex functions $\S$ is defined in Section~\ref{subsec: convex}. Similarly, we denote \textit{Q-PD} as the class of quadratic functions with positive definite (PD) Hessian matrices, i.e., $\bigcup_{\x^* \in \R^n} \bigcup_{\sigma \in \R_+} \Q (\x^*, \sigma)$, where the set of quadratic functions $\Q$ is defined in Section~\ref{subsec: convex}. One important application of quadratic programming includes least squares problems \cite{lawson1995solving, bjorck1996numerical}, which are applicable to a wide range of engineering problems, including optimal control \cite{willems1971least}, system identification \cite{aastrom1971system, chen1989orthogonal}, and signal processing \cite{chen1991orthogonal}.

Based on these definitions, it can be concluded that RSI is weaker than RSC, RSC is weaker than SC, and SC is weaker than Q-PD. In fact, the core of our analysis of the outer approximation $\Mup$ (defined in \eqref{def: set M up}) is based on the inequality \eqref{eqn: strongly convex eq1}, which is essentially equivalent to RSI functions with $\mathcal{X}^*$ being a singleton. Hence, our conclusions regarding the outer approximation $\Mup$, e.g., Theorem~\ref{thm: up BD}, also hold particularly for RSI functions with a singleton $\mathcal{X}^*$ and RSC functions.

Specifically, our results are applicable to the distributed version of basis pursuit (BP) problems, low-rank matrix recovery problems, as well as the aforementioned applications involving distributed least squares problems \cite{sayed2006distributed, cattivelli2008diffusion, mateos2009distributed, mateos2012distributed}. Interestingly, our analysis techniques involving angle analysis as depicted in Figure~\ref{fig: angle and nc} used in Section~\ref{sec: outer} can be extended to general RSI functions, although some complications may arise due to the arbitrary form of the set of minimizers $\mathcal{X}_i^*$ of functions $f_i$.

%% file: contents/sec-conclusion.tex
\section{Conclusion} \label{sec: conclusion}

In this work, we studied the possible locations of the minimizer of the sum of two strongly convex functions. Based on the location of the two minimizers of the individual functions, strong convexity parameters and a bound on the gradients at the minimizer of the sum, we established a necessary condition and a sufficient condition for a given point to be a minimizer, and called the set of points that satisfies the conditions as the outer approximation $\Mup$ and inner approximation $\Mdo$, respectively. We then explicitly characterized the boundary and interior of the outer and inner approximations.
The characterization of these boundaries and interiors turned out to be identical. Subsequently, we showed that the boundary of the potential solution region $\partial \M$ is also identical to those boundaries. In particular, we showed that it is sufficient to consider quadratic functions to establish (almost) the entire set of potential minimizers. To visualize the boundary of the potential solution region $\partial \M$, we provided examples with different distances between the two original minimizers in Fig.~\ref{fig: boundary M}.

Our work in this paper has focused on the case of two functions. Future research avenues could include identifying the region within which the minimizer of the sum can lie in the case of multiple strongly convex functions. Alternatively, exploring different classes of functions, such as those satisfying the Polyak-Lojasiewicz (PL) inequality \cite{karimi2016linear}, pre-invex functions \cite{weir1988pre, weir1988class, yang2001properties}, or general non-convex functions \cite{rockafellar1980generalized, jain2017non}, could provide valuable insights. Additionally, one might consider modifying certain assumptions, such as incorporating strongly convex functions with Lipschitz continuous gradient conditions \cite{bauschke2017descent} or imposing constraints on the condition number of the functions \cite{demmel1987condition}.

%% file: contents/supp-outer.tex
\section{Proofs of theoretical results for outer approximation}  \label{sec: proof lemma outer}

\subsection{Proof of Lemma~\ref{lem: 1D analyze}}

\begin{proof}[Proof of Lemma~\ref{lem: 1D analyze}]
Consider part~\ref{lem: 1D case1} of the lemma. Suppose $\x_{\epsilon} = \x_1^* + \epsilon \boldsymbol{e}_1$ for $\epsilon \in \R$. We want to show that
\begin{enumerate}[label=(\alph*)]
    \item If $r \in \big( 0, \; \frac{L}{2 \sigma_2} \big]$ then for all $\epsilon \in \Big( 0, \; \min \big\{ \frac{L}{ \sigma_1}, 2r \big\} \Big)$, we have $\tilphi_1 (\x_{\epsilon}) + \tilphi_2 (\x_{\epsilon}) > \psi (\x_{\epsilon})$.
    \label{case lem: inside}
    \item If $r \in \big( 0, \; \frac{L}{2 \sigma_2} \big)$ then for all $\epsilon \in \Big( - \min \big\{ \frac{L}{\sigma_1}, \frac{L}{\sigma_2} - 2r \big\}, \; 0 \Big)$, we have $\tilphi_1 (\x_{\epsilon}) + \tilphi_2 (\x_{\epsilon}) < \psi (\x_{\epsilon})$.
    \label{case lem: outside1}
    \item If $r = \frac{L}{2 \sigma_2}$ then for all $\epsilon \in (-\infty, 0)$, we have $\x_{\epsilon} \notin (\overline{\B}_1 \cap \overline{\B}_2) \setminus \{ \x_1^*, \x_2^* \}$.
    \label{case lem: outside2}
\end{enumerate}

First, consider case~\ref{case lem: inside}. Suppose $\epsilon \in \big( 0, \; \min \big\{ \frac{L}{ \sigma_1}, 2r \big\} \big)$. Let $x_{\epsilon, 1}$ be the first component of the point $\x_{\epsilon}$. We have $x_{\epsilon, 1} \in \big( -r, \; \min \big\{ -r + \frac{L}{\sigma_1}, \; r \big\} \big)$. Since $r \in \big( 0, \; \frac{L}{2 \sigma_2} \big]$, we have $-r \in \big[ r - \frac{L}{\sigma_2}, r \big)$. Therefore, $\x_{\epsilon} \in (\B_1 \cap \B_2) \setminus \{ \x_1^*, \x_2^* \}$.
By the location of $\x_{\epsilon}$, from \eqref{def: alpha_i}, we get $\alpha_1 (\x_{\epsilon}) = 0$ and  $\alpha_2 (\x_{\epsilon}) = \pi$. Consequently, from \eqref{def: psi}, we obtain $\psi (\x_{\epsilon}) = 0$. Since $\x_{\epsilon} \notin \partial \B_1 \cup \partial \B_2$, from \eqref{def: phi tilde}, we have that $\tilphi_i (\x_{\epsilon}) \in \big( 0, \frac{\pi}{2} \big]$ for $i \in \{1, 2 \}$. This means that $\tilphi_1 (\x_{\epsilon}) + \tilphi_2 (\x_{\epsilon}) > \psi (\x_{\epsilon})$.

Second, consider case~\ref{case lem: outside1}. Suppose $\epsilon \in \big( - \min \big\{ \frac{L}{\sigma_1}, \frac{L}{\sigma_2} - 2r \big\}, \; 0 \big)$. We have $x_{\epsilon, 1} \in \big( \max \big\{ -r - \frac{L}{\sigma_1}, \; r - \frac{L}{\sigma_2} \big\}, \; -r \big)$ which implies that $\x_{\epsilon} \in (\B_1 \cap \B_2) \setminus \{ \x_1^*, \x_2^* \}$.
By the location of $\x_{\epsilon}$, from \eqref{def: alpha_i}, we get $\alpha_1 (\x_{\epsilon}) = \pi$ and  $\alpha_2 (\x_{\epsilon}) = \pi$. Consequently, from \eqref{def: psi}, we obtain $\psi (\x_{\epsilon}) = \pi$. Since $\x_{\epsilon} \notin \{ \x_1^*, \x_2^* \}$, from \eqref{def: phi tilde}, we have that $\tilphi_i (\x_{\epsilon}) \in \big[ 0, \frac{\pi}{2} \big)$ for $i \in \{1, 2 \}$. This means that $\tilphi_1 (\x_{\epsilon}) + \tilphi_2 (\x_{\epsilon}) < \psi (\x_{\epsilon})$.

Third, consider case~\ref{case lem: outside2}. Suppose $\epsilon \in (-\infty, 0)$. If $\x \in \overline{\B}_2$ then $x_1 \in \big[ r - \frac{L}{\sigma_2}, r + \frac{L}{\sigma_2} \big] = [ -r, 3r ]$. However, we have $x_{\epsilon, 1} \in (-\infty, -r)$. This means that $\x_{\epsilon} \notin \overline{\B}_2$ which implies that $\x_{\epsilon} \notin (\overline{\B}_1 \cap \overline{\B}_2) \setminus \{ \x_1^*, \x_2^* \}$. We complete the proof of our claim.

Recall the definition of $\Mup$ and $\Mdo$ in \eqref{def: set M up} and \eqref{def: set M down}, respectively. From our claim, we can see that if $r \in \big( 0, \; \frac{L}{2 \sigma_2} \big]$ then for all $\delta \in \R_{>0}$, there exist $\x_{\text{in}}, \x_{\text{out}} \in \B (\x_1^*, \delta)$ such that $\x_{\text{in}} \in \Mup$, $\x_{\text{in}} \in \Mdo$, $\x_{\text{out}} \notin \Mup$ and $\x_{\text{out}} \notin \Mdo$. 
This implies that $\x_1^* \in \partial \Mup$ and $\x_1^* \in \partial \Mdo$. The analysis for part~\ref{lem: 1D case2} of the lemma follows in an identical manner. 
\end{proof}

\subsection{Proof of Lemma~\ref{lem: point below}}

\begin{proof}[Proof of Lemma~\ref{lem: point below}]
Recall from \eqref{def: alpha_i} that $\alpha_i (\x) = \angle ( \x - \x_i^*, \; \x_2^* - \x_1^* )$ for $i \in \{ 1, 2 \}$. Suppose $\y \notin \{ \x_1^*, \x_2^* \}$. Since $-r \leq x_1 = y_1 \leq r$ and $\Vert \tilde{\mathbf{x}} \Vert > \Vert \tilde{\mathbf{y}} \Vert$, we have $\alpha_1 (\x) \geq \alpha_1 (\y)$ and $\alpha_2(\x) \leq \alpha_2 (\y)$, which implies that
\begin{equation}
    \psi (\x) = 
    \pi - (\alpha_2(\x) - \alpha_1(\x)) 
    \geq \pi - (\alpha_2(\y) - \alpha_1(\y))
    = \psi (\y).
    \label{eqn: alpha ineq}
\end{equation}
On the other hand, since $x_1 = y_1$ and $\Vert \tilde{\mathbf{x}} \Vert > \Vert \tilde{\mathbf{y}} \Vert$, we get 
\begin{equation*}
    \| \x - \x_1^* \| > \| \y - \x_1^* \| > 0
    \quad \text{and} \quad 
    \| \x - \x_2^* \| > \| \y - \x_2^* \| > 0.
\end{equation*}
Using the above inequalities and the definition of \eqref{def: phi tilde}, we get that $\tilphi_i(\x) < \tilphi_i (\y)$ for $i \in \{1,2 \}$. Applying $\tilphi_1 (\x) + \tilphi_2 (\x) \geq \psi (\x)$ and inequality \eqref{eqn: alpha ineq}, we can write
\begin{equation*}
    \psi (\y) \leq \psi (\x)
    \leq \tilphi_1 (\x) + \tilphi_2 (\x) 
    < \tilphi_1 (\y) + \tilphi_2 (\y),
\end{equation*}
which completes the proof.
\end{proof}

\subsection{Proof of Lemma~\ref{lem: point on R_n}}

\begin{proof}[Proof of Lemma~\ref{lem: point on R_n}]
Consider part~\ref{lem: lambda_1} of the lemma. Since $\x \in \partial \B_1$, we get $\| \x - \x_1^* \| = \frac{L}{\sigma_1}$ and thus $\tilphi_1(\x) = 0$ from \eqref{def: phi tilde}. Consider the inequality $\tilphi_1 (\x) + \tilphi_2 (\x) \lessgtr \psi (\x)$.
Substitute $\tilphi_1(\x) = 0$ and $\psi (\x) = \pi - ( \alpha_2 (\x) - \alpha_1 (\x) )$, and take cosine of both sides of the inequality (and use $\cos \tilphi_2 (\x) = \frac{\sigma_2}{L} d_2 (\x)$ from \eqref{def: phi tilde}) to get 
\begin{equation*}
    \frac{\sigma_2}{L} d_2 (\x) \gtrless  - \cos \big( \alpha_2 (\x) - \alpha_1 (\x) \big). 
\end{equation*}
Expand the cosine and substitute the equations \eqref{eqn: cos alpha} and \eqref{eqn: sin alpha} to obtain
\begin{equation}
    \frac{\sigma_2}{L} d_2 (\x) \gtrless  -\frac{x_1^2 + \Vert \tilde{\mathbf{x}} \Vert^2 - r^2}{d_1 (\x) \cdot d_2 (\x)}. 
\label{eqn: point in H ineq}
\end{equation}
Since $\x \in \partial \B_1$, we have $d_1 (\x) = \frac{L}{\sigma_1}$ and $\Vert \tilde{\mathbf{x}} \Vert^2 = \frac{L^2}{\sigma_1^2} - (x_1 + r)^2$ from \eqref{def: set B boundary}. 
Also, from \eqref{def: distance}, we have 
\begin{equation}
    d_2^2 (\x) 
    = (x_1 - r)^2 + \Vert \tilde{\mathbf{x}} \Vert^2 
    = (x_1 - r)^2 + \frac{L^2}{\sigma_1^2} - (x_1 + r)^2 
    = -4r x_1 + \frac{L^2}{\sigma_1^2}.
    \label{eqn: d_2 square}
\end{equation}
Multiply inequality \eqref{eqn: point in H ineq} by $d_1 (\x) \cdot d_2 (\x)$ and then substitute $d_1 (\x)$, $\Vert \tilde{\mathbf{x}} \Vert^2$, and $d_2^2 (\x)$ to get
\begin{align*}
    &\frac{\sigma_2}{\sigma_1} \Big(-4r x_1 + \frac{L^2}{\sigma_1^2} \Big) \gtrless 2r^2 + 2r x_1 - \frac{L^2}{\sigma_1^2}, \\
    \Leftrightarrow \quad 
    &x_1 \Big(2r + 4r \frac{\sigma_2}{\sigma_1} \Big) \lessgtr \frac{\sigma_2}{\sigma_1} \cdot \frac{L^2}{\sigma_1^2} + \frac{L^2}{\sigma_1^2} - 2r^2, \\
    \Leftrightarrow \quad 
    &x_1 \lessgtr \Big(\frac{1 + \beta}{1 + 2 \beta} \Big) \frac{\gamma_1}{2 r} - \frac{r}{1 + 2 \beta} = \lambda_1,
\end{align*}
where $\gamma_1$ and $\beta$ are defined in \eqref{def: gamma beta}. The proof of the second part is similar to the first part. 
\end{proof}

\subsection{Proof of Lemma~\ref{lem: intersection}}

\begin{proof}[Proof of Lemma~\ref{lem: intersection}]
We will prove the case that $i = 1$; however, the proof of the case that $i = 2$ can be obtained using the same approach. Suppose $\x \in \partial \B_1$.
Substituting  $\Vert \tilde{\mathbf{x}} \Vert^2 = \frac{L^2}{\sigma_1^2} - (x_1 + r)^2$, $d_1(\x) = \frac{L}{\sigma_1}$ and \eqref{eqn: d_2 square} into the expression of $\T$ in \eqref{def: set T}, we get
\begin{equation}
    \frac{x_1^2 + \| \tilde{\mathbf{x}} \|^2 - r^2}{d_1^2(\x) \cdot d_2^2(\x)} + \frac{\sigma_1 \sigma_2}{L^2}
	= \frac{\frac{L^2}{\sigma_1^2} -2r x_1 -2r^2}{\frac{L^2}{\sigma_1^2} \big( -4r x_1 + \frac{L^2}{\sigma_1^2} \big) } + \frac{\sigma_1 \sigma_2}{L^2} = 0.
	\label{eqn: sub B1 into T}
\end{equation}
Multiply both sides of the above equation by $\frac{L^2}{\sigma_1^2} \big( -4r x_1 + \frac{L^2}{\sigma_1^2} \big)$ and then use the definition of $\gamma_1$ and $\beta$ in \eqref{def: gamma beta} to get
\begin{displaymath}
	4 \beta r x_1 - \beta \gamma_1 = \gamma_1 -2r x_1 -2r^2,
	\quad \Leftrightarrow \quad
	x_1 = \Big(\frac{1+\beta}{1 + 2 \beta} \Big) \frac{\gamma_1}{2r} - \frac{r}{1+2 \beta} = \lambda_1.
\end{displaymath}
Then, substituting $x_1 = \big(\frac{1+\beta}{1 + 2 \beta} \big) \frac{\gamma_1}{2r} - \frac{r}{1+2 \beta}$ back into the expression of $\partial \B_1$ in \eqref{def: set B boundary}, we get
\begin{align*}
	\Vert \tilde{\mathbf{x}} \Vert^2 &= \gamma_1 - \bigg[ \Big(\frac{1+\beta}{1 + 2 \beta} \Big) \frac{\gamma_1}{2r} - \frac{r}{1+2 \beta} + r \bigg]^2 \\
	&= - \frac{1}{4r^2 (1+2\beta)^2} \big[ (\gamma_1 - 4r^2) (\gamma_1 (1+ \beta)^2 - 4\beta^2 r^2) \big], \\
	\Rightarrow \quad
	\Vert \tilde{\mathbf{x}} \Vert &= \frac{r}{2(1+ 2\beta)} \sqrt{- \Big(\frac{\gamma_1}{r^2} - 4 \Big) \Big(\frac{\gamma_1}{r^2} (1+\beta)^2 - 4 \beta^2 \Big)} = \nu_1.
\end{align*}
This equation is valid only when the term in the square root is a non-negative real number. Equivalently,
\begin{equation*}
    \frac{4 \beta^2}{(1+ \beta)^2} \leq \frac{\gamma_1}{r^2} \leq 4,
    \quad \Leftrightarrow \quad
    \frac{L}{2 \sigma_1} \leq r \leq \frac{L}{2 \sigma_1}  \Big(1+\frac{1}{\beta} \Big) = \frac{L}{2} \Big( \frac{1}{\sigma_1} + \frac{1}{\sigma_2} \Big).
\end{equation*}
Since we multiplied \eqref{eqn: sub B1 into T} by $d_1^2(\x) \cdot d_2^2(\x) = \frac{L^2}{\sigma_1^2} \big( -4r x_1 + \frac{L^2}{\sigma_1^2} \big)$ to obtain the result, $\x_1^*$ or $\x_2^*$ might appear in the intersection $\T \cap \partial \B_1$ even if $\{ \x_1^*, \x_2^* \} \not\subseteq \T$. Therefore, we need to verify that the intersection points are not $\x^*_1$ or $\x^*_2$.
Considering the solution of the equation $\Vert \tilde{\mathbf{x}} \Vert = \nu_1 = 0$ (i.e., $\x$ is on the $x_1$-axis), we see that $r = \frac{L}{2 \sigma_1}$ and $r = \frac{L}{2} \big( \frac{1}{\sigma_1} + \frac{1}{\sigma_2} \big)$ are the candidates that we need to check. Substituting $r = \frac{L}{2 \sigma_1}$ into the equation $x_1 = \lambda_1$, we get $x_1 = \frac{L}{2 \sigma_1} = r$ which is $\x_2^* = (r, \mathbf{0})$ and therefore we conclude that there are no intersection points when $r = \frac{L}{2 \sigma_1}$. Next, substituting $r = \frac{L}{2} \big( \frac{1}{\sigma_1} + \frac{1}{\sigma_2} \big)$ into the equation $x_1 = \lambda_1$, we get $x_1 = \frac{L}{2} \big( \frac{1}{\sigma_1} - \frac{1}{\sigma_2} \big)$ which is a legitimate value of an intersection point.    
\end{proof}

\subsection{Proof of Lemma~\ref{lem: compare x1}}

\begin{proof}[Proof of Lemma~\ref{lem: compare x1}]
To simplify notations in the proof, we define a new variable
\begin{equation}
    \chi_1 \triangleq \frac{L}{\sigma_1 r}
    = \frac{\sqrt{\gamma_1}}{r}. \label{var: chi}
\end{equation}
First, consider part \ref{lem: compare x1 part2} of the lemma. By the definition of $\lambda_1$ in \eqref{var: lambda}, we can rewrite the inequality $\lambda_1 < \frac{L}{\sigma_1} -r$ as $\big(\frac{1 + \beta}{1 + 2 \beta} \big) \frac{\gamma_1}{2 r^2} - \frac{1}{1 + 2 \beta} < \frac{L}{\sigma_1 r}- 1$. Using $\chi_1$ defined in \eqref{var: chi}, the inequality becomes
\begin{equation}
\frac{1}{2} \Big(\frac{1 + \beta}{1 + 2 \beta} \Big) \chi_1^2 - \frac{1}{1 + 2 \beta} < \chi_1 - 1. 
 \label{eqn: lem_ineq}
\end{equation}
Multiplying both sides by $1+ 2 \beta$ and then rearranging the resulting inequality, we get 
\begin{equation*}
    (\chi_1 - 2) \Big( \frac{1}{2} (1 + \beta) \chi_1 - \beta \Big) < 0,
    \quad \Leftrightarrow \quad 
    \frac{2 \beta}{1 + \beta} < \chi_1 < 2.
\end{equation*}
The last equivalence holds since $\frac{2 \beta}{1+ \beta} < 2$. Substituting $\chi_1 = \frac{L}{\sigma_1 r}$ and $\beta = \frac{\sigma_2}{\sigma_1}$, the inequality $\frac{2 \beta}{1 + \beta} < \chi_1$ becomes $r < \frac{L}{2} \big( \frac{1}{\sigma_1} + \frac{1}{\sigma_2} \big)$ and the inequality $\chi_1 < 2$ becomes $r > \frac{L}{2 \sigma_1}$, which completes the proof of part \ref{lem: compare x1 part2}. 

Next, we can see that part \ref{lem: compare x1 part2} of the lemma implies that if $r \in \big( 0, \; \frac{L}{2 \sigma_1} \big]$ then $\lambda_1 \geq \frac{L}{\sigma_1} -r$. By proceeding the algebraic simplification similar to \eqref{eqn: lem_ineq} except using equality instead, we can conclude that given $r \in \big( 0, \; \frac{L}{2 \sigma_1} \big]$, we have $\lambda_1 = \frac{L}{\sigma_1} -r$ only when $r = \frac{L}{2 \sigma_1}$. This completes the proof of part \ref{lem: compare x1 part1}.

The proof of part \ref{lem: compare x1 part3} and part \ref{lem: compare x1 part4} can be carried out using similar steps as the proof given above. 
\end{proof}

\subsection{Proof of Lemma~\ref{lem: boundary}}

\begin{proof}[Proof of Lemma~\ref{lem: boundary}]
Let $\x = ( x_1, \tilde{\mathbf{x}} ) \in \T$, and $\boldsymbol{e}_2' (\x) = \frac{\v (\x)}{ \| \v (\x) \| }$ 
where $\v (\x) = (\x - \x_1^*) - \langle \x - \x_1^*, \boldsymbol{e}_1 \rangle \; \boldsymbol{e}_1$. 
Suppose $\theta (\x) = \frac{1}{2} (\alpha_1(\x) + \alpha_2(\x))$ where $\alpha_i (\x)$ for $i \in \{1, 2\}$ are defined in \eqref{def: alpha_i} and 
\begin{equation}
    \x'_{\epsilon} = \x + \epsilon \big( \cos ( \theta (\x) ) \; \boldsymbol{e}_1 + \sin ( \theta  (\x) ) \; \boldsymbol{e}_2' (\x) \big)
    \quad \text{for} \quad \epsilon \in \R.
    \label{eqn: x epsilon}
\end{equation}
Our claim is that if $r \in \big(0, \; \frac{L}{2} \big( \frac{1}{\sigma_1} + \frac{1}{\sigma_2} \big) \big)$ then
\begin{enumerate}[label=(\roman*)]
    \item for all 
    \begin{equation*}
        \epsilon \in \bigg( \max \Big\{ \max_{i=1,2}  - 2 d_i (\x) \sin  \frac{1}{2} \big( \tilphi_1 (\x) + \tilphi_2 (\x) \big), \;  -\| \tilde{\mathbf{x}} \| \csc ( \theta (\x) ) \Big\} , \; 0 \bigg),
    \end{equation*}
    it holds that $\tilphi_1(\x'_{\epsilon}) + \tilphi_2(\x'_{\epsilon}) > \psi(\x'_{\epsilon})$ (i.e., $\x'_\epsilon \in \Mup$ and $\x'_\epsilon \in \Mdo$), and
    \label{lem: boundary part1}
    
    \item for all $\epsilon \in (0, \infty )$, either $\tilphi_1(\x'_{\epsilon}) + \tilphi_2(\x'_{\epsilon}) < \psi(\x'_{\epsilon})$, or $\tilphi_1 (\x'_{\epsilon})$ or $\tilphi_2 (\x'_{\epsilon})$ is not well-defined (i.e., $\x'_\epsilon \in (\Mup)^c$ and $\x'_\epsilon \in (\Mdo)^c$).
    \label{lem: boundary part2}
\end{enumerate}

Recall from Proposition~\ref{prop: T_n} that $\T = \big\{ \z \in \R^n : \tilphi_1 (\z) + \tilphi_2 (\z) = \psi (\z) \big\}$.
First, we will show that $\boldsymbol{e}_2' (\x)$ is well-defined, i.e., $\v (\x) \neq \mathbf{0}$ if $\x \in \T$. Note that $\v (\x) = \mathbf{0}$ if and only if $\x = \x_1^* + k \boldsymbol{e}_1$ for some $k \in \R$. In fact, this is equivalent to $\x = (k', \mathbf{0})$ for some $k' \in \R$ since $\x_1^* = (-r, \mathbf{0})$.
Suppose $\x = (x_1, \mathbf{0})$. Consider the following cases.
\begin{itemize}
    \item Consider $x_1 \in (-\infty, -r) \cup (r, \infty )$. Then, we have $\psi(\x) = \pi$ (since $\alpha_1 (\x) = \alpha_2 (\x) = \pi$ when $x_1 \in (-\infty, r)$, and $\alpha_1 (\x) = \alpha_2 (\x) = 0$ when $x_1 \in (r, \infty )$). This implies that $\tilphi_1 (\x) = \tilphi_2 (\x) = \frac{\pi}{2}$ since $\tilphi_i \in \big[ 0, \frac{\pi}{2} \big]$ for $i \in \{1, 2\}$. Using the definition of $\tilphi_i$ for $i \in \{ 1, 2 \}$ in \eqref{def: phi tilde}, we obtain that $\x = \x_1^*$ and $\x = \x_2^*$ which contradicts the fact that $r = \frac{1}{2} \| \x_2^* - \x_1^* \| > 0$.
    
    \item Consider $x_1 \in (-r, r)$. Then, we have $\psi (\x) = 0$ (since $\alpha_1 (\x) = 0$ and $\alpha_1 (\x) = \pi$). This implies that $\tilphi_1 (\x) = \tilphi_2 (\x) = 0$ since $\tilphi_i \in \big[ 0, \frac{\pi}{2} \big]$ for $i \in \{1, 2\}$. Using the definition of $\tilphi_i$ for $\in \{ 1, 2 \}$ in \eqref{def: phi tilde}, we obtain that $| x_1 + r | = \frac{L}{\sigma_1}$ and $| x_1 - r | = \frac{L}{\sigma_2}$. Since $x_1 \in (-r, r)$, we have $x_1 = -r + \frac{L}{\sigma_1}$ and $x_1 = r - \frac{L}{\sigma_2}$. This implies that $r = \frac{L}{2} \big( \frac{1}{\sigma_1} + \frac{1}{\sigma_2} \big)$ which contradict our assumption.
    \item Consider $x_1 = -r$ or $x_1 = r$. We have that $\x = \x_1^*$ or $\x = \x_2^*$. However, $\x_1^* \notin \T$ and $\x_2^* \notin \T$ by the definition of $\T$.
\end{itemize}
Since $\x \neq (k, \mathbf{0})$ for all $k \in \R$, we conclude that $\boldsymbol{e}_2' (\x)$ is well-defined.

Next, we will verify that for $\x = (x_1, \tilde{\mathbf{x}}) \in \T$ and $i \in \{1, 2\}$, the statements $ - 2 d_i (\x) \cdot \sin  \frac{1}{2} ( \tilphi_1 (\x) + \tilphi_2 (\x) ) < 0$ and  $-\| \tilde{\mathbf{x}} \| \cdot \csc ( \theta (\x) ) < 0$, which are used to specify a range of $\epsilon$'s value in our claim, hold. Since $\x \notin \{ \x_1^*, \x_2^* \}$, we have $d_i (\x) > 0$ for $i \in \{1, 2\}$. 
In addition, from the above analysis, since $\x \notin \text{span} \{ \boldsymbol{e}_1 \}$, we have $\tilphi_i (\x) \in \big( 0, \frac{\pi}{2} \big)$ for $i \in \{ 1, 2 \}$. Therefore, it holds that $\frac{1}{2} ( \tilphi_1 (\x) + \tilphi_2 (\x)) \in \big( 0, \frac{\pi}{2} \big)$ and then combining the two facts, we verify the first statement.
Since $\x \notin \text{span} \{ \boldsymbol{e}_1 \}$ from the analysis above, we have $\| \tilde{\mathbf{x}} \| > 0$, and $\alpha_i (\x) \in (0, \pi)$ for $i \in \{1, 2 \}$ (from the definition of $\alpha_i$ in \eqref{def: alpha_i}). Therefore, we have $\csc (\theta (\x)) = \csc \big( \frac{1}{2} ( \alpha_1 (\x) + \alpha_2 (\x) ) \big) \in [1, \infty )$. Combining the two facts, we verify the second statement.

Then, given a point $\x \in \T$, we consider the new set of bases $\mathcal{J} (\x) = \{ \boldsymbol{e}_1' (\x), \boldsymbol{e}_2' (\x),$ $\ldots, \boldsymbol{e}_n' (\x) \}$ where $\boldsymbol{e}_1' (\x) = \boldsymbol{e}_1$, $\boldsymbol{e}_2' (\x) = \frac{\v (\x)}{\| \v (\x) \|}$ (as discussed at the beginning of the proof), and $\boldsymbol{e}_3' (\x), \boldsymbol{e}_4' (\x), \ldots,$ $\boldsymbol{e}_n' (\x)$ can be obtained by applying Gram-Schmidt procedure to the set of standard bases. 
Note that $\langle \boldsymbol{e}_1, \boldsymbol{e}_2' (\x) \rangle = 0$ by our construction. We denote $( \cdot, \cdot, \ldots, \cdot)_{\mathcal{J} (\x)}$ as the coordinate (or vector's component) with respect to the set of new bases. Let $\hat{x}_2 = \| \tilde{\mathbf{x}} \| \in \R_{>0}$. Then, by our construction, we have $\x = (x_1, \hat{x}_2, \mathbf{0})_{\mathcal{J} (\x)}$. Let the point 
\begin{equation*}
    \y = \big(x_1 - \hat{x}_2 \cot{ \theta (\x) }, \; \mathbf{0} \big)_{\mathcal{J}(\x)}.
\end{equation*}
Recall that $\x_{\epsilon}'$ is defined as given in \eqref{eqn: x epsilon}. We can write $\y = \x'_{\epsilon}$ with $\epsilon = -\hat{x}_2 \csc \theta (\x)$. We will show that $y_1 = x_1 - \hat{x}_2 \cot{ \theta (\x) } \in (-r, r)$. Since $\alpha_2 (\x) > \alpha_1 (\x)$ (by the fact that $\x \notin \text{span} \{ \boldsymbol{e}_1 \}$), we have
\begin{equation*}
    \cot \alpha_1 (\x) > \cot \theta (\x),
    \quad \Leftrightarrow \quad
    \frac{x_1 + r}{\hat{x}_2} > \cot \theta(\x),
    \quad \Leftrightarrow \quad
    x_1 - \hat{x}_2 \cot{ \theta (\x) } > -r.
\end{equation*}
The last inequality is due to $\hat{x}_2 > 0$. Similarly, since $\alpha_2 (\x) > \alpha_1 (\x)$ (by the fact that $\x \notin \text{span} \{ \boldsymbol{e}_1 \}$), we have
\begin{equation*}
    \cot \alpha_2 (\x) < \cot \theta (\x),
    \quad \Leftrightarrow \quad
    \frac{x_1 - r}{\hat{x}_2} < \cot \theta(\x),
    \quad \Leftrightarrow \quad
    x_1 - \hat{x}_2 \cot{ \theta (\x) } < r.
\end{equation*}

For convenience, let 
\begin{equation}
    a (\x) = \max \Big\{ \max_{i \in \{ 1,2 \}} - 2 d_i (\x) \sin  \frac{1}{2} \big( \tilphi_1 (\x) + \tilphi_2 (\x) \big), \; - \| \tilde{\mathbf{x}} \| \csc  \theta (\x)  \Big\}.
    \label{var: a(x)}
\end{equation}
Next, we will show that for all $\epsilon \in ( a(\x), \; 0 )$, we have $\alpha_1 (\x'_{\epsilon}) < \alpha_1 (\x )$ and $\alpha_2 (\x'_{\epsilon}) > \alpha_2 (\x )$, and
for all $\epsilon \in (0, \infty )$, we have $\alpha_1 (\x'_{\epsilon}) > \alpha_1 (\x )$ and $\alpha_2 (\x'_{\epsilon}) < \alpha_2 (\x )$ which will lead to proving parts~\ref{lem: boundary part1} and \ref{lem: boundary inclusion part2}, respectively, of our claim.
Consider the triangle formed by the points $\x_1^*$, $\x$, and $\y$. Suppose $\epsilon \in ( a(\x), \; 0 )$. Then, the point $\x'_{\epsilon}$ is on the line segment connecting $\y$ and $\x$. Since for a given $\z \in \R^n \setminus \{ \x_1^* \}$, $\alpha_1 (\z) = \angle (\z - \x_1^*, \; \boldsymbol{e}_1)$ (from the definition of $\alpha_1$ in \eqref{def: alpha_i}) and $y_1 > -r$ (with all other entries of $\y$ are zeros by its definition), we have $\alpha_1 (\x'_{\epsilon}) < \alpha_1 (\x)$. 
Next, suppose $\epsilon \in (0, \infty )$. Then, the point $\x'_{\epsilon}$ is on the ray from $\y$ to $\x$ but not on the line segment connecting $\y$ and $\x$ which implies that $\alpha_1 (\x'_{\epsilon}) > \alpha_1 (\x)$. 
Now, consider the triangle formed by the points $\x_2^*$, $\x$, and $\y$. By using similar argument as before (with the fact that $\alpha_2 (\z) = \angle (\z - \x_2^*, \; \boldsymbol{e}_1)$ for a given $\z \in \R^n \setminus \{ \x_2^* \}$, and $y_1 < r$), we can verify that $\alpha_2 (\x'_{\epsilon}) < \alpha_2 (\x )$.

We will show that for all $\epsilon \in ( a(\x), \; 0 )$ and $i \in \{1 ,2 \}$, we have $\| \x'_{\epsilon} - \x_i^* \| < \| \x - \x_i^* \|$, and for all $\epsilon \in (0, \infty )$ for $i \in \{1 ,2 \}$, we have $\| \x'_{\epsilon} - \x_i^* \| > \| \x - \x_i^* \|$.
Suppose $\epsilon \in ( a(\x), \; 0 )$. Then, we have $\epsilon > - 2 d_1 (\x) \sin  \frac{1}{2} \big( \tilphi_1 (\x) + \tilphi_2 (\x) \big)$. Using $\tilphi_1 (\x) + \tilphi_2 (\x) = \pi - (\alpha_2 (\x) - \alpha_1 (\x))$ and $\sin \big( \frac{\pi}{2} - z \big) = \cos z$ for all $z \in \R$, we can write 
\begin{equation*}
    \epsilon > -2 d_1(\x) \cos \frac{1}{2} \big( \alpha_2 (\x) - \alpha_1 (\x) \big),
    \quad \Leftrightarrow \quad
    \epsilon + 2 d_1 (\x) \cos \big( \theta (\x) - \alpha_1 (\x) \big) > 0.
\end{equation*}
Expanding the cosine and using $d_1 (\x) \cos \alpha_1 (\x) = x_1 +r$ and $d_1 (\x) \sin \alpha_1 (\x) = \hat{x}_2$ from \eqref{eqn: cos alpha} and \eqref{eqn: sin alpha}, respectively, we obtain $\epsilon + 2(x_1 + r) \cos \theta (\x) + 2 \hat{x}_2 \sin \theta (\x) > 0$. Multiplying the inequality by $\epsilon$ (which is negative) and then adding $(x_1 + r )^2 + \hat{x}_2^2$ to both sides, we get
\begin{equation*}
    (x_1 + \epsilon \cos \theta (\x) +r )^2 
    + (\hat{x}_2 + \epsilon \sin \theta (\x) )^2
    < (x_1 + r )^2 + \hat{x}_2^2,
\end{equation*}
which is $\| \x'_{\epsilon} - \x_1^* \| < \| \x - \x_1^* \|$. 
Next, suppose $\epsilon \in (0, \infty )$. Recall the definition of $d_i(\x)$ for $i \in \{1, 2 \}$ in \eqref{def: distance}. Since $0 < \alpha_1 (\x) \leq \alpha_2 (\x) < \pi$ (by the fact that $\x \notin \text{span} \{ \boldsymbol{e}_1 \}$), we have
\begin{equation*}
    \cos \big( \theta (\x) - \alpha_1 (\x) \big) > 0,
    \quad \Leftrightarrow \quad
    d_1 (\x) \cos \alpha_1(\x) \cos \theta(\x) 
    + d_1 (\x) \sin \alpha_1(\x) \sin \theta(\x) > 0.
\end{equation*}
Since $d_1 (\x) \cos \alpha_1 (\x) = x_1 +r$ and $d_1 (\x) \sin \alpha_1 (\x) = \hat{x}_2$ from \eqref{eqn: cos alpha} and \eqref{eqn: sin alpha}, respectively, we obtain $(x_1 + r) \cos \theta(\x) + \hat{x}_2 \sin \theta(\x) > 0$. Using the assumption that $\epsilon > 0$, we can write
\begin{equation*}
    \frac{\epsilon}{2} + (x_1 + r) \cos \theta(\x) + \hat{x}_2 \sin \theta(\x) > 0.
\end{equation*}
Multiplying the above inequality by $2 \epsilon$ (which is positive), then adding $(x_1 + r )^2 + \hat{x}_2^2$ to both sides, and rearranging, we get $\| \x'_{\epsilon} - \x_1^* \| > \| \x - \x_1^* \|$. By using similar steps as above, we can also show that $\| \x'_{\epsilon} - \x_2^* \| < \| \x - \x_2^* \|$ for $\epsilon \in ( a(\x), \; 0 )$ and $\| \x'_{\epsilon} - \x_2^* \| > \| \x - \x_2^* \|$ for $\epsilon \in (0, \infty )$.

Here, we will prove part~\ref{lem: boundary part1} of our claim. Suppose $\epsilon \in ( a(\x), \; 0 )$. From the above analysis, we have $\alpha_1 (\x'_{\epsilon}) < \alpha_1 (\x )$, $\alpha_2 (\x'_{\epsilon}) > \alpha_2 (\x )$ and $\| \x'_{\epsilon} - \x_i^* \| < \| \x - \x_i^* \|$ for $i \in \{ 1, 2 \}$.
Since $\alpha_1 (\x'_{\epsilon}) < \alpha_1 (\x )$ and $\alpha_2 (\x'_{\epsilon}) > \alpha_2 (\x )$, we obtain that $\psi (\x'_{\epsilon}) < \psi (\x)$ by the definition of $\psi$ in \eqref{def: psi}. For $i \in \{1 ,2 \}$, since $\| \x'_{\epsilon} - \x_i^* \| < \| \x - \x_i^* \|$, we obtain that $\tilphi_i (\x'_{\epsilon}) > \tilphi_i (\x)$ by the definition of $\tilphi_i$ in \eqref{def: phi tilde}. Therefore, it holds that
\begin{equation*}
    \tilphi_1 (\x'_{\epsilon}) + \tilphi_2 (\x'_{\epsilon})
    > \tilphi_1 (\x) + \tilphi_2 (\x) = \psi (\x) > \psi (\x'_{\epsilon}).
\end{equation*}

Here, we will prove part~\ref{lem: boundary part2} of our claim. Suppose $\epsilon \in (0, \infty )$. From the above analysis, we have $\alpha_1 (\x'_{\epsilon}) > \alpha_1 (\x )$, $\alpha_2 (\x'_{\epsilon}) < \alpha_2 (\x )$ and $\| \x'_{\epsilon} - \x_i^* \| > \| \x - \x_i^* \|$ for $i \in \{ 1, 2 \}$.
By using similar argument as the proof of part~\ref{lem: boundary part1}, we get $\psi (\x'_{\epsilon}) > \psi (\x)$. However, for $i \in \{1 ,2 \}$, $\| \x'_{\epsilon} - \x_i^* \| > \| \x - \x_i^* \|$ implies that either $\tilphi_i (\x'_{\epsilon}) < \tilphi_i (\x)$, or $\tilphi_i (\x'_{\epsilon})$ is not well-defined. In the case that $\tilphi_i (\x'_{\epsilon}) < \tilphi_i (\x)$ for $i \in \{ 1, 2 \}$, it holds that
\begin{equation*}
    \tilphi_1 (\x'_{\epsilon}) + \tilphi_2 (\x'_{\epsilon})
    < \tilphi_1 (\x) + \tilphi_2 (\x) = \psi (\x) < \psi (\x'_{\epsilon}).
\end{equation*}

Finally, suppose $\x \in \T$. Recall the quantity $a(\x)$ from \eqref{var: a(x)}. For $\delta \in \R_{>0}$, let $\x_{\text{in}} = \x_{\epsilon}'$ with $\epsilon = \max \big\{ a(\x), - \frac{\delta}{2} \big\}$ and $\x_{\text{out}} = \x_{\epsilon}'$ with $\epsilon = \frac{\delta}{2}$. 
From our claim, we have that for all $\delta \in \R_{>0}$, it holds that $\x_{\text{in}}, \x_{\text{out}} \in \B(\x, \delta)$, $\x_{\text{in}} \in \Mup$, $\x_{\text{in}} \in \Mdo$, $\x_{\text{out}} \in ( \Mup )^c$, and $\x_{\text{out}} \in ( \Mdo )^c$. Thus, we conclude that $\x \in \partial \Mup$ and $\x \in \partial \Mdo$.
\end{proof}

\subsection{Proof of Lemma~\ref{lem: boundary inclusion}}

\begin{proof}[Proof of Lemma~\ref{lem: boundary inclusion}]
Consider part~\ref{lem: boundary inclusion part1} of the lemma. From Lemma~\ref{lem: compare x1} part~\ref{lem: compare x1 part2}, we have $\lambda_1 < \frac{L}{\sigma_1} -r$. So, the interval $\big[ \lambda_1, \; -r + \frac{L}{\sigma_1} \big]$ is well-defined.
From the definition of $\beta$ and $\gamma_1$ in \eqref{def: gamma beta}, we have $(1 + \beta) \gamma_1 \geq 0$ and $-4 \beta r^2 < 0$. Combining the two inequalities, we get $(1 + \beta) \gamma_1 > -4 \beta r^2$. By subtracting $2r^2$ from both sides and then rearranging the inequality, we obtain
\begin{equation*}
    \lambda_1 = \Big(\frac{1 + \beta}{1 + 2 \beta} \Big) \frac{\gamma_1}{2 r} - \frac{r}{1 + 2 \beta} > -r. 
\end{equation*}
On the other hand, since $r > \frac{L}{2 \sigma_1}$, we get $-r + \frac{L}{\sigma_1} < r$. Combining the two inequalities yields $\big[ \lambda_1, \; -r + \frac{L}{\sigma_1} \big] \subseteq (-r, r)$. 

For $\x \in \partial \B_1 \cap \H_1^+$, we have $x_1 \in \big[ \lambda_1, \; -r + \frac{L}{\sigma_1} \big]$ (from \eqref{def: set B boundary}) which from above, implies that $\x \in \overline{\B}_1  \setminus \{ \x_1^*, \x_2^* \}$. 
It remains to show that $\x \in \overline{\B}_2$.
By examining the equation describing the set $\partial \B_1$ in \eqref{def: set B boundary} and the inequality describing the set $\overline{\B}_2 = \{ \z \in \R^n : (z_1 - r)^2 + \Vert \tilde{\mathbf{z}} \Vert^2 \leq \gamma_2 \}$ where $\gamma_2$ is defined in \eqref{def: gamma beta}, one can verify that
\begin{equation}
    \partial \B_1 \cap \overline{\B}_2
    = \Big\{ \z \in \partial \B_1: 
    z_1 \in \Big[ \frac{1}{4r} (\gamma_1 -\gamma_2), \; -r + \frac{L}{\sigma_1} \Big] \Big\}.
    \label{eqn: partial B1 cap B2}
\end{equation}
On the other hand, by performing some algebraic manipulation and using the definition of $\beta$ and $\gamma_i$ for $i \in \{ 1,2 \}$ in \eqref{def: gamma beta}, one can verify that
\begin{equation*}
    r \in \bigg( 0, \; \frac{L}{2} \Big( \frac{1}{\sigma_1} + \frac{1}{\sigma_2} \Big) \bigg),
    \quad \Leftrightarrow \quad
    \lambda_1 = \Big(\frac{1 + \beta}{1 + 2 \beta} \Big) \frac{\gamma_1}{2 r} - \frac{r}{1 + 2 \beta} > \frac{1}{4r} (\gamma_1 -\gamma_2).
\end{equation*}
From the definition of $\H_1^+$ in \eqref{def: set H_i} and equation \eqref{eqn: partial B1 cap B2}, this means that $\partial \B_1 \cap \H_1^+ \subset \partial \B_1 \cap \overline{\B}_2 \subseteq \overline{\B}_2$. For part~\ref{lem: boundary inclusion part2} of the lemma, we can proceed in a similar manner as the above analysis.
\end{proof}

%% file: contents/supp-inner.tex
\section{Proofs of theoretical results for inner approximation}  \label{sec: proof lemma inner}

\subsection{Proof of Proposition~\ref{prop: function exist n-dim}}
\label{subsec: proof quad}

Before proving Proposition~\ref{prop: function exist n-dim}, we consider an equivalent condition for the existence of a quadratic function with two independent variables satisfying certain properties.

\begin{lemma}
Let $\Q$ be defined as in \eqref{def: Quadratic Coll}.
Suppose we are given $\x^* \in \R^2$, $\x_0 \in \R^2$ such that $\x_0 \neq \x^*$, $\g \in \R^2$, and $\sigma \in \R_{>0}$.  Then, there exists a function $f \in  \Q^{(2)}(\x^*, \sigma)$ with a gradient $\nabla f(\x_0) = \g$ if and only if 
\begin{enumerate}[label=(\roman*)]
\item \label{lem: function exist part 1} $\x_0 \in \overline{\B} \big( \x^*, \frac{\| \g \|}{\sigma} \big)$ and
\item \label{lem: function exist part 2} $\angle ( \g, \; \x_0 - \x^*) \in \{ 0 \} \cup \Big[0, \; \arccos \big( \frac{\sigma}{\Vert \g \Vert} \Vert \x_0 - \x^* \Vert \big) \Big)$.
\end{enumerate}
Note that if $\sigma \Vert \x_0 - \x^* \Vert = \Vert \g \Vert$, then $\Big[0, \; \arccos (\frac{\sigma}{\Vert \g \Vert} \Vert \x_0 - \x^* \Vert) \Big) = \emptyset$.
\label{lem: function exist}
\end{lemma}

\begin{proof}
For the forward direction, suppose that $f \in \Q^{(2)} (\x^*, \sigma)$ and $\nabla f(\x_0) = \g$. Since $\Q(\x^*, \sigma) \subset \S(\x^*, \sigma)$, from \eqref{def: strongly cvx}, we have
\begin{equation}
    \| \g \| \; \| \x_0 - \x^* \|
    \geq \langle \nabla f(\x_0), \; \x_0 - \x^* \rangle 
    \geq \sigma \| \x_0 - \x^* \|^2.
    \label{eqn: strongly cvx}
\end{equation}
We then have $\| \x_0 - \x^* \| \leq \frac{ \| \g \|}{\sigma}$ (i.e., $\x_0 \in \overline{\B} \big( \x^*, \frac{\| \g \|}{\sigma} \big)$) which corresponds to part~\ref{lem: function exist part 1}. On the other hand, we can rewrite the second inequality of \eqref{eqn: strongly cvx} as
\begin{equation*}
    \| \g \|  \cos \angle (\g, \; \x_0 - \x^*) \geq \sigma \| \x_0 - \x^* \| > 0,
\end{equation*}
which implies that $\| \g \| \in \R_{>0}$ and $\angle (\g, \; \x_0 - \x^*) \in \big[ 0, \frac{\pi}{2} \big)$.

Since $f \in \Q^{(2)} (\x^*, \sigma)$, we can write $f(\x) = \frac{1}{2} \x^\intercal \boldsymbol{P} \x + \b^\intercal \x + c$ 
where $\boldsymbol{P} = 
\begin{bmatrix}
p_{11}       & p_{12} \\
p_{12}       & p_{22} 
\end{bmatrix} 
\in \mathsf{S}^2$, $\b \in \R^2$, and $c \in \R$. The gradient of $f$ is $\nabla f(\x) = \boldsymbol{P} \x + \b$. Since $\x^*$ is the minimizer of the quadratic function, by substituting $\x^*$ into the gradient equation and using $\nabla f(\x^*) = \mathbf{0}$, we get $\b = - \boldsymbol{P} \x^*$, and we can rewrite the gradient as $\g = \nabla f(\x_0) = \boldsymbol{P} (\x_0 - \x^*)$.
Let $\v = \x_0 - \x^*$ and then rewrite the gradient equation into two equations as follows:
\begin{equation}
\begin{cases} 
    g_1 = p_{11} v_1 + p_{12} v_2, \\ 
    g_2 = p_{12} v_1 + p_{22} v_2,
\end{cases} 
\quad \Leftrightarrow \quad
\begin{cases} 
    p_{12} v_2 = -p_{11} v_1 + g_1, \\ 
    p_{22} v_2^2 = v_1 (p_{11}v_1 - g_1) + g_2 v_2. 
\end{cases}
\label{eqn: g_i}
\end{equation} 
Since $\sigma$ is an eigenvalue of matrix $\boldsymbol{P}$, by expanding the equation $\det (\sigma \boldsymbol{I} - \boldsymbol{P}) = 0$ and rearranging the resulting equation, we get 
\begin{equation}
    p_{12}^2 = \sigma^2 - (p_{11}+p_{22}) \sigma + p_{11}p_{22}. \label{eqn: eigen eq}
\end{equation}
Multiply \eqref{eqn: eigen eq} by $v_2^2$, then substitute $p_{12} v_2$ and $p_{22} v_2^2$ from \eqref{eqn: g_i} into the resulting equation, and rearrange it to get
\begin{equation}
    (\sigma v_1^2 - g_1 v_1 + \sigma v_2^2 - g_2 v_2) p_{11} 
    = \sigma g_1 v_1 + \sigma^2 v_2^2 - \sigma g_2 v_2 - g_1^2.
\label{eqn: p11 old}
\end{equation}

Let the function $\boldsymbol{R}: (-\pi, \pi] \to \R^{2 \times 2}$ be such that $\boldsymbol{R}(\theta) =
\begin{bmatrix}
\cos \theta  & -\sin \theta \\
\sin \theta  & \cos \theta  
\end{bmatrix}$ (i.e., a rotation matrix),
$d = \| \v \| = \| \x_0 - \x^* \|$,
and $\phi = \measuredangle ( \g, \; \v)$ where $\measuredangle (\cdot, \cdot)$ is defined in \eqref{def: angle func}.
Since $\angle (\g, \v) \in \big[ 0, \frac{\pi}{2} \big)$ from the above analysis, we have that $\phi \in \big( - \frac{\pi}{2}, \frac{\pi}{2} \big)$ and we can decompose gradient $\g$ as
\begin{equation}
    \g 
    = \| \g \| \; \boldsymbol{R}(\phi) \; \Big( \frac{\x_0 - \x^*}{\| \x_0 - \x^* \|} \Big)
    = \frac{\| \g \|}{d} \; \boldsymbol{R}(\phi) \; \v.
    \label{eqn: gradient repar}
\end{equation}
For simplicity of notations, we let $\hat{L} = \frac{\Vert \g \Vert }{ d }$ which can be viewed as the norm of the gradient at $\x_0$ (i.e., $\| \nabla f(\x_0) \|$) normalized by distance from the minimizer (i.e., $\| \x_0 - \x^* \|$). Note that since $\| \g \| > 0$ and $d > 0$, we have $\hat{L} > 0$. Then, rewrite the expression \eqref{eqn: gradient repar} into two equations as follows:
\begin{equation}
\begin{cases} 
    g_1 = \hat{L} (v_1 \cos \phi - v_2 \sin \phi),  \\ 
    g_2 = \hat{L} (v_1 \sin \phi + v_2 \cos \phi). 
\end{cases}
\label{eqn: g_i repar}
\end{equation} 
We then substitute \eqref{eqn: g_i repar} into \eqref{eqn: p11 old} to get
\begin{equation}
    d^2 ( \hat{L} \cos \phi - \sigma) p_{11} 
    = ( \hat{L} \cos \phi - \sigma) (\sigma v_2^2 - 2 \hat{L} v_1 v_2 \sin \phi + \hat{L} v_1^2 \cos \phi) 
    + (\hat{L} v_2 \sin \phi)^2.
    \label{eqn: p11}
\end{equation}
Notice the symmetry inherent in the equations that if we interchange the indices between $1$ and $2$ for all the quantities (with $p_{12} = p_{21}$) in equations \eqref{eqn: g_i}, \eqref{eqn: eigen eq}, and \eqref{eqn: g_i repar}, and simultaneously substitute $\phi$ in \eqref{eqn: g_i repar} with $-\phi$, we find that the resulting set of equations remains unchanged. Leveraging this symmetry, we apply it to \eqref{eqn: p11}, leading to
\begin{equation}
    d^2 ( \hat{L} \cos \phi - \sigma) p_{22} 
    = ( \hat{L} \cos \phi - \sigma) (\sigma v_1^2 + 2 \hat{L} v_1 v_2 \sin \phi + \hat{L} v_2^2 \cos \phi) 
    + (\hat{L} v_1 \sin \phi)^2.
    \label{eqn: p22}
\end{equation}
We then add \eqref{eqn: p11} to \eqref{eqn: p22} and simplify the expression using the fact that $\| \v \|^2 = v_1^2 + v_2^2 = d^2$ to obtain
\begin{equation}
    (\hat{L} \cos \phi - \sigma) (p_{11} + p_{22}) 
    = \hat{L}^2 - \sigma^2.
    \label{eqn: p11+p22}
\end{equation}
Let $\mu \in \R$ be the other eigenvalue of matrix $\boldsymbol{P}$. We want to compute $\mu$ in terms of $\sigma$, $\hat{L}$, and $\phi$. Since $\text{Tr}(\boldsymbol{P}) = p_{11} + p_{22} = \sigma + \mu$, using equation \eqref{eqn: p11+p22}, we have 
\begin{equation}
    (\hat{L} \cos \phi - \sigma) \mu 
    = \hat{L} ( \hat{L} - \sigma \cos \phi).
    \label{eqn: maxeig}
\end{equation}
Consider the following cases on the term $\hat{L} \cos \phi - \sigma$.
\begin{itemize}
    \item Suppose $\sigma = \hat{L} \cos \phi$. Then, \eqref{eqn: maxeig} becomes $0 = \hat{L}^2 (1 - \cos^2 \phi)$ which implies that $\phi = 0$.
    \item Suppose $\sigma \neq \hat{L} \cos \phi$. Then, we can rewrite equation \eqref{eqn: maxeig} as
    \begin{equation}
        \mu = \frac{\hat{L} ( \hat{L} - \sigma \cos \phi)}{\hat{L} \cos \phi - \sigma}.
        \label{eqn: maxeig explicit}
    \end{equation}
    Since $\lambda_{\text{min}} (\boldsymbol{P}) = \sigma$, we have that $\mu \geq \sigma$ where $\mu$ is given in \eqref{eqn: maxeig explicit}. Specifically, we will show that
    \begin{multline}
        \text{given} \; \mu \; \text{as expressed in \eqref{eqn: maxeig explicit} with} \; \hat{L} > 0 \; \text{and} \; \sigma > 0, \\
      \text{it holds that} \; \mu \geq \sigma \; \text{if and only if} \; \hat{L} \cos \phi - \sigma > 0.
      \label{state: maxeig}
    \end{multline}
    For the forward direction, suppose $\mu \geq \sigma$ and $\hat{L} \cos \phi - \sigma < 0$. We have
    \begin{equation*}
        \mu = \frac{ \hat{L} (\hat{L} - \sigma \cos \phi) }{ \hat{L} \cos \phi - \sigma } \geq \sigma,
        \quad \Leftrightarrow \quad
        \cos \phi \geq \frac{1}{2} \Big( \frac{\hat{L}}{\sigma} + \frac{\sigma}{\hat{L}} \Big).
    \end{equation*}
    However, since $\hat{L} > 0$ and $\sigma > 0$, we get $\frac{1}{2} \big( \frac{\hat{L}}{\sigma} + \frac{\sigma}{\hat{L}}  \big) \geq 1$. Therefore, we obtain that $\cos \phi = \frac{1}{2} \big( \frac{\hat{L}}{\sigma} + \frac{\sigma}{\hat{L}}  \big) = 1$ which implies that $\phi = 0$ and $\hat{L} = \sigma$. However, we get $\hat{L} \cos \phi - \sigma = 0$ which makes $\mu$ undefined. 
    For the converse, suppose $\mu < \sigma$ and $\hat{L} \cos \phi - \sigma > 0$, we have 
    \begin{equation*}
        \mu = \frac{ \hat{L} (\hat{L} - \sigma \cos \phi) }{ \hat{L} \cos \phi - \sigma } < \sigma,
        \quad \Leftrightarrow \quad
        \cos \phi > \frac{1}{2} \Big( \frac{\hat{L}}{\sigma} + \frac{\sigma}{\hat{L}} \Big).
    \end{equation*}
    However, this is not possible since $\frac{1}{2} \big( \frac{\hat{L}}{\sigma} + \frac{\sigma}{\hat{L}}  \big) \geq 1$ for $\frac{\sigma}{\hat{L}} > 0$ and we have proved the claim.
\end{itemize}
Since $\| \x_0 - \x^* \| \in \big( 0, \; \frac{ \| \g \|}{\sigma} \big]$ (or equivalently $\frac{\sigma d}{ \| \g \|} \in (0, 1]$), the expression $\arccos \big( \frac{\sigma}{\hat{L}} \big)$ is well-defined. 
Combining the two cases ($\sigma = \hat{L} \cos \phi$ and $\sigma \neq \hat{L} \cos \phi$), we have shown that if there exists $f \in \Q(\x^*, \sigma)$ with the gradient $\nabla f(\x_0) = \g$ then from the definition of $\measuredangle (\cdot, \cdot)$ and $\angle (\cdot, \cdot)$ in \eqref{def: angle func}, we have
\begin{align*}
    &\measuredangle (\g, \v)
    \in \{ 0 \} \cup \Big( -\arccos \Big( \frac{\sigma}{\hat{L}} \Big) , \; \arccos \Big( \frac{\sigma}{\hat{L}} \Big) \Big), \\
    \Leftrightarrow \quad
    &\angle (\g, \v) 
    \in \{ 0 \} \cup \Big[ 0, \; \arccos \Big( \frac{\sigma}{\hat{L}} \Big) \Big),
\end{align*}
which corresponds to part~\ref{lem: function exist part 2}.

For the converse, let $f (\x) = \frac{1}{2} \x^\intercal \boldsymbol{P} \x - (\x^*)^\intercal \boldsymbol{P} \x$ and each entry of $\boldsymbol{P}$ will be specified below.
First, suppose that $\sigma < \hat{L}$ (i.e., $\x_0 \in \B \big( \x^*, \frac{\| \g \|}{\sigma} \big)$) and $\phi = \angle (\g, \; \x_0 - \x^*) \in \big[ 0, \; \arccos{ \big(\frac{\sigma}{ \hat{L}} \big)} \big)$. This implies that $\hat{L} \cos \phi - \sigma > 0$.
If $v_2 \neq 0$, let 
\begin{equation}
    p_{12} = - \frac{v_1}{d^2 v_2} ( \hat{L} v_1^2 \cos \phi -2 \hat{L} v_1 v_2 \sin \phi - \hat{L} v_2^2 \cos \phi) 
    - \frac{ (\hat{L}^2 - \sigma^2) v_1 v_2 }{d^2 ( \hat{L} \cos \phi - \sigma)} + \hat{L} \Big( \frac{v_1}{v_2} \cos{\phi} -  \sin{\phi} \Big). 
    \label{eqn: p12}
\end{equation}
We choose $p_{11}$, $p_{12}$, and $p_{22}$ as given in \eqref{eqn: p11}, \eqref{eqn: p12}, and \eqref{eqn: p22}, respectively. Note that $p_{11}$, $p_{12}$, and $p_{22}$ are well-defined since $\hat{L} \cos \phi - \sigma \neq 0$.
If $v_2 = 0$, we choose 
\begin{equation*}
    p_{11} = \hat{L} \cos \phi, \quad
    p_{12} = \hat{L} \sin \phi,
    \quad \text{and} \quad 
    p_{22} = \sigma + \frac{(\hat{L} \sin \phi)^2}{\hat{L} \cos \phi - \sigma}.
\end{equation*}
In both cases ($v_2 \neq 0$ and $v_2 = 0$), one can easily verify that $\nabla f(\x_0) = \boldsymbol{P} \v = \g$ and $\nabla f(\x^*) = \mathbf{0}$. In order to show that $\lambda_{\text{min}} (\boldsymbol{P}) = \sigma$, first, we check that $p_{11}$, $p_{12}$, and $p_{22}$ satisfy \eqref{eqn: eigen eq} which means that $\sigma$ is an eigenvalue of $\boldsymbol{P}$.
Next, using the fact that the other eigenvalue $\mu = \text{Tr} (\boldsymbol{P}) - \sigma = (p_{11} + p_{22}) - \sigma$, we obtain $\mu$ as expressed in \eqref{eqn: maxeig explicit} for both cases. Since $\hat{L} \cos \phi - \sigma > 0$, from the statement \eqref{state: maxeig}, we have $\mu \geq \sigma$.

Now suppose that $\sigma = \hat{L}$ (i.e., $\x_0 \in \partial \B \big( \x^*, \frac{\| \g \|}{\sigma} \big)$) and $\phi = \angle (\g, \; \x_0 - \x^*) = 0$. 
Let the orthogonal matrix $\boldsymbol{T} = \begin{bmatrix}
\frac{\v}{\Vert \v \Vert} & \frac{\v_{\bot}}{\Vert \v_{\bot} \Vert}
\end{bmatrix} \in \R^{2 \times 2}$ where $\v_{\bot} \neq \mathbf{0}$ is a vector perpendicular to vector $\v$ and $\boldsymbol{P} = \sigma \boldsymbol{T} \boldsymbol{T}^\intercal = \sigma \boldsymbol{I}$.
We have $\nabla f(\x_0) = \boldsymbol{P} \v = \sigma \v = \| \g \| \cdot \frac{\v}{ \| \v \|}$. However, since $\angle ( \g, \; \v ) = 0$, we have $\nabla f(\x_0) = \g$. We also have $\nabla f(\x^*) = \mathbf{0}$ and $\lambda_{\text{min}} (\boldsymbol{P}) = \sigma$ by our construction of $f$. Thus, for both cases, we obtain that $f \in \Q^{(2)} (\x^*, \sigma)$ and $\nabla f(\x_0) = \g$.
\end{proof}

\begin{proof}[Proof of Proposition~\ref{prop: function exist n-dim}]
For given points $\x^* \in \R^n$, $\x_0 \in \R^n$ such that $\x_0 \neq \x^*$ and vector $\g \in \R^n$, let 
\begin{equation}
    \boldsymbol{e}_1' = \frac{ \x_0 - \x^*}{\| \x_0 - \x^* \|}, \quad
    \boldsymbol{e}_2' = 
    \begin{cases}
    \frac{\g - \langle \g, \boldsymbol{e}_1' \rangle  \boldsymbol{e}_1' }{\| \g - \langle \g, \boldsymbol{e}_1' \rangle  \boldsymbol{e}_1' \|}
    \quad &\text{if} \quad \g \notin \text{span}( \{\boldsymbol{e}_1' \}), \\
    \u \; \text{where} \; \u \in \mathcal{N} ( (\boldsymbol{e}_1')^\intercal )
    \;\; &\text{otherwise},
    \end{cases}
    \label{var: e_1' e_2'}
\end{equation}
and $\boldsymbol{E} = \begin{bmatrix}
\boldsymbol{e}_1' & \boldsymbol{e}_2'
\end{bmatrix} \in \R^{n \times 2}$.
We let $\x_{\g} = \x^* + \g$ to be the point generated by vector $\g$ and 
\begin{equation*}
    \mathcal{P}_2 = \{ \x \in \R^n : \x
    = \x^* + \boldsymbol{E} \boldsymbol{s}
    \quad \text{for some} \; \boldsymbol{s} \in \R^2 \}
    = \x^* + \mathcal{R}(\boldsymbol{E})
\end{equation*}
to be a $2$-D plane in $\R^n$. One can verify that $\boldsymbol{e}_1'$ and $\boldsymbol{e}_2'$ are orthonormal, and $\{ \x^*, \x_0, \x_{\g} \} \subset \mathcal{P}_2$.
We let $\boldsymbol{T}: \R^2 \to \mathcal{P}_2$ to be such that $\boldsymbol{T}( \boldsymbol{s} ) = \x^* + \boldsymbol{E} \boldsymbol{s}$. Since function $\boldsymbol{T}$ is bijective, we then have $\boldsymbol{T}^{-1}: \mathcal{P}_2 \to \R^2$ such that $\boldsymbol{T}^{-1} (\x) = \boldsymbol{E}^\intercal (\x - \x^*)$.
Next, consider a property of the norm and angle function which we will use in the subsequent analysis. Using $\boldsymbol{E}^\intercal \boldsymbol{E} = \boldsymbol{I}$ and that $\boldsymbol{T}$ is bijective, we have that for all $i \in \{1, 2, 3, 4 \}$, $\boldsymbol{s}_i \in \R^2$ and $\x_i \in \mathcal{P}_2$, 
\begin{itemize}
    \item \textit{distance invariance:} $\| \boldsymbol{T}( \boldsymbol{s}_1 ) - \boldsymbol{T}( \boldsymbol{s}_2 ) \| 
    = \| \boldsymbol{s}_1 - \boldsymbol{s}_2 \|$ 
    and $\| \boldsymbol{T}^{-1} ( \x_1 ) - \boldsymbol{T}^{-1} ( \x_2 ) \| 
    = \| \x_1 - \x_2 \|$, and
    \item \textit{angle invariance:} $\angle ( \boldsymbol{T}( \boldsymbol{s}_1 ) - \boldsymbol{T}( \boldsymbol{s}_2 ), \; \boldsymbol{T}( \boldsymbol{s}_3 ) - \boldsymbol{T}( \boldsymbol{s}_4 ))
    = \angle (\boldsymbol{s}_1- \boldsymbol{s}_2, \; \boldsymbol{s}_3 - \boldsymbol{s}_4)$ 
    and $\angle ( \boldsymbol{T}^{-1} ( \x_1 ) - \boldsymbol{T}^{-1} ( \x_2 ), \; \boldsymbol{T}^{-1} ( \x_3 ) - \boldsymbol{T}^{-1} ( \x_4 ))
    = \angle (\x_1- \x_2, \; \x_3 - \x_4)$.
\end{itemize}

For the forward direction, suppose $f \in \Q^{(n)} (\x^*, \sigma)$ and $\nabla f(\x_0) = \g$. Let $f( \x ) = \frac{1}{2} \x^\intercal  \boldsymbol{P} \x + \b^\intercal \x + c$ for some $\boldsymbol{P} \in \mathfrak{S}^n$, $\b \in \R^n$, and $c \in \R$ that satisfies the conditions. 
Let $\tilde{f}(\boldsymbol{s}) = f (\x^* + \boldsymbol{E} \boldsymbol{s})$. We then have
\begin{equation}
    \tilde{f}(\boldsymbol{s}) 
    = \frac{1}{2} \boldsymbol{s}^\intercal \boldsymbol{E}^\intercal \boldsymbol{PEs} 
    + (\boldsymbol{Px}^* + \b)^\intercal \boldsymbol{Es}
    + \Big(\frac{1}{2} ( \x^*)^\intercal \boldsymbol{Px}^* + \b^\intercal \x^* + c \Big),
    \label{eqn: f_tilde}
\end{equation}
and 
\begin{equation*}
    \nabla \tilde{f} ( \boldsymbol{T}^{-1} (\x_0) ) 
    = \boldsymbol{E}^\intercal \big( \boldsymbol{PE} \boldsymbol{T}^{-1} (\x_0) + \boldsymbol{Px}^* + \b \big)
    = \boldsymbol{E}^\intercal \big( \boldsymbol{P} \boldsymbol{T}( \boldsymbol{T}^{-1} (\x_0) ) + \b \big)
    = \boldsymbol{E}^\intercal \g,
\end{equation*}
where the second and third equalities are obtained by using the transformation $\boldsymbol{T} (\boldsymbol{s}) = \x^* + \boldsymbol{Es}$ at $\boldsymbol{s} = \boldsymbol{T}^{-1} (\x_0)$ and $\nabla f(\x_0) = \boldsymbol{Px}_0 + \b = \g$, respectively.
Similarly, we have $\nabla \tilde{f} ( \boldsymbol{T}^{-1} (\x^*) ) = \mathbf{0}$. Let $\tilde{\sigma}$ be the strong convexity parameter of $\tilde{f}$. 
We conclude that $\tilde{f} \in \Q^{(2)} (\boldsymbol{T}^{-1} (\x^*), \; \tilde{\sigma} )$ and $\nabla \tilde{f} ( \boldsymbol{T}^{-1} (\x_0) ) = \boldsymbol{E}^\intercal \g$. However, from \eqref{eqn: f_tilde}, we have that
\begin{equation}
    \tilde{\sigma} 
    = \lambda_{\text{min}}( \boldsymbol{E}^\intercal \boldsymbol{PE} ) 
    = \min_{\boldsymbol{s} \in \R^2} \frac{\boldsymbol{s}^\intercal ( \boldsymbol{E}^\intercal \boldsymbol{PE ) s}}{\boldsymbol{s}^\intercal \boldsymbol{s}}
    = \min_{\x \in \mathcal{R}(\boldsymbol{E}) } \frac{\x^\intercal \boldsymbol{Px}}{ \x^\intercal \x}
    \geq \min_{\x \in \R^n} \frac{\x^\intercal \boldsymbol{Px}}{ \x^\intercal \x}
    = \sigma.
    \label{eqn: sigma compare}
\end{equation}
Using Lemma~\ref{lem: function exist}, the distance invariance of $\boldsymbol{T}^{-1}$ and \eqref{eqn: sigma compare}, we have
\begin{equation*}
    \| \x_0 - \x^* \|
    = \| \boldsymbol{T}^{-1} (\x_0) - \boldsymbol{T}^{-1} (\x^*) \|
    \leq \frac{\| \boldsymbol{E}^\intercal \g \|}{\tilde{\sigma}}
    = \frac{\| \boldsymbol{T}^{-1}(\x_{\g}) - \boldsymbol{T}^{-1} (\x^*) \|}{\tilde{\sigma}}
    \leq \frac{\| \g \|}{\sigma}.
\end{equation*}
In addition, using Lemma~\ref{lem: function exist}, the angle invariance of $\boldsymbol{T}^{-1}$ and \eqref{eqn: sigma compare}, we have
\begin{align*}
    \angle (\g, \; \x_0 - \x^*) 
    &= \angle \big( \boldsymbol{T}^{-1}(\x_{\g}) - \boldsymbol{T}^{-1} (\x^*), \; \boldsymbol{T}^{-1} (\x_0) - \boldsymbol{T}^{-1} (\x^*) \big) \\
    &= \angle \big( \boldsymbol{E}^\intercal \g, \; \boldsymbol{T}^{-1} (\x_0) - \boldsymbol{T}^{-1} (\x^*)  \big) \\
    &\in \{ 0 \} \cup \bigg[ 0, \; \arccos \Big( \frac{\tilde{\sigma}}{ \| \boldsymbol{E}^\intercal \g \|} \| \boldsymbol{T}^{-1} (\x_0) - \boldsymbol{T}^{-1} (\x^*) \| \Big) \bigg) \\
    &\subseteq \{ 0 \} \cup \bigg[ 0, \; \arccos \Big( \frac{\sigma}{\| \g \|} \| \x_0 - \x^* \| \Big) \bigg),
\end{align*}
which complete the proof of the forward direction of the proposition.

For the converse, suppose $\| \x_0 - \x^* \| \leq \frac{\| \g \|}{\sigma}$ and $\angle ( \g, \x_0 - \x^*) \in \Big[0,  \arccos \big( \frac{\sigma}{\Vert \g \Vert} \cdot \Vert \x_0 - \x^* \Vert \big) \Big) \cup  \{ 0 \}$. 
Using similar techniques as above, we can write 
\begin{itemize}
    \item $\| \boldsymbol{T}^{-1} (\x_0) - \boldsymbol{T}^{-1} (\x^*) \| \leq \frac{\| \boldsymbol{E}^\intercal \g \|}{\sigma}$ and 
    \item $\angle ( \boldsymbol{E}^\intercal \g, \; \boldsymbol{T}^{-1} (\x_0) - \boldsymbol{T}^{-1} (\x^*) ) \in \Big[0, \; \arccos \big( \frac{\sigma}{\Vert \boldsymbol{E}^\intercal \g \Vert}  \Vert \boldsymbol{T}^{-1} (\x_0) - \boldsymbol{T}^{-1} (\x^*) \Vert \big) \Big) \cup  \{ 0 \}$. 
\end{itemize}
Using Lemma~\ref{lem: function exist}, we have that there exists a function $f \in \Q^{(2)} (\boldsymbol{T}^{-1} (\x^*), \; \sigma)$ with the gradient $\nabla f( \boldsymbol{T}^{-1} (\x_0) ) = \boldsymbol{E}^\intercal \g$.
Let $f: \R^2 \to \R$ be such that $f(\boldsymbol{s}) = \frac{1}{2} \boldsymbol{s}^\intercal \boldsymbol{P} \boldsymbol{s} + \b^\intercal \boldsymbol{s} + c$ where $\boldsymbol{P} \in \mathfrak{S}^2$, $\b \in \R^2$ and $c \in \R$. To satisfy the conditions of $f$, we have $\mathbf{0} = \nabla f( \boldsymbol{T}^{-1} (\x^*)) = \nabla f( \mathbf{0} ) = \b$,
\begin{equation}
    \boldsymbol{E}^\intercal \g 
    = \nabla f( \boldsymbol{T}^{-1} (\x_0) ) 
    = \nabla f( \boldsymbol{E}^\intercal (\x_0 - \x^*) )
    = \boldsymbol{PE}^\intercal (\x_0 - \x^*),
    \label{eqn: grad transform}
\end{equation}
and $\lambda_{\text{min}} (\boldsymbol{P}) = \sigma$. 

To construct a quadratic function $\tilde{f}: \R^n \to \R$ satisfying $\tilde{f} \in \Q^{(n)} (\x^*, \sigma)$ and $\nabla \tilde{f} (\x_0) = \g$, recall the expression of $\boldsymbol{e}_1'$ and $\boldsymbol{e}_2'$ from \eqref{var: e_1' e_2'}.
Let $\{ \tilde{\v}_1, \tilde{\v}_2, \ldots, \tilde{\v}_{n-2} \}$ be a set of unit vectors in $\R^n$ for which $\begin{bmatrix} \boldsymbol{e}_1' & \boldsymbol{e}_2' & \tilde{\v}_1 & \tilde{\v}_2 & \cdots & \tilde{\v}_{n-2} \end{bmatrix} \in \R^{n \times n}$ is orthogonal.
Let $\widetilde{\v} = 
\begin{bmatrix}
\tilde{\v}_1 & \tilde{\v}_2 & \cdots & \tilde{\v}_{n-2}
\end{bmatrix} \in \R^{n \times (n-2)}$,
$\widetilde{\boldsymbol{\Lambda}} = \text{diag} (\tilde{\sigma}_1, \tilde{\sigma}_2, \ldots, \tilde{\sigma}_{n-2}) \in \R^{(n-2) \times (n-2)}$ such that $\{ \tilde{\sigma}_1, \tilde{\sigma}_2, \ldots, \tilde{\sigma}_{n-2} \}$ is chosen to satisfy $\widetilde{\boldsymbol{\Lambda}} \succeq \sigma \boldsymbol{I}$, and
$\tilde{f}: \R^n \to \R$ be such that 
\begin{equation*}
    \tilde{f} (\x) 
    = \frac{1}{2} \x^\intercal \boldsymbol{Qx} 
    + (- \boldsymbol{Qx}^* + \boldsymbol{Eb})^\intercal \x 
    + \Big( \frac{1}{2} (\x^*)^\intercal \boldsymbol{Qx}^* - \b^\intercal \boldsymbol{E}^\intercal \x^* + c \Big),
\end{equation*}
where $\boldsymbol{Q} = \boldsymbol{EPE}^\intercal + \widetilde{\v} \widetilde{\boldsymbol{\Lambda}} \widetilde{\v}^\intercal \in \R^{n \times n}$.
Note that $\boldsymbol{E}^\intercal \boldsymbol{E} = \boldsymbol{I}$, $\widetilde{\v}^\intercal \widetilde{\v} = \boldsymbol{I}$ and $ \widetilde{\v}^\intercal \boldsymbol{E} = \mathbf{0}$. 
Since $\b = \mathbf{0}$, we also have $\nabla \tilde{f} (\x^*) = \boldsymbol{Q} \x^* + (- \boldsymbol{Q} \x^* + \boldsymbol{Eb}) = \mathbf{0}$ and $\nabla \tilde{f} (\x_0) = \boldsymbol{Q} (\x_0 - \x^*)$. Using
\begin{equation*}
    \boldsymbol{EE}^\intercal \g
    = \boldsymbol{E} \big( \boldsymbol{E}^\intercal (\x_{\g} - \x^*) \big)
    = \boldsymbol{E T}^{-1} (\x_{\g})
    = \boldsymbol{T} ( \boldsymbol{T}^{-1} (\x_{\g})) - \x^*
    = \g,
\end{equation*}
$\widetilde{\v}^\intercal (\x_0 - \x^*) = \mathbf{0}$ (since $\x_0 - \x^* \in \text{span} (\{ \boldsymbol{e}_1' \})$)
and \eqref{eqn: grad transform}, we can write 
\begin{equation*}
    \nabla \tilde{f} (\x_0) = \boldsymbol{Q} (\x_0 - \x^*)
    = \boldsymbol{EPE}^\intercal (\x_0 - \x^*) + \widetilde{\v} \widetilde{\boldsymbol{\Lambda}} \widetilde{\v}^\intercal (\x_0 - \x^*) 
    = \boldsymbol{EE}^\intercal \g 
    = \g.
\end{equation*}
It remains to show that $\lambda_{\text{min}} (\boldsymbol{Q}) = \sigma$.

Suppose $\boldsymbol{P} \in \mathfrak{S}^2$ can be decomposed as $\boldsymbol{P} = \boldsymbol{V \Lambda} \v^\intercal$ where $\v \in \R^{2 \times 2}$ is a matrix whose $i$-th column is the (unit) eigenvector $\v_i$ of $\boldsymbol{P}$, and $\boldsymbol{\Lambda}$ is the diagonal matrix whose diagonal elements are the corresponding eigenvalues.
We can rewrite $\boldsymbol{Q} = \boldsymbol{EPE}^\intercal + \widetilde{\v} \widetilde{\boldsymbol{\Lambda}} \widetilde{\v}^\intercal$ as
\begin{equation}
    \boldsymbol{Q} 
    = \boldsymbol{E} (\boldsymbol{V \Lambda} \v^\intercal) \boldsymbol{E}^\intercal + \widetilde{\v} \widetilde{\boldsymbol{\Lambda}} \widetilde{\v}^\intercal =
    \begin{bmatrix}
    \boldsymbol{E V} & \widetilde{\v}
    \end{bmatrix}
    \begin{bmatrix}
    \boldsymbol{\Lambda} & \mathbf{0} \\ 
    \mathbf{0} & \widetilde{\boldsymbol{\Lambda}}
    \end{bmatrix}
    \begin{bmatrix}
    \v^\intercal \boldsymbol{E}^\intercal \\ \widetilde{\v}^\intercal
    \end{bmatrix}.
    \label{eqn: eigendecom Q}
\end{equation}
Since $\begin{bmatrix} \boldsymbol{E V} & \widetilde{\v} \end{bmatrix}$ is an orthogonal matrix (using the fact that $(\boldsymbol{E V})^\intercal \widetilde{\v} = \mathbf{0}$ and $\| \boldsymbol{E v}_i \| = 1$ for $i \in \{1, 2\}$), we have that equation \eqref{eqn: eigendecom Q} is the eigendecomposition of $\boldsymbol{Q}$. By our construction of $\widetilde{\boldsymbol{\Lambda}}$ and $\lambda_{\text{min}} (\boldsymbol{P}) = \sigma$, we have $\lambda_{\text{min}} (\boldsymbol{Q}) = \sigma$. Thus, we conclude that $\tilde{f} \in \Q^{(n)} (\x^*, \sigma)$ with $\nabla \tilde{f} (\x_0) = \g$.
\end{proof}

\subsection{Proof of Theorem~\ref{thm: down BD}}
\label{subsec: proof inner}

\begin{proof}[Proof of Theorem~\ref{thm: down BD}]
\textbf{Part \ref{thm: down BD part1}:} $r \in \big( 0, \; \frac{L}{2 \sigma_1} \big]$. The proof of the characterization of the boundary $\partial \Mdo$ can be carried out in the same manner as in the proof of part~\ref{thm: up BD part1} of Theorem~\ref{thm: up BD}. We are left to show the property related to the set $\Mdo$ and the characterization of $( \Mdo)^{\circ}$. By the definition of $\Mdo$ in \eqref{def: set M down} and $\partial \Mdo = \T \sqcup \{ \x_1^*, \x_2^* \}$, we have that $\partial \Mdo \subset ( \Mdo )^c$. Therefore, we can conclude that $\Mdo$ is open. Since $\Mdo$ is open, $( \Mdo)^{\circ} = \Mdo = \widetilde{\T}$, where $\tilde{\T}$ is defined in \eqref{def: set T tilde}.

\textbf{Part \ref{thm: down BD part2}:} $r \in \big( \frac{L}{2 \sigma_1}, \frac{L}{2 \sigma_2} \big]$. Consider points in the set $(\H_1^-)^c$. By proceeding the same steps as in the proof of \eqref{eqn: T empty}, we obtain that 
\begin{equation}
    \T \cap (\H_1^-)^c = \emptyset.
    \label{eqn: T empty copy}
\end{equation}
By using the same reasoning as in the corresponding part of the proof of Theorem~\ref{thm: up BD} part~\ref{thm: up BD part2}, we obtain the results, which are similar to \eqref{eqn: M cap H} and \eqref{eqn: nc thm part2.1}, that
\begin{equation}
    \Mdo \cap (\H_1^-)^c = \overline{\B}_1 \cap (\H_1^-)^c
    \quad \text{and} \quad
    \partial \Mdo \cap (\H_1^-)^c 
    = \partial \B_1 \cap (\H_1^-)^c.
    \label{eqn: sc thm part2.1}
\end{equation}
Then, recall the definition of $\lambda_1$ and $\nu_1$ from \eqref{var: lambda} and \eqref{var: nu}, respectively. Consider points in the set $\H_1^+ \cap \H_1^-$, i.e., $\{ \z \in \R^n : z_1 = \lambda_1 \}$. Considering three disjoint regions in $\H_1^+ \cap \H_1^-$: $\| \tilde{\mathbf{z}} \| > \nu_1$, $\| \tilde{\mathbf{z}} \| < \nu_1$, and $\| \tilde{\mathbf{z}} \| = \nu_1$, we obtain similar results as in the corresponding part of the proof of Theorem~\ref{thm: up BD} part~\ref{thm: up BD part2} summarized as follows:
\begin{equation}
    \begin{cases}
    \{ \z \in \R^n : z_1 = \lambda_1, \; \| \tilde{\mathbf{z}} \| > \nu_1 \} \subseteq ((\Mdo)^c)^{\circ} \cap \T^c, \\
    \{ \z \in \R^n : z_1 = \lambda_1, \; \| \tilde{\mathbf{z}} \| < \nu_1 \} \subseteq (\Mdo)^{\circ} \cap \T^c, \\
    \{ \z \in \R^n : z_1 = \lambda_1, \; \| \tilde{\mathbf{z}} \| = \nu_1 \} = \C_1 \subseteq \partial \Mdo \cap \T.
    \end{cases}
    \label{eqn: region z1 = lambd1}
\end{equation}
Combining these results, we get a similar result as in \eqref{eqn: nc thm part2.2}, i.e.,
\begin{equation}
    \partial \Mdo \cap ( \H_1^+ \cap \H_1^-) 
    = \C_1
    = \T \cap ( \H_1^+ \cap \H_1^-).
    \label{eqn: sc thm part2.2}
\end{equation}
From the first equation in \eqref{eqn: sc thm part2.1}, \eqref{eqn: region z1 = lambd1}, and $\C_1 \subseteq (\Mdo)^c$ (since $\C_1 \subseteq \T$), we have
\begin{equation}
\begin{aligned}
    \Mdo \cap \H_1^+
    &= \big[ \Mdo \cap (\H_1^-)^c \big] \cup \big[ \Mdo \cap ( \H_1^+ \cap \H_1^- ) \big] \\
    &= \big[ \overline{\B}_1 \cap (\H_1^-)^c \big] \cup \{ \z \in \R^n : z_1 = \lambda_1, \; \| \tilde{\mathbf{z}} \| < \nu_1 \}.
\end{aligned}
\label{eqn: M down RHS}
\end{equation}
Next, consider points in the set $(\H_1^+)^c$. Recall the definition of $\varphi$ from \eqref{def: varphi}. By using the same reasoning as in the proof of \eqref{eqn: part 2 claim 2} and \eqref{eqn: nc thm part2.3}, we obtain that
\begin{equation}
    \partial (\B_1 \cap \B_2) \cap (\H_1^+)^c 
    \subseteq \{ \z \in \R^n :
    \varphi(\z) < 0 \} \cup \{ \x_1^* \},
    \label{eqn: bd + half plane}
\end{equation}
and
\begin{equation}
    \partial \Mdo \cap (\H_1^+)^c 
    = \big[  \T  \cap ( \H_1^+ )^c \big] \sqcup \{ \x^*_1 \},
    \label{eqn: sc thm part2.3}
\end{equation}
respectively. Using \eqref{eqn: T empty copy}, we can write $\T = \big[ \T \cap ( \H_1^+ \cap \H_1^-) \big] \sqcup \big[ \T \cap ( \H_1^+ )^c \big]$, and combining the second equation in \eqref{eqn: sc thm part2.1}, equation \eqref{eqn: sc thm part2.2} and equation \eqref{eqn: sc thm part2.3} together yields the characterization of $\partial \Mdo$:
\begin{equation*}
    \partial \Mdo 
    = \big[ \partial \B_1 \cap (\H_1^-)^c \big] \sqcup \overbrace{\big[ \T \cap ( \H_1^+ \cap \H_1^-) \big] \sqcup \big[  \T \cap ( \H_1^+ )^c \big]}^{= \T} \sqcup \{ \x^*_1 \}.
\end{equation*}
For the characterization of $( \Mdo )^{\circ}$, we can also use the same technique as in the proof of Theorem~\ref{thm: up BD} part~\ref{thm: up BD part2}.

For the property related to the set $\Mdo$, recall the definition of $\C_1$ from \eqref{def: set C_i}. Using equation \eqref{eqn: M down RHS} and $\{ \z \in \R^n : z_1 = \lambda_1, \; \| \tilde{\mathbf{z}} \| \leq \nu_1 \} = \overline{\B}_1 \cap (\H_1^+ \cap \H_1^-)$, we have 
\begin{equation}
    (\Mdo \cup \C_1) \cap \H_1^+ 
    = (\Mdo \cap \H_1^+) \cup \C_1 
    = \big[ \overline{\B}_1 \cap (\H_1^-)^c \big] \cup \big[ \overline{\B}_1 \cap (\H_1^+ \cap \H_1^-) \big]
    = \overline{\B}_1 \cap \H_1^+,
    \label{eqn: down2 set prop 1}
\end{equation}
which is closed. Next, from \eqref{eqn: bd + half plane}, we can write 
\begin{equation*}
    \Mdo \subseteq \{ \z \in \R^n : \varphi(\z) \geq 0 \} \setminus \{ \x_1^* \}
    \subseteq \big( \partial (\B_1 \cap \B_2) \big)^c \cup (\H_1^+)^c.
\end{equation*}
Using the fact that $\Mdo \subseteq \overline{\B}_1 \cap \overline{\B}_2$, the above inclusion implies that $\Mdo \cap (\H_1^+)^c \subseteq (\B_1 \cap \B_2) \cap (\H_1^+)^c$. Suppose $\x \in \Mdo \cap (\H_1^+)^c$. However, from the definition of $\Mdo$ in \eqref{def: set M down}, we can write 
\begin{equation}
    \Mdo \cap (\H_1^+)^c 
    = \{ \z \in \R^n: \varphi (\z) > 0 \} \cap (\H_1^+)^c 
    = \{ \z \in \B_1 \cap \B_2: \varphi (\z) > 0 \} \cap (\H_1^+)^c.
    \label{eqn: down2 set prop 2}
\end{equation}
Since $\varphi$ is continuous, and $\B_1 \cap \B_2$ and $(\H_1^+)^c$ are open, there exists $\epsilon \in \R_{>0}$ such that for all $\x_0 \in \B(\x, \epsilon)$ such that $\x_0 \in \{ \z \in \R^n: \varphi (\z) > 0 \} \cap (\B_1 \cap \B_2) \cap (\H_1^+)^c = \Mdo \cap (\H_1^+)^c$. Thus, we conclude that $\Mdo \cap (\H_1^+)^c$ is open.

\textbf{Part \ref{thm: down BD part3}:} $r \in \Big( \frac{L}{2 \sigma_2}, \; \frac{L}{2} \big( \frac{1}{\sigma_1} + \frac{1}{\sigma_2} \big) \Big)$. In this case, we can use similar argument as in the proof of part~\ref{thm: down BD part2} to show the following statements.
\begin{equation*}
    \begin{cases}
    \partial \Mdo \cap (\H_1^-)^c = \partial \B_1 \cap (\H_1^-)^c, 
    \quad \text{(similar to proving \eqref{eqn: sc thm part2.1})} \\
    \partial \Mdo \cap (\H_2^+)^c = \partial \B_2 \cap (\H_2^+)^c, 
    \quad \text{(similar to proving \eqref{eqn: sc thm part2.1})} \\
    \partial \Mdo \cap ( \H_1^+ \cap \H_1^-) = \T \cap ( \H_1^+ \cap \H_1^-), 
    \quad \text{(similar to proving \eqref{eqn: sc thm part2.2})} \\
    \partial \Mdo \cap ( \H_2^+ \cap \H_2^-) = \T \cap ( \H_2^+ \cap \H_2^-), 
    \quad \text{(similar to proving \eqref{eqn: sc thm part2.2})} \\
    \partial \Mdo \cap ( \H_1^+ \cup \H_2^- )^c = \T \cap ( \H_1^+ \cup \H_2^- )^c.
    \quad \text{(similar to proving \eqref{eqn: sc thm part2.3})} \\
    \end{cases}
\end{equation*}
Combining these equations, we obtain the characterization of $\partial \Mdo$. For the characterization of $( \Mup )^{\circ}$, we can use the same technique as shown in the proof of Theorem~\ref{thm: up BD} part~\ref{thm: up BD part2}.
For the property related to the set $\Mdo$, recall the definition of $\C_i$ from \eqref{def: set C_i}. Using a similar approach to part~\ref{thm: down BD part2}, we can show that 
\begin{equation*}
    \begin{cases}
    (\Mdo \cup \C_1) \cap \H_1^+ = \overline{\B}_1 \cap \H_1^+
    \;\; \text{which is closed}, \quad \text{(similar to proving \eqref{eqn: down2 set prop 1})} \\
    (\Mdo \cup \C_2) \cap \H_2^- = \overline{\B}_2 \cap \H_2^-
    \;\; \text{which is closed}, \quad \text{(similar to proving \eqref{eqn: down2 set prop 1})} \\
    \begin{aligned}
    \Mdo \cap ( \H_1^+ \cup \H_2^- )^c = \{ \z \in \B_1 \cap \B_2: \varphi (\z) > 0 \} \cap ( \H_1^+ \cup \H_2^- )^c \\
    \text{which is open}. \quad \text{(similar to proving \eqref{eqn: down2 set prop 2})} \\
    \end{aligned}
    \end{cases}
\end{equation*}

\textbf{Part~\ref{thm: down BD part4}:} $r = \frac{L}{2} \big( \frac{1}{\sigma_1} + \frac{1}{\sigma_2} \big)$. In this case, we obtain that $\overline{\B}_1 \cap  \overline{\B}_2  = \Big\{ \Big( \frac{L}{2} \big(\frac{1}{\sigma_1} - \frac{1}{\sigma_2} \big), \; \mathbf{0} \Big) \Big\}$. 
Suppose $\x = \Big( \frac{L}{2} \big(\frac{1}{\sigma_1} - \frac{1}{\sigma_2} \big), \; \mathbf{0} \Big)$. Since $\Mdo \subseteq \overline{\B}_1 \cap  \overline{\B}_2$, we only need to check the point $\x$. However, the point $\x \in \mathcal{X}$ where $\mathcal{X}$ is defined in \eqref{def: set X}, and we obtain the result due to the definition of $\Mdo$.

\textbf{Part~\ref{thm: down BD part5}:} $r \in \Big( \frac{L}{2} \big( \frac{1}{\sigma_1} + \frac{1}{\sigma_2} \big), \; \infty \Big)$. Since $r > \frac{L}{2} \big(\frac{1}{\sigma_1} + \frac{1}{\sigma_2} \big)$, we have $\overline{\B}_1 \cap  \overline{\B}_2 = \emptyset$. Since $\Mdo \subseteq \overline{\B}_1 \cap  \overline{\B}_2$, we conclude that $\Mdo = \emptyset$.
\end{proof}

%% file: contents/supp-discussion.tex
\section{Proof of Theorem~\ref{thm: correspondance}}  \label{sec: proof-correspondance}

\begin{lemma}  \label{lem: measure-T}
Consider the set $\T \subset \R^n$ defined in \eqref{def: set T}. It holds that the set $\T$ has measure zero.
\end{lemma}

\begin{proof}
First, we aim to demonstrate that if $\x = ( x_1, \Tilde{\mathbf{x}} ) \in \T$ and $\| \x \|^2$ is fixed to be a constant, then $x_1$ has at most three real solutions. Formally, for any $c \in \R$, it follows that $\big| \{ x_1 \in \R : \x \in \T, \; \| \x \|^2 = c \} \big| \leq 3$.

Suppose $\x \in \T$ with $\| \x \|^2 = c$ for some constant $c \in \R$. Recall that $\gamma_i = \frac{L^2}{\sigma_i^2}$ as defined in \eqref{def: gamma beta}. From the algebraic equation defining $\T$ in \eqref{def: set T}, we can express
\begin{equation}
    \frac{\| \x \|^2 - r^2}{d_1^2 (\x) \cdot d_2^2 (\x) } + \frac{1}{\sqrt{\gamma_1 \gamma_2}}
    = \sqrt{\frac{1}{d_1^2 (\x)} - \frac{1}{\gamma_1}} \cdot \sqrt{\frac{1}{d_2^2 (\x)} - \frac{1}{\gamma_2}}.
    \label{eqn: T-non-square}
\end{equation}
The above equation is, in fact, equivalent to
\begin{equation}
    \left( \frac{\| \x \|^2 - r^2}{d_1^2 (\x) \cdot d_2^2 (\x)} + \frac{1}{\sqrt{\gamma_1 \gamma_2}} \right)^2
    = \left( \frac{1}{d_1^2 (\x)} - \frac{1}{\gamma_1} \right) 
    \left( \frac{1}{d_2^2 (\x)} - \frac{1}{\gamma_2} \right),
    \label{eqn: T-square}
\end{equation}
where $\x \in \R^n$ satisfies $\frac{\| \x \|^2 - r^2}{d_1^2 (\x) \cdot d_2^2 (\x)} + \frac{1}{\sqrt{\gamma_1 \gamma_2}} \geq 0$. Simplifying the expression in \eqref{eqn: T-square}, we obtain that
\begin{equation}
    \frac{(\| \x \|^2 - r^2)^2}{d_1^2 (\x) \cdot d_2^2 (\x)} 
    + \frac{2}{\sqrt{\gamma_1 \gamma_2}} \cdot \left( \| \x \|^2 - r^2 \right) 
    = 1 - \frac{1}{\gamma_1} d_1^2(\x) - \frac{1}{\gamma_2} d_2^2 (\x).
    \label{eqn: T-square 2}
\end{equation}
From the definition of $d_1 (\x)$ and $d_2 (\x)$ in \eqref{def: distance}, we have the following equations:
\begin{align}
\begin{cases}
    d_1^2 (\x) \cdot d_2^2 (\x) = (\| \x \|^2 + r^2)^2 - 4 r^2 x_1^2, \\
    \frac{1}{\gamma_1} d_1^2(\x) + \frac{1}{\gamma_2} d_2^2 (\x) 
    = \left( \frac{1}{\gamma_1} + \frac{1}{\gamma_2} \right) (\| \x \|^2 + r^2) 
    + \left( \frac{1}{\gamma_1} - \frac{1}{\gamma_2} \right) 2 r x_1.
\end{cases} \label{eqn: d1d2-derivatives} 
\end{align}
Substituting \eqref{eqn: d1d2-derivatives} into \eqref{eqn: T-square 2}, we obtain that
\begin{equation*}
    1 - \frac{(\| \x \|^2 - r^2)^2}{(\| \x \|^2 + r^2)^2 - 4 r^2 x_1^2} 
    = \frac{2}{\sqrt{\gamma_1 \gamma_2}} (\| \x \|^2 - r^2) + \left( \frac{1}{\gamma_1} + \frac{1}{\gamma_2} \right) (\| \x \|^2 + r^2) + \left( \frac{1}{\gamma_1} - \frac{1}{\gamma_2} \right) 2 r x_1.
\end{equation*}
Rearranging the above equation and using $\| \x \|^2 = c$ yields
\begin{equation}
    \frac{4 r^2 x_1^2 - 4 r^2 c}{4 r^2 x_1^2 - (c + r^2)^2} 
    = \left( \frac{1}{\gamma_1} - \frac{1}{\gamma_2} \right) 2 r x_1
    + \left[ \frac{2}{\sqrt{\gamma_1 \gamma_2}} (c - r^2) + \left( \frac{1}{\gamma_1} + \frac{1}{\gamma_2} \right) (c + r^2) \right].
    \label{eqn: T-square 3}
\end{equation}
For simplicity, we let $c_1 = 4 r^2 c$, $c_2 = (c + r^2)^2$, $c_3 = \frac{1}{\gamma_1} - \frac{1}{\gamma_2}$ and $c_4 = \frac{2}{\sqrt{\gamma_1 \gamma_2}} (c - r^2) + \left( \frac{1}{\gamma_1} + \frac{1}{\gamma_2} \right) (c + r^2)$. Additionally, we define $\hat{x}_1 = 2r x_1$. With these definitions, we can rewrite \eqref{eqn: T-square 3} as follows:
\begin{equation*}
    \frac{\hat{x}_1^2 - c_1}{\hat{x}_1^2 - c_2} = c_3 \hat{x}_1 + c_4.
\end{equation*}
Since the above equation involves solving a polynomial of degree three, according to the fundamental theorem of algebra \cite{fine1997fundamental}, there are at most three real solutions for $\hat{x}_1 \in \R$ (and thus $x_1$) that satisfy the equation. This implies that there are at most three $x_1$ values that satisfy \eqref{eqn: T-non-square} when $\| \x \|$ is held constant, proving the claim.

Let $\mu$ represent the Lebesgue measure on $\R^n$. We aim to demonstrate that $\mu (\T) = 0$. Since $\T \subseteq \overline{\B}_1 \cap \overline{\B}_2$, where $\B_1$ and $\B_2$ are defined in \eqref{def: set B}, it follows that $\T \subseteq \overline{\B} (\mathbf{0}, R)$, where $R = L \left( \frac{1}{\sigma_1} + \frac{1}{\sigma_2} \right)$.

For a given set $\mathcal{A} \subseteq \R^n$, we define the indicator function $\mathbbm{1}_{\mathcal{A}} (\x) \triangleq \begin{cases} 1 \;\; &\text{if} \; \x \in \mathcal{A}, \\ 0 &\text{otherwise} \end{cases}$. The Lebesgue measure of the set $\T$ can be expressed as
\begin{equation}
    \mu (\T) 
    = \int_{\overline{\B}_1 \cap \overline{\B}_2} \mathbbm{1}_{\T} (\x) \diff \x
    = \int_{\overline{\B} (\mathbf{0}, R)}  \mathbbm{1}_{\T} (\x) \diff \x.
    \label{eqn: measure-T}
\end{equation}
Let $\x = (x_1, x_2, \ldots, x_n) \in \R^n$ and $\v = (v_1, v_2, \ldots, v_n) \in \R^n$. We define the coordinate transformation as follows: $v_i = x_i$ for $i \in \{ 1, 2, \ldots, n-1 \}$ and $v_n = \| \x \|^2$. Computing the partial derivatives, we have:
\begin{align*}
    &\frac{\partial x_i}{\partial v_i} = 1 \quad \text{for} \; i \neq n,
    \qquad 
    \frac{\partial x_i}{\partial v_j} = 0 \quad \text{for} \; i \neq j \; \text{and} \; i \neq n, \\
    &\frac{\partial x_n}{\partial v_j} = - \frac{v_j}{x_n} \quad \text{for} \; j \neq n,
    \qquad 
    \frac{\partial x_n}{\partial v_n} = \frac{1}{2} \cdot \frac{1}{x_n}.
\end{align*}
From the computed partial derivatives, the Jacobian of the transformation $\x (\v)$ can be expressed as follows:
\begin{equation*}
    \frac{\partial \x}{\partial \v} 
    \triangleq \det \left( \left[ \frac{\partial x_i}{\partial v_j} \right] \right)
    = \frac{1}{2} \cdot \frac{1}{x_n}.
\end{equation*}
Utilizing this Jacobian, we can reformulate $\mu(\T)$ from \eqref{eqn: measure-T} as
\begin{equation*}
    \int_0^{R^2} \int_{- \sqrt{v_n}}^{\sqrt{v_n}} \int_{- \sqrt{v_n - v_{n-1}^2}}^{\sqrt{v_n - v_{n-1}^2}}  \cdots \int_{- \sqrt{v_n - v_{n-1}^2 - \ldots - v_2^2}}^{\sqrt{v_n - v_{n-1}^2 - \ldots - v_2^2}}
    \left( \frac{1}{2} \cdot \frac{1}{| x_n |} \mathbbm{1}_{\T} (\v) \right)
    \diff v_1 \cdots \diff v_{n-2}  \diff v_{n-1} \diff v_{n}.
\end{equation*}
Given the first claim, the indicator function $\mathbbm{1}_{\T} (\v)$ is zero almost everywhere for a fixed $(v_2, v_3, \ldots, v_{n-1}, v_n) = (x_2, x_3, \ldots, x_{n-1}, \| \x \|^2 )$. Consequently, the innermost integral (i.e., integral with respect to $v_1$) evaluates to zero, leading to the conclusion that $\mu (\T) = 0$.
\end{proof}

\begin{proof}[Proof of Theorem~\ref{thm: correspondance}]
Let 
\begin{equation}
\begin{split}
    \mathcal{A}_1 
    &= \left\{ \x \in \R^n : \tilphi_1 (\x) + \tilphi_2 (\x) > \psi(\x) \right\} \cap (\B_1 \cap \B_2), \\
    \mathcal{A}_2 
    &= \left\{ \x \in \R^n : \tilphi_1 (\x) + \tilphi_2 (\x) = \psi(\x) \right\} \cup \partial (\B_1 \cap \B_2)
    = \T \cup \partial (\B_1 \cap \B_2).
\end{split}  \label{eqn: set-A1-A2}
\end{equation}
Since $\mathcal{X} \subseteq \B_1 \cap \B_2$ (from the definition of $\mathcal{X}$ in \eqref{def: set X}) and applying Proposition~\ref{prop: M inclusion}, we can establish the relationship between $\mathcal{A}_1$, $\mathcal{A}_2$, and $\M$ as:
\begin{equation}
    \mathcal{A}_1 \subseteq \Mdo \subseteq \M \subseteq \Mup 
    \subseteq \left\{ \x \in \R^n : \tilphi_1 (\x) + \tilphi_2 (\x) \geq \psi(\x) \right\} \cap (\overline{\B_1 \cap \B_2})
    = \mathcal{A}_1 \cup \mathcal{A}_2.
    \label{eqn: chain-A1-A2-M}
\end{equation}

First, suppose $\x \in \mathcal{A}_1$. We aim to show that $| \F_{\S} (\x) | = \mathfrak{c}$. Following a similar approach as in the proof of Proposition~\ref{prop: angle sc}, we can establish the existence of a vector $\g \in \R^n$ such that $\angle (\g, \u_1 (\x)) \in \Phi_1 (\x)$, $\angle (- \g, \u_2 (\x)) \in \Phi_2 (\x)$, $\langle \g, \boldsymbol{e}_i \rangle = 0$ for $i \in \{ 3, 4, \ldots, n \}$ and $\| \g \| = L$, where $\Phi_i (\x)$ for $i \in \{ 1, 2 \}$ are defined in \eqref{def: Phi}. However, since $\x \in \B_1 \cap \B_2$, we have $\tilphi_1 (\x) \neq 0$ and $\tilphi_2 (\x) \neq 0$. This implies that the sets of angles $\Phi_1 (\x)$ and $\Phi_2 (\x)$ do not degenerate to $\{ 0 \}$.

Let $\boldsymbol{R}_n: ( - \pi, \pi ] \to \R^{n \times n}$ be a rotation matrix such that $\boldsymbol{R}_n (\theta) = \begin{bmatrix} \boldsymbol{R} (\theta) & \mathbf{0} \\ \mathbf{0} & \boldsymbol{I} \end{bmatrix}$, where $\boldsymbol{R} (\theta) = \begin{bmatrix} \cos \theta & - \sin \theta \\ \sin \theta & \cos \theta \end{bmatrix} \in \R^{2 \times 2}$ is the rotation matrix in two-dimensional space.
It follows that there exists $\delta' \in \R_{>0}$ such that for all $| \delta | \leq \delta'$, we have $\angle ( \Tilde{\g} (\delta), \u_1 (\x) ) \in \Phi_1 (\x)$ and $\angle ( - \Tilde{\g} (\delta), \u_2 (\x) ) \in \Phi_2 (\x)$, where $\Tilde{\g} (\delta) = \boldsymbol{R}_n (\delta) \g$. 
It follows that for any $\delta \in [ - \delta', \delta' ]$, according to Corollary~\ref{cor: function exist union}, there exist functions $f_i ( \; \cdot \; ; \delta) \in \bigcup_{\hat{\sigma} \geq \sigma_i} \Q (\x_i^*, \hat{\sigma})$ for $i \in \{ 1, 2 \}$ such that $\Tilde{\g} (\delta) = \nabla f_1 ( \x; \delta) = - \nabla f_2 (\x; \delta)$ and $\| \nabla f_1 (\x ; \delta) \| = \| \nabla f_2 (\x ; \delta) \| = L$.

Let $\F_{\Q} (\x)$ be the set defined as 
\begin{equation*}
    \F_{\Q} (\x) = \bigg\{ (f_1, f_2) \in \bigcup_{\hat{\sigma} \geq \sigma_1} \Q (\x_1^*, \hat{\sigma}) \times \bigcup_{\hat{\sigma} \geq \sigma_2} \Q (\x_2^*, \hat{\sigma}): 
    \nabla f_1 (\x) = - \nabla f_2 (\x), \; \| \nabla f_1 (\x) \| = \| \nabla f_2 (\x) \| \leq L \bigg\}.
\end{equation*}
Thus, each angle $\delta \in [ - \delta', \delta' ]$ corresponds to at least one pair $( f_1, f_2 ) \in \F_{\Q} (\x) \subset \F_{\S} (\x)$. Furthermore, if there exists another angle $\Tilde{\delta} \in [ - \delta', \delta' ]$ with $\Tilde{\delta} \neq \delta$, it corresponds to a different pair $( \Tilde{f}_1, \Tilde{f}_2 ) \in \F_{\Q} (\x) \subset \F_{\S} (\x)$ since no quadratic function can have two different gradients at a given point. Given that the set $[-\delta', \delta']$ has the cardinality of the continuum for $\delta' \in \R_{>0}$, we can conclude that $| \F_{\S} (\x) | = \mathfrak{c}$, as asserted.

Next, let's consider the set $\mathcal{A}_2$ defined in \eqref{eqn: set-A1-A2}. Given Lemma~\ref{lem: measure-T}, which establishes that $\mu (\T) = 0$, and the fact that $\partial (\B_1 \cap \B_2) \subseteq \partial \B_1 \cup \partial \B_2$ along with $\mu (\partial \B_1 ) = \mu (\partial \B_2 ) = 0$, we can utilize the monotonicity of $\mu$ to deduce:
\begin{equation*}
    \mu (\mathcal{A}_2) = \mu \left( \T \cup \partial (\B_1 \cap \B_2) \right)
    \leq \mu (\T) + \mu (\partial \B_1) + \mu (\partial \B_2) = 0.
\end{equation*}
Hence, we can conclude that $\mu (\mathcal{A}_2) = 0$.

Combining this with the fact that $\M \subseteq \mathcal{A}_1 \cup \mathcal{A}_2$ from \eqref{eqn: chain-A1-A2-M}, along with the earlier established result $| \F_{\S} (\x) | = \mathfrak{c}$ for $\x \in \mathcal{A}_1$, we derive that the equation $| \F_{\S} (\x) | = \mathfrak{c}$ holds almost everywhere on $\M \subset \R^n$.
\end{proof}